\documentclass[11pt]{article}
 \usepackage{benstyle}

\usepackage[ruled,linesnumbered]{algorithm2e}

\newcommand{\benbox}[1]{\begin{center}\fbox{\begin{minipage}{0.95\textwidth}#1\end{minipage}}\end{center}}

\usepackage{enumerate}

\usepackage{cite}

\graphicspath{{figures/}}

\title{On efficient algorithms for computing near-best polynomial approximations to high-dimensional, Hilbert-valued functions from limited samples}

\author{Ben Adcock\thanks{Department of Mathematics, Simon Fraser University, Burnaby BC, Canada} \and Simone Brugiapaglia\thanks{Department of Mathematics and Statistics, Concordia University, Canada} \and Nick Dexter\footnotemark[1] \and Sebastian Moraga\footnotemark[1]}

\begin{document}

\maketitle

\begin{abstract}
Sparse polynomial approximation has become an indispensable technique for learning accurate approximations to smooth, high- or infinite-dimensional functions from limited sample values. This is a key task in computational science and engineering -- for example, surrogate model construction in Uncertainty Quantification (UQ), where the underlying function is the solution map of a parametric (or stochastic) Differential Equation (DE). Yet, sparse polynomial approximation lacks a complete theory. On the one hand, there is a well-developed theory of \textit{best $s$-term polynomial approximation}, which asserts exponential or algebraic rates of convergence for holomorphic functions. On the other hand, there are increasingly mature methods such as (weighted) $\ell^1$-minimization for such approximations. While the sample complexity of these methods has been analyzed through compressed sensing theory, the matter of whether they achieve the rates of the best $s$-term approximation is not fully understood. Furthermore, these methods are not algorithms per se, since they involve exact minimizers of nonlinear (albeit convex) optimization problems.

This paper closes these gaps. Specifically, we pose and answer the following question: \textit{are there robust, efficient algorithms for computing sparse polynomial approximations to finite- or infinite-dimensional, holomorphic and Hilbert-valued functions from limited samples that achieve the same rates as the best $s$-term approximation?} We answer this in the affirmative by introducing algorithms and theoretical guarantees that assert exponential or algebraic rates of convergence in terms of the number of samples, along with robustness to \textit{sampling}, \textit{algorithmic} and \textit{physical discretization} errors.  We tackle both scalar- and Hilbert-valued functions, this being particularly relevant in parametric or stochastic DEs. Our results involve several significant developments of existing techniques, including a novel restarted primal-dual iteration for solving weighted $\ell^1$-minimization problems in Hilbert spaces. Our theory is supplemented by numerical experiments demonstrating the practical efficacy of these algorithms.

\end{abstract}

\section{Introduction}

A fundamental task in computational science and engineering involves accurately approximating a smooth target function from limited data. Such a task arises notably in the study of parametric models of physical processes. Here the variables represent the parameters in the system, e.g., material properties, forcing terms, or boundary information, and the parametric model is often represented as a (system of) Differential Equations (DEs) or Partial Differential Equations (PDEs) depending on these parameters. Important objectives involve understanding how the choice of such parameters affect the output(s) of the system and, in the stochastic setting, understanding how uncertainty in the parameter values propagates to its output -- the latter being one of the key tasks in computational Uncertainty Quantification (UQ) \cite{ghanem2017handbook,le2010spectral,smith2013uncertainty,sullivan2015introduction}.

\subsection{High-dimensional function approximation from limited samples}

Abstractly, this task can be recast as that of approximating an unknown target function
\bes{
f : \cU \rightarrow \cV,\ \bm{y} \mapsto f(\bm{y}),
}
from \textit{sample values} (or \textit{snapshots}) 
\be{
\label{f-snapshots}
f(\bm{y}_1),\ldots,f(\bm{y}_m).
}
Here, the input space $\cU$ is typically a subset of $\bbR^d$ (in the finite-dimensional case) or $\bbR^{\bbN}$ (in the infinite-dimensional case). The output space $\cV$ could either be a scalar field, a finite-dimensional vector space or an infinite-dimensional Banach or Hilbert space.

This problem is challenging in a number of ways. First, the dimension $d$ is high, since modern parametric models typically involve many parameters. It may also be infinite, e.g., in the case of a random field represented via its Karhunen--Lo\`eve expansion. Therefore, care must be taken to design methods that scale well with dimension. In addition, the amount of samples $m$ is often highly limited. For example, in the parametric DE setting, each evaluation of $f$ involves an expensive computational simulation. The data \R{f-snapshots} is also always corrupted by errors, due to noise in physical experiments or numerical error in solving a DE. And finally, since the output $f(\bm{y})$ is often the solution of DE parametrized by the vector $\bm{y}$, it may consequently take values in an infinite-dimensional Banach or Hilbert space. While it is commonplace to circumvent this issue in practice by considering scalar-valued \textit{quantities of interest} (i.e., functions of the form $g(\bm{y}) = \cQ(f(\bm{y}))$ for some known map $\cQ : \cV \rightarrow \bbC$), approximating the full function $f$ is both of theoretical interest and practical importance \cite{dexter2019mixed}.

\rem{
As a further consideration, we note that in many scenarios one may have substantial flexibility to choose the sample points $\bm{y}_1,\ldots,\bm{y}_m \in \cU$ in \R{f-snapshots}. However, in other scenarios they may be fixed, e.g., when dealing with legacy data. In this work, we consider Monte Carlo sampling -- which may be considered either as a designed sampling strategy or a fixed one, depending on the setting. Here, the samples are drawn randomly and independently from an underlying probability measure on $\cU$.  This is very common in practice, in particular in UQ settings. 
}

\subsection{Smoothness and best $s$-term polynomial approximation}

A key characteristic of parametric model problems is that the target function $f$ is often smooth. There is now a large body of literature that has established that solution maps of a wide range of different parametric DEs are \textit{holomorphic} (i.e., \textit{analytic}) functions of their parameters. We mention in passing problems such as: elliptic PDEs with affine and nonaffine parametric dependence, parabolic PDEs, PDEs over parametrized domains and shape uncertainty, parametric Initial Value Problems (IVPs), parametric hyperbolic problems and parametric control problems. Classical results in this direction can be found in, e.g., \cite{walter1998ordinary} and references therein. For surveys of more recent results, we refer to \cite{cohen2015approximation} and \cite{adcock2021sparse}, and references therein.

In tandem with the effort to establish holomorphic regularity of parametric DEs, there has also been a focus on applying polynomial methods, and in particular, \textit{best $s$-term polynomial approximation} to construct finite approximations to such functions. In best $s$-term approximation, the function $f$ is approximated by an $s$-term expansion corresponding to its largest $s$ coefficients (measured in the $\cV$-norm) with respect to a polynomial basis. Common choices include Taylor polynomials, tensor-product Legendre and Chebyshev polynomials on bounded hypercubes or tensor-product Hermite and Laguerre polynomials on $\bbR^d$ or $[0,\infty)^d$. Over the last fifteen years, there have been significant developments in the approximation theory of such techniques (see the aforementioned references, plus those in \S \ref{s:related-work}). Signature results have established \textit{exponential} and \textit{algebraic} convergence rates for the best $s$-term approximation. The former assert that the error decays at least exponentially fast in $s^{1/d}$ in finite dimensions for any holomorphic function. The latter assert that the error decays algebraically fast; specifically, like $s^{1/2-1/p}$ for some $0 < p < 1$. These algebraic rates also hold in infinite dimensions, thus establishing best $s$-term approximation as a (theoretical) means to approximate holomorphic functions of infinitely many variables. We review several such results in \S \ref{ss:best-s-term}.

\subsection{Computing sparse polynomial approximations}

Unfortunately, the best $s$-term approximation cannot generally be computed from the samples \R{f-snapshots}. Indeed, constructing it theoretically involves computing and then searching over infinitely-many coefficients. Both tasks are generally impossible. Therefore, there has also been a focus on methods to compute accurate polynomial approximations from sample values. 

One line of work focuses on \textit{least-squares} methods, wherein a polynomial approximation (or sequence of approximations) is computed in a fixed polynomial subspace (or sequence of nested subspaces). See \S \ref{s:related-work} for relevant references. Such methods are essentially optimal if a (sequence of) polynomial subspace that gives a quasi-best $s$-term approximation is known. 

However, this information is generally unavailable in practice (although it may be for certain simple parametric DEs). It essentially equates to knowing the region of holomorphy of the underlying function, which is itself similar to knowing the order of importance of the parametric variables, and their relative strengths. 
To counter this, there are \textit{adaptive} least-squares methods \cite{chkifa2013sparse,chkifa2014high,cohen2018multivariate,gittelson2013adaptive,cohen2015approximation,migliorati2015adaptive,cohen2018multivariate,migliorati2019adaptive}. Here one strives to construct such subspaces adaptively using the given data \R{f-snapshots}, typically via a greedy procedure. However, these currently lack theoretical guarantees \cite{chkifa2014high,cohen2015approximation}. 

To overcome this limitation, there has also been a substantial focus on methods inspired by \textit{compressed sensing} \cite{vidyasagar2019introduction,foucart2013mathematical,adcock2021compressive}. See \S \ref{s:related-work} for references. These methods seek a polynomial approximation in a larger subspace, whose coefficients are defined as a minimizer of an $\ell^1$- or weighted $\ell^1$-minimization problem. A key component of this endeavour has been to determine the \textit{sample complexity} of such schemes, i.e., quantifying how many (Monte Carlo) samples $m$ are sufficient to obtain an approximation with a certain guaranteed error bound, involving a (weighted) best approximation error plus a truncation error. 
Yet, precise rates of approximation (i.e., algebraic or exponential in $m$) have typically not been derived in previous work for these schemes. Another key limitation of previous work is that such methods are not algorithms per se. Indeed, they consider exact minimizers of nonlinear optimization problems, which cannot be computed exactly in finitely-many arithmetic operations.

\subsection{Problem and main contributions}\label{s:main-contribs}

Least-squares and compressed sensing techniques are commonly applied to compute polynomial approximations to parametric and stochastic DEs. However, there is a key gap between theory and practice. The theory of the best $s$-term approximation asserts the existence of polynomial approximations that attain specific algebraic or exponential rates of convergence for arbitrary holomorphic functions. Yet, it is currently unknown whether similar rates in terms of the number of samples $m$ can be obtained via an algorithm that computes a polynomial approximation from the samples \R{f-snapshots} in finite time. The purpose of this work is to close this gap.

We now describe the problem considered in this paper. Let $\cU = [-1,1]^{d}$, where $d \in \bbN$ or $d = \infty$, and $\cV$ be an arbitrary separable Hilbert space. Let $\varrho$ be either the uniform or Chebyshev (arcsine) measure and consider the associated tensor-product Legendre or Chebyshev polynomials. Now let  $f : \cU \rightarrow \cV$ be the unknown target function that we seek to approximate, draw $m$ sample points $\bm{y}_1,\ldots,\bm{y}_m$ i.i.d.\ from $\varrho$ and let
\be{
\label{f-snapshots-noisy}
d_i = f(\bm{y}_i) + n_i,\quad i = 1,\ldots,m,
}
be $m$ noisy samples of $f$. Then, informally stated, the problem we study in this paper is the following: \textit{devise algorithms that take \R{f-snapshots-noisy} as input and compute the coefficients of a polynomial approximation $\hat{f}$ to $f$ with guarantees on both the computational complexity and the error $f - \hat{f}$}. Note that the formal problem statement involves several technicalities (in particular, the definition of an algorithm), so we defer it to \S \ref{ss:prob-stat}. 

Our main contributions are on the existence of such algorithms (see Tables \ref{tab:the-algorithms} and \ref{tab:the-effic-algorithms} and Algorithms \ref{a:primal-dual-wSRLASSO-coeff} and \ref{a:primal-dual-wSRLASSO-coeff-restart}). In all cases, we establish error bounds of the form
\be{
\label{main-err-intro}
\nmu{f - \hat{f} }_{L^2_{\varrho}(\cU ; \cV)} \lesssim E_{\textsf{app}} + E_{\textsf{samp}} + E_{\textsf{disc}} + E_{\textsf{alg}},
}
with probability at least $1-\epsilon$ with respect to the (Monte Carlo) draw of the sample points $\bm{y}_i$. 
Here $\nm{\cdot}_{L^2_{\varrho}(\cU;\cV)}$ is the Lebesgue--Bochner norm. The bound \R{main-err-intro} provides a complete accounting for the main sources of error in the problem:
\bulls{
\item $E_{\textsf{app}}$ is a \textit{polynomial approximation error} term. Depending on the specific setup, it decays algebraically (Theorems \ref{t:main-res-map-alg}--\ref{t:main-res-effic-algo-alg_infty}) or exponentially (Theorems \ref{t:main-res-map-alg_exp}
--\ref{t:main-res-effic-algo-alg_exp}) with respect to $m$ (up to several log terms). For instance, in the infinite-dimensional setting (Theorems \ref{t:main-res-map-alg_infty}--\ref{t:main-res-effic-algo-alg_infty}), this term is given by
\be{
\label{alg-err-intro}
E_{\textsf{app}} = C \cdot \left ( \frac{m}{c_0 L} \right )^{1/2-1/p}, \qquad L = \log(m) \cdot (\log^3(m) + \log(\epsilon^{-1}) ),
}
where $c_0 \geq 1$ is a universal constant, $C$ is a constant depending on (the region of holomorphy of) $f$ only, $p \in (0,1)$ is a parameter determined by the region of holomorphy of $f$ and $0 < \epsilon < 1$ is the failure probability of \eqref{main-err-intro}. It is completely equivalent to the corresponding algebraic decay rate (Theorem \ref{t:best_s_term_alg_inf}) for the best $s$-term approximation error, except with $s$ replaced by $m/(c_0 L)$.
\item $E_{\textsf{samp}}$ is the \textit{sampling error} and is equal to
\bes{
E_{\textsf{samp}} = \sqrt{\frac1m \sum^{m}_{i=1} \nm{n_i}^2_{\cV}},
}
 i.e., the norm of the error in the samples \R{f-snapshots-noisy}. In other words, this means that the algorithms are robust to noise in the samples.
\item $E_{\textsf{disc}}$ is the \textit{physical discretization error}. This term accounts for the fact that an algorithm cannot work with (i.e., take as input, or perform computations in) $\cV$ when it is an infinite-dimensional Hilbert space. The algorithms (see Tables \ref{tab:the-algorithms} and \ref{tab:the-effic-algorithms} and Algorithms \ref{a:primal-dual-wSRLASSO-coeff} and \ref{a:primal-dual-wSRLASSO-coeff-restart}) therefore work in a finite-dimensional discretization space $\cV_h \subseteq \cV$. This is a standard step in parametric DEs, where discretization is often performed via techniques such as the Finite Element Method (FEM). In this case, $\cV_h$ is a finite element space. The term  $E_{\textsf{disc}}$ quantifies the effect of this error. It is given by
\bes{
E_{\textsf{disc}} = \nmu{f - \cP_h(f)}_{L^{\infty}(\cU ; \cV)}, 
}
where $\cP_h : \cV \rightarrow \cV_h$ is the orthogonal projection onto $\cV$. In other words, the effect of working in $\cV_h$ instead of $\cV$ is determined by the error of the (pointwise) best approximation $\cP_h(f)$ to $f$ from $\cV_h$. If $\cV$ has finite dimension we assume $\cV_h =\cV$, which implies that $E_{\textsf{disc}} = 0$. 

\item $E_{\textsf{alg}}$ is the \textit{algorithmic error}. It depends on the number of iterations $t$ performed by the algorithm that computes the coefficients of the polynomial approximation $\hat{f}$. We construct one type of algorithm (see Table \ref{tab:the-algorithms} and and Algorithm \ref{a:primal-dual-wSRLASSO-coeff}) where this term is $\ord{1/t}$ as $t \rightarrow \infty$. This decay is relatively slow, especially in the regime where $E_{\textsf{app}}$ is exponentially small in $m$. However, we also present an \textit{efficient} algorithm (Table \ref{tab:the-effic-algorithms} and Algorithm \ref{a:primal-dual-wSRLASSO-coeff-restart}) for which this term decays exponentially-fast in $t$ (specifically, $\ordu{\E^{-t}}$ as $t \rightarrow \infty$), subject to an additional theoretical constraint. This constraint is seemingly an artefact of the proof. Our numerical experiments suggest it is unnecessary in practice. 
}   
We also determine the computational cost of the algorithms in all cases. Here, we draw two main conclusions.
\bulls{
\item In the infinite-dimensional case (Theorems \ref{t:main-res-algo-alg_infty}--\ref{t:main-res-effic-algo-alg_infty}), the computational cost is \textit{subexponential} in $m$. Specifically, after $t$ iterations of the algorithm, it is
\bes{
\ord{t \cdot m^{1 + (\alpha + 1) \log(4 m) / \log(2)}},\quad m \rightarrow \infty,
}
where $\alpha = 1$ (Chebyshev) or $\alpha = \log(3)/\log(4) \approx 0.79$ (Legendre).
\item In the finite-dimensional, exponential setting (Theorems \ref{t:main-res-algo-alg_exp}--\ref{t:main-res-effic-algo-alg_exp}), the computational cost is \textit{algebraic} in $m$ for fixed $d$. Namely,
\bes{
\ord{t \cdot m^{\alpha + 2} (\log(m))^{(d-1)(\alpha+1)}},\quad m \rightarrow \infty.
}
}
Note that these computational cost estimates also depend polynomially on the dimension of the discretization space $\cV_h$.

\subsection{Discussion and further contributions}\label{ss:discussion}
This work bridges a gap between the best $s$-term polynomial approximation theory and algorithms for computing such approximations from sample values. In particular, it asserts that algebraic and exponential rates with respect to the number of samples $m$ that are highly similar to those of the best approximation.  In other words, polynomial approximations of holomorphic functions can be achieved in a \textit{sample efficient} manner. Furthermore, they can be computed in supexponential or algebraic computational cost.

Our main results assume holomorphy of the underlying function in order to attain these rates. However, they assume no \textit{a priori} knowledge of the region of holomorphy. As discussed, if such information is available, then least-squares methods can be used more straightforwardly to compute an approximation.  The holomorphy assumption is made in order to have concrete algebraic and exponential rates. However, our algorithms exist independently of the smoothness assumption. It would be possible to also provide rates for other classes of functions, e.g., those possessing finite orders of (mixed) smoothness. We use holomorphy as our assumption due to its strong connections with the theory of parametric DEs.

Our algorithms and analysis are based on compressed sensing theory and involve computing approximate minimizers of certain weighted $\ell^1$-minimization problems. Here we make several additional contributions:
\begin{enumerate}[(i)]
\item We provide precise error rates for polynomial approximation via compressed sensing. As noted, most prior work on compressed sensing involves quantifying the sample complexity to obtain a certain (weighted) best approximation error. Subject to a holomorphy assumption, we use this to obtain specific algebraic and exponential rates.
\item Prior works consider polynomial approximations formed by exact minimizers of nonlinear optimization problems. We introduce novel, efficient algorithms to compute approximate minimizers in finite computational time (see also below).
\item While these algorithms are motivated by the desire to have full error bounds, they are also completely practical. We present a series of numerical experiments demonstrating their practical efficacy. In fact, our experiments show that these algorithms work even better than our theoretical results suggest.
\item Most prior works on compressed sensing (with the exception of \cite{dexter2019mixed}) focus on scalar-valued functions, e.g., quantities of interest of parametric DEs. We develop algorithms that work in the Hilbert-valued setting, and, crucially, provide error bounds that take into account discretization error.
\end{enumerate}
More precisely, our algorithms first formulate the approximation problem as the recovery of a finite, Hilbert-valued vector (i.e., an element of $\cV^N$) via a so-called weighted, Square-Root LASSO (SR-LASSO) optimization problem. The use of the SR-LASSO, as opposed to the classical LASSO or various constrained formulations, is crucial to this work. It is \textit{noise-blind}. Hence it allows us to devise algorithms that do not require any a priori (and generally unavailable) estimates on the measurement error $n_i$ in \R{f-snapshots-noisy} or the truncation error with respect to the finite polynomial space in which the approximation is constructed.

To develop algorithms, we use two key ideas. First, we use a powerful, general-purpose first-order optimization method for solving nonsmooth, convex optimization problems. Second, we use the technique of \textit{restarts} to drastically accelerate its convergence. For the former, we employ the \textit{primal-dual iteration} (also known as the Chambolle–Pock algorithm) \cite{ChambolleEtAl2016,ChambollePock2011}. We present error bounds for this method for solving the Hilbert-valued, weighted SR-LASSO, which decay like $\ord{1/t}$, where $t$ is the iteration number. Next, we use a novel restarting procedure, recently introduced in \cite{colbrook2021can,colbrook2021warpd}, to obtain faster, exponential decay of the form $\ordu{\E^{-t}}$.

To the best of our knowledge, this is the first time either the primal-dual iteration or a restarting scheme has been applied to the problem of sparse polynomial approximation. Many existing works use blackbox solvers such as SPGL1 \cite{berg2009probing,spgl1:2007}. See \cite{dexter2019mixed} for a forward-backwards splitting technique in combination with Bregman iterations and fixed-point continuation and \cite{tsilifis2019compressive} for an approach based on Douglas--Rachford splitting. Besides its amenability to theoretical analysis, the primal-dual scheme is also particularly attractive because of its insensitivity to parameter choices and the possibility of performing acceleration via restarts.

As noted, polynomial-based methods have become popular tools for the practical approximation high-dimensional, holomorphic functions arising in problems in computational science and engineering. However, they are by no means the only method. Other popular techniques include Gaussian processes (also known as kriging) \cite{smith2013uncertainty,sullivan2015introduction}, radial basis methods \cite{smith2013uncertainty,jung2010recovery}, reduced-order methods \cite{hesthaven2015certified,quarteroni2015reduced} and, recently, methods based on deep neural networks and deep learning \cite{adcock2022nearoptimal,adcock2021deep,adcock2020between,dung2021deep,dung2021computation,
daws2019analysis,opschoor2022exponential,herrman2022constructive,opschoor2022deep,schwab2021deep,li2020better}. Our goal in this work is to develop algorithms for constructing polynomial approximations that achieve the same rates as the theoretical benchmark provided by the best $s$-term polynomial approximation. An important consideration that we do not address in this work is \textit{tractability} and the \textit{information complexity} \cite{novak2010trac,novak2008trac} of these classes of functions and, in particular, whether polynomial-based methods constitute \textit{optimal algorithms}. This question has been studied in the infinite-dimensional case in recent work \cite{adcock2023optimala}. Here, it is shown that the rate $m^{1/2-1/p}$ is a lower bound for the \textit{(adaptive) $m$-width}, i.e., no combination of $m$ (adaptive) linear samples and a (potentially nonlinear) reconstruction map can achieve an approximation error decaying faster than this rate. Notice that this rate is the same, up to constants and logarithmic factors, as \R{alg-err-intro}. Unfortunately, we cannot claim that our algorithms are near optimal  for this problem -- and, moreover, that \textit{standard information}, i.e., pointwise samples, constitutes near-optimal information --  because our theoretical results in the infinite-dimensional case are \textit{nonuniform}. See Remark \ref{rem:uniform-vs-nonuniform} for further discussion on this point, and \S \ref{s:conclusions} for further comments on tractability.

\subsection{Related work}\label{s:related-work}

The systematic study of best $s$-term polynomial approximation of high- or infinite-dimensional holomorphic functions began around 2010 with the works of \cite{bieri2009sparse,cohen2010convergence,todor2007convergence,cohen2011analytic,hansen2013analytic}. For reviews, see \cite{cohen2015approximation} and \cite[Chpt.\ 3]{adcock2021sparse}. 
Note that many of these works assume the function is a solution of a parametric PDE, and therefore first demonstrate that such a function is holomorphic. However, other works avoid this step and use specific properties of the DE to obtain refined estimates. See, e.g., \cite{bachmayr2017sparseI,bachmayr2017sparseII} for results of this type. Other recent works such as \cite{bonito2021polynomial} also study the problem without assuming the function is a solution of a parametric PDE.

The study of least-squares method for constructing such approximations from sample points began in the early 2010s \cite{chkifa2015discrete,cohen2013stability,migliorati2014analysis,migliorati2013polynomial}. There has since been significant research on this topic. Many works have pursued various extensions, such as enhanced sampling strategies \cite{migliorati2015analysis,narayan2017christoffel,guo2018weighted,seshadri2017effectively,tang2014discrete,zhou2015weighted,zhou2014multivariate}, near-optimal sampling strategies \cite{adcock2020nearoptimal,hampton2015coherence,cohen2017optimal}, optimal sampling strategies \cite{cohen2021optimal,kammerer2019worst,limonova2021sampling,temlyakov2021optimal,bartel2022constructive,dolbeault2023sharp}, methods for general domains \cite{migliorati2021multivariate,adcock2020approximating,dolbeault2020optimal}, optimal and adaptive methods \cite{migliorati2015adaptive,cohen2017discrete,migliorati2019adaptive} and multilevel strategies \cite{hajiali2020multilevel}. See \cite{hadigol2018least,guo2020constructing,cohen2018multivariate} and \cite[Chpt.\ 5]{adcock2021sparse} for reviews.

Compressed sensing was introduced in the context of image and signal processing by modelling image and signals as sparse vectors \cite{adcock2021compressive,donoho2006compressed,foucart2013mathematical,candes2006robust}. Its use in polynomial approximation started early in the last decade with the works of \cite{blatman2011adaptive,doostan2011nonadapted,mathelin2012compressed,rauhut2012sparse,yan2012stochastic}. This has also led to substantial research. See \cite{dexter2018sparse,
dexter2019mixed,
doostan2011nonadapted,
mathelin2012compressed,
rauhut2017compressive,yang2013reweighted} and references therein for applications to parametric PDEs. Various extensions include refined sampling strategies
\cite{alemazkoor2018near-optimal,diaz2018sparse,guo2017stochastic,hampton2015compressive,jakeman2017generalized,liu2016stochastic,tang2014subsampled}, iterative methods and basis selection techniques
 \cite{alemazkoor2017divide,
 hampton2018basis,
 tsilifis2019compressive,
  yang2016enhancing,
 yang2018sliced,
 yang2019general,yang2016enhancing},  
nonconvex optimization methods
 \cite{guo2017sparse,
 tran2019class,
 xu2020analysis,
 yan2017sparse},
 sublinear-time algorithms
 \cite{choi2021sparse,
 choi2021sparse2},
gradient-enhaced minimization techniques
 \cite{adcock2019compressive,
 guo2017gradient,
 jakeman2015enhancing,
 peng2016polynomial,
 sui2020weighted,
 tang2013methods},
methods for dealing with corrupted samples
 \cite{adcock2019correcting,
 adcock2018compressed,
 ho2020recovery,
 shin2016correcting} and
multilevel and multifidelity strategies
 \cite{bouchot2017multilevel,
 ng2012multifidelity}.
For additional information and reviews, see
 \cite{luthen2021sparsesolvers,
 narayan2015stochastic,
 hampton2017compressive,
 kougioumtzoglou2020sparse,
 luthen2021sparseliterature} and \cite[Chpt.\ 7]{adcock2021sparse}.

Our work combines and extends several key elements of this literature. First, weighted $\ell^1$-minimization, which was developed in \cite{adcock2020sparse,adcock2017infinite,adcock2018infinite,adcock2019correcting,chkifa2018polynomial,peng2014weighted,rauhut2016interpolation,yang2013reweighted} and \cite[Chpts.\ 6-7]{adcock2021sparse}. Second, the notions of \textit{lower} and \textit{anchored} sets (see \S \ref{s:lower-anchored}). These have been extensively studied in the best $s$-term polynomial approximation literature. Compressed sensing techniques aiming to exploit such structures were first considered in \cite{adcock2018infinite,adcock2019correcting,chkifa2018polynomial} and \cite[Chpt.\ 7]{adcock2021sparse}. Third, the extension of classical compressed sensing theory from vectors in $\bbR^N$ (or $\bbC^N$) to Hilbert-valued vectors in $\cV^N$. This was first developed in \cite{dexter2019mixed}. In order to prove our main results, we also extend this framework to the weighted setting.

See \cite{ChamPockAN,ChambolleEtAl2016,ChambollePock2011} for more on the primal-dual iteration and \cite{renegar2022simple,roulet2020sharpness,roulet2020computational} for the general notion of restarts in continuous optimization. Note that there are also various non-optimization based techniques in the compressed sensing literature (see, e.g., \cite{foucart2013mathematical}), including iterative  threshold and greedy methods (the latter are closely related to the adaptive least-squares methods discussed earlier \cite[\S 6.2.5]{adcock2021sparse}). However, these do not currently possess theoretical guarantees in the weighted setting.

There have been several previous attempts to connect compressed sensing theory for analyzing the sample complexity of polynomial approximations via (weighted) $\ell^1$-minimization and best $s$-term polynomial approximation theory. In \cite{rauhut2017compressive}, the authors consider approximating scalar quantities of interest of solutions to affine parametric operator equations in Banach spaces. Assuming a certain weighted summability criterion, they first show holomorphy of the parametric solution map and then use a weighted $\ell^1$-minimization procedure in combination with Chebyshev polynomials to derive algebraic rates of convergence, similar to \R{alg-err-intro}. Our work is more general, since its starting point is a holomorphic function, not a solution of a parametric operator equation. We also consider Hilbert-valued functions, i.e., the whole solution map, not a scalar quantity of interest of it. Moreover, the work of \cite{rauhut2017compressive} is based on exact minimizers of certain constrained, weighted $\ell^1$-minimization problems, whereas we construct full algorithms. Recently, at the same time as writing this paper, some similar results were presented in the book \cite{adcock2021sparse} written by two of the authors. However, these only consider the scalar-valued case and do not address algorithms, which is the main focus of this work.

\subsection{Outline}

The remainder of this paper proceeds as follows. We commence in \S \ref{s:prelims} with various preliminaries, including key notation and best $s$-term polynomial approximation theory. Next, in \S \ref{s:prob-main-res} we first formally define the problem and then state our main results on the existence of algorithms. In \S \ref{s:constr-algs} we derive these algorithms. Then in \S \ref{s:num-exp} we present numerical experiments demonstrating their practical performance. \S \ref{s:proof-overview}--\ref{s:final-args} are devoted to the proofs of the main results. See \S \ref{s:proof-overview} for a detailed overview of these sections. Finally, in \S \ref{s:conclusions} we present our conclusions.

\section{Preliminaries}\label{s:prelims}

In this section, we introduce key preliminary material needed later in the paper. After some initial notation, we define the domains (the symmetric hypercubes), probability measures (the uniform and Chebyshev measures, respectively) and the Lebesgue--Bochner spaces. We next formalize our main smoothness assumption: namely, holomorphy in suitable (unions of) Bernstein polyellipses. We then introduce orthogonal polynomial expansions and best $s$-term polynomials approximations, before discussing sequence spaces and best $s$-term approximations of sequences. Finally, we conclude by reviewing algebraic and exponential rates of convergence for best $s$-term polynomial approximations, before a short discussion on lower and anchored sets.

\subsection{Notation}
We first introduce some notation. For $d \in \bbN$, we write $[d] = \{1,\ldots,d\}$. We also extend this to allow for $d = \infty$, in which case $[d] = \bbN$ is the set of positive integers. For $d \in \bbN \cup \{ \infty \}$, we write $\bm{e}_j$, $j \in [d]$, for the standard basis vectors, i.e.\ $\bm{e}_j = (\delta_{jk})_{k \in [d]}$. Also for $d \in \bbN \cup \{ \infty \}$, we write $\bbR^d$ or $\bbC^d$ for the vector space of real or complex vectors of length $d$. Note that when $d = \infty$, $\bbR^d$ and $\bbC^d$ are the vector spaces $\bbR^{\bbN}$ and $\bbC^{\bbN}$ of real- or complex-valued sequences indexed over $\bbN$.

For $1 \leq p \leq \infty$, we write $\nm{\cdot}_{p}$ for the usual vector $\ell^p$-norm and for the induced matrix $\ell^p$-norm. When $0 < p <1$, we use the same notation to denote the $\ell^p$-quasinorm. For $1 \leq p,q < \infty$ we define the matrix $\ell^{p,q}$-norm  of an $m \times n$ matrix $\bm{G} = (G_{ij})^{m,n}_{i,j=1}$ as  $\| \bm{G}\|^q_{p,q} := \sum_{j=1}^n \left ( \sum^{m}_{i=1} | G_{ij} |^p \right )^{q/p}$, and similarly when $p = \infty$ or $q = \infty$.

Throughout this paper, we consider sets of multi-indices. Let $d \in \bbN$. Then we define the multi-index set $\cF$ as the set of nonnegative multi-indices, i.e.
\be{
\label{Fdef-1}
\cF : = \mathbb{N}_0^d = \lbrace \bm{\nu} =(\nu_ k)_{k=1}^d: \nu_k \in \mathbb{N}_0 \rbrace,\qquad d < \infty.
}
When $d = \infty$, we consider multi-indices in $\bbN^{\bbN}_{0}$ with at most finitely-many nonzero terms, i.e.,\ we define
\begin{equation}
\label{Fdef-2}
\cF := \lbrace  \bnu = (\nu_k)^{\infty}_{k=1} \in \bbN_0^{\bbN} : | \{ k : \nu_k \neq 0 \} | < \infty\rbrace,\qquad d = \infty.
\end{equation}
In either finite or infinite dimensions, we write $\bm{0}$ and $\bm{1}$ for the multi-indices consisting of all zeros and all ones, respectively. Finally, the inequality $\bm{\mu} \leq \bm{\nu}$ is understood componentwise for any multi-indices $\bm{\mu}$ and $\bm{\nu}$.

\subsection{Domains and function spaces}

Let $\varrho = \varrho^{(1)}$ be a probability measure on $[-1,1]$.  In this paper, we focus on two main examples, the uniform and Chebyshev (arcsine) measures. These are defined by
\begin{equation}\label{meas-unif-1D}
\D \varrho(y) =  2^{-1}\D y,
\quad \mathrm{and} \quad \D \varrho(y) =  \frac{1}{\pi \sqrt{1-y^2}} \D y,\quad y \in \cU,
\end{equation}
respectively. See \S \ref{s:conclusions} for a short discussion on other domains and measures.
In finite dimensions, we let $\cU = [-1,1]^d$ be the symmetric $d$-dimensional hypercube and write $\bm{y} = (y_1,\ldots,y_d) \in \cU$ for the variable in this domain. We define a probability measure on $\cU$ as the product measure
\bes{
\varrho = \varrho^{(d)} : = \varrho^{(1)} \otimes \cdots \otimes \varrho^{(1)}.
}
In particular, the $d$-dimensional uniform and Chebyshev measures are given by
\begin{equation}\label{meas-unif}
\D \varrho(\bm{y}) =  2^{-d}\D \bm{y},
\quad \mathrm{and} \quad \D \varrho(\bm{y}) =  \prod_{k=1}^d \dfrac{1}{\pi \sqrt{1-y_k^2}} \D \bm{y},\quad \forall\bm{y} \in \cU,
\end{equation}
respectively.
In infinite dimensions, we consider the domain $\cU = [-1,1]^{\bbN}$ and write $\bm{y} = (y_1,y_2,\ldots) \in \cU$ for the variable in this domain. The Kolmogorov extension theorem   (see, e.g., \cite[\S 2.4]{tao2011introduction}) guarantees the existence of a tensor-product probability measure on $\cU$, which we denote as
\bes{
\varrho = \varrho^{(\infty)} =  \prod_{k\in \bbN} \varrho^{(1)}.
}
In either finite or infinite dimensions, for $1 \leq p \leq \infty$ we write $L^{p}_{\varrho}(\cU)$ for the corresponding weighted Lebesgue spaces of complex scalar-valued functions over $\cU$ and $\nm{\cdot}_{L^p_{\varrho} (\cU)}$ for their norms.

Throughout, we let $\cV$ be a separable Hilbert space over $\bbC$ (it presents few difficulties to consider a complex field instead of the real field). We write $\ip{\cdot}{\cdot}_{\cV}$ and $\nm{\cdot}_{\cV}$ for its inner product and norm.
We define the weighted (Lebesgue-)Bochner space $L^p_{\varrho}(\cU;\cV)$
as the space consisting of (equivalence classes of) strongly $\varrho$-measurable functions $f: \cU \rightarrow \cV$ for which $\nm{f}_{L^{p}_{\varrho}(\cU ; \cV)} < \infty$, where 
\be{
\nm{f}_{L^p_{\varrho}(\cU;\cV)} : = 
\begin{cases} 
\left( \int_{\cU} \nm{f( \y)}_{\cV}^p \D \varrho (\y) \right)^{1/p} & 1 \leq p < \infty ,
\\
\mathrm{ess} \sup_{\y \in \cU} \nm{f(\y)}_{\cV}  & p = \infty .
\end{cases}
\label{L_p_U_V}
}
Note that $L^{p}_{\varrho}(\cU)$ is a special case of $L^{p}_{\varrho}(\cU ; \cV)$ corresponding to $\cV = (\bbC,\abs{\cdot})$.

When $\cV$ is infinite dimensional, we usually cannot work directly with it. Hence, we consider a finite-dimensional discretization 
\begin{equation}
\label{eq:conforming}
\cV_h \subseteq \cV.
\end{equation}
Here $h > 0$ denotes a discretization parameter, e.g., the mesh size in the case of a finite element discretization (as is common in parametric DEs). In the context of finite elements, assuming \eqref{eq:conforming} corresponds to considering so-called \textit{conforming} discretizations. We let $\{ \varphi_k \}^{K}_{k=1}$ be a (not necessarily orthonormal) basis of $\cV_h$, where
$
K = K(h) = \dim(\cV_h).
$
We write $\cP_h: \cV \rightarrow \cV_h$
for the orthogonal projection onto $\cV_h$ and, for $f \in L^2_{\varrho}(\cU ; \cV)$, we let $\cP_h f \in L^2_{\varrho}(\cU ; \cV_h)$ be the function defined almost everywhere as
\be{
\label{Phf-def}
(\cP_h f)(\bm{y}) = \cP_h (f(\bm{y})), \quad \y \in \cU.
}

\subsection{Holomorphy}
Here we recall the definition of holomorphy and holomorphic extension for Hilbert-valued functions. We note that equivalent definitions are possible (see, e.g., \cite[Chapter 2]{herve1989analyticity}) and that the definition employed in this work is based on the notion of the Gateaux partial derivative. For other details on differentiability of Hilbert-valued functions we refer to \cite[Chapter 17]{BauschkeCombettes2017},   and the references  therein. Note the following definitions apply in both the finite- ($d \in \bbN$) and infinite- ($d = \infty$) dimensional settings, where we recall that $[d] = \bbN$ and $\bbC^d = \bbC^{\bbN}$ when $d = \infty$.

\begin{definition}
[Holomorphy; finite- or infinite-dimensional case]
Let $d \in \bbN \cup \{ \infty \}$, $\mathcal{O} \subseteq \mathbb{C}^{d}$ be an open set and  $\cV$ be a separable Hilbert space. A  function $f: \mathcal{O} \rightarrow \cV$ is holomorphic in $\mathcal{O}$ if and only if it is holomorphic with respect to each variable  in  $\mathcal{O}$. That is to say, for any $z \in \mathcal{O}$ and any $j \in [d]$, the following limit exists in $\cV$:
\[
\lim_{\substack{h \in \bbC \\ h \rightarrow 0}}  \dfrac{f(z+h \bm{e}_j)-f(z)}{h} \in \cV.
\]
\end{definition}

Let $f : \cU \rightarrow \cV$ and $\cU \subset \cO \subseteq \bbC^d$ be an open set. If there is a function $\tilde{f} : \cO \rightarrow \cV$ that is holomorphic in $\cO$ and for which $\tilde{f} |_{\cU} = f$, then we say that $f$ has a \textit{holomorphic extension to $\cO$}, or simply, that $f$ is \textit{holomorphic in $\cO$}.  In this case, we also define $\nm{f}_{L^{\infty}(\cO; \cV)}:=\nmu{{\widetilde{f}}}_{L^{\infty}(\cO; \cV)}$ or, when $\cV = \bbC$, simply $\nm{f}_{L^{\infty}(\cO)}$. If $\cO$ is a closed set, then we say that $f$ is holomorphic in $\cO$ if it has a holomorphic extension to some open neighbourhood of $\cO$.

We are interested in approximating Hilbert-valued functions $f : \cU \rightarrow \cV$ that are holomorphic in suitable complex regions containing $\cU$ -- specifically, regions defined by \textit{Bernstein (poly)ellipses}.  When $d = 1$ the Bernstein ellipse of parameter $\rho > 1$ is defined by
\bes{\cE_\rho = \left\lbrace \tfrac{1}{2} (z +z^{-1}): z\in \bbC, 1 \leq |z|\leq \rho\right\rbrace \subset \bbC. 
}
This is an ellipse with $\pm 1$ as its foci and major and minor semi-axis lengths given by $\frac12 (\rho\pm\rho^{-1})$. For $d \in \bbN \cup \{\infty\}$, given $\brho = (\rho_j)^{d}_{j=1} \in \bbR^d$ with $\brho> \bm{1}$, we define the Bernstein polyellipse as the Cartesian product
\[
\cE(\brho) =\cE({\rho_1}) \times \cE({\rho_2}) \times \cdots   \subset \bbC^d.
\]
We denote the class of 
Hilbert-valued functions that are holomorphic in $\cE(\bm{\rho})$ with norm at most one as
\be{
\label{B-def}
\cB(\bm{\rho}) = \left \{ f : \cU \rightarrow \cV, \mbox{$f$ holomorphic in $\cE(\bm{\rho})$, $\nm{f}_{L^{\infty}(\cE({\bm{\rho})} ; \cV)} \leq 1$}\right \}.
}
In infinite dimensions, we also consider a class of functions that are holomorphic in a certain union of Bernstein polyellipses. Let $0 < p <1$, $\varepsilon > 0$ and $\bm{b} = (b_{j})_{j \in \bbN} \in \ell^p(\bbN)$. We define
\bes{
\cR({\b,\varepsilon}) =  \bigcup \left\lbrace \cE(\brho): \brho \geq \bm{1},\  \sum_{j=1}^{\infty} \left( \dfrac{\rho_j+\rho_j^{-1}}{2} -1 \right) b_j \leq  \varepsilon \right\rbrace.
}
In analogy with $\cB(\bm{\rho})$, we write
\be{
\label{B-b-eps-def}
\cB(\bm{b},\varepsilon) = \left \{ f : \cU \rightarrow \cV, \mbox{$f$ holomorphic in $\cR({\b,\varepsilon})$, $\nm{f}_{L^{\infty}(\cR({\b,\varepsilon}) ; \cV)} \leq 1$}\right \}
}
for the corresponding space of functions that are  holomorphic in $\cR(\bm{b},\varepsilon)$ with norm at most one.

\subsection{Orthogonal polynomials, polynomial expansions and best $s$-term polynomial approximation}

Under mild assumptions on $\varrho^{(1)}$ (see, e.g., \cite[\S 2.1]{narayan2018computation} or \cite[\S 2.2]{szego1975orthogonal}), there exists a unique orthonormal polynomial basis $\{ \Psi_{\nu} \}_{\nu \in \bbN_0}$ of $L^2_{\varrho}([-1,1])$, where $\Psi_{\nu} = \Psi^{(1)}_{\nu}$ is a polynomial of degree $\nu$. For the measures \R{meas-unif-1D}, these are the Legendre and Chebyshev polynomials, respectively.
 Given the corresponding tensor-product measure $\varrho$ on $\cU = [-1,1]^d$, we construct an orthonormal basis
\bes{
\{ \Psi_{\bm{\nu}} \}_{\bm{\nu} \in \cF} \subset L^2_{\varrho}(\cU)
}
of $L^2_{\varrho}(\cU)$ via tensorization
\bes{
\Psi_{\bm{\nu}}(\bm{y}) = \prod_{k \in [d]} \Psi_{\nu_k}(y_k),\quad \bm{y} \in \cU,\ \bm{\nu} \in \cF.
}
Note that $\Psi^{(1)}_{0} = 1$ since $\varrho^{(1)}$ is a probability measure. Therefore, since $\bm{\nu} \in \cF$ has only finitely-many nonzero entries, in infinite dimensions this equivalent to 
\bes{
\Psi_{\bm{\nu}}(\bm{y}) = \prod_{k : \nu_k \neq 0} \Psi_{\nu_k}(y_k),
}
which is a product of finitely-many terms.

Let $f \in L^2_{\varrho}(\cU ; \cV)$. Then it has the convergent expansion (in $L^2_{\varrho}(\cU ; \cV)$) given by
\be{
\label{f_exp}
f = \sum_{\bm{\nu} \in \cF} c_{\bm{\nu}} \Psi_{\bm{\nu}},\qquad c_{\bm{\nu}} :=\int_{\cU} f(\bm{y}) \Psi_{\bm{\nu}}(\bm{y})    \D \varrho (\bm{y}) \in \cV,
}
where the \textit{coefficients} $c_{\bm{\nu}}$ are elements of $\cV$. 
Now let $S \subset \cF$ be a finite index set and  
\be{
\label{span_P}
\cP_{S;\cV} = \left \{ \sum_{\bm{\nu} \in S} c_{\bm{\nu}} \Psi_{\bm{\nu}} : c_{\bm{\nu}} \in \cV \right \} \subset L^2_{\varrho}(\cU ; \cV).
}
Given this, the $L^2(\cU ; \cV)$-norm \textit{best $s$-term polynomial approximation} $f_s$ of $f$ is defined as
\be{
\label{f_best_s_term}
f_{s} \in \argmin{} \left \{ \nm{f - g}_{L^2_{\varrho}(\cU ; \cV)} : g \in \cP_{S,\cV},\ S \subset \cF,\ |S| = s \right \}.
}
Note that $f_s$ is has the explicit expression
\be{
\label{f_best_s_term_equiv}
f_{s} = \sum_{\bm{\nu} \in S^*} c_{\bm{\nu}} \Psi_{\bm{\nu}},
}
where $S^* \subset \cF$, $|S^*| = s$, is a set of consisting of the multi-indices of the largest $s$ values of the coefficient norms $(\nm{c_{\bm{\nu}}}_{\cV})_{\bm{\nu} \in \bbN^d_0}$. By Parseval's identity, the error satisfies
\be{
\label{err-equiv}
\nmu{f - f_{s}}_{L^2_{\varrho}(\cU ; \cV)} = \sqrt{\sum_{\bm{\nu} \notin S^*} \nm{c_{\bm{\nu}}}^2_{\cV} }.
}

\subsection{Sequence spaces and best $s$-term approximation of sequences}\label{ss:seq-space}

The equivalence \R{err-equiv} motivates studying $s$-term approximation of the sequences of polynomial coefficients. To do this, we now introduce necessary further notation.

Let $\Lambda \subseteq \cF$ denote a (possibly infinite) multi-index set. We write $\bm{v} = (v_{\bm{\nu}})_{\bm{\nu} \in \Lambda}$ for a sequence with $\cV$-valued entries, $v_{\bm{\nu}} \in \cV$. For $1 \leq p \leq \infty$, we define the space $\ell^p(\Lambda;\cV)$ as the set of those sequences $\bm{v} = (v_{\bm{\nu}})_{\bm{\nu} \in \Lambda}$ for which $\|{\bm{v}}\|_{p;\cV} < \infty$, where 
\bes{
\|{\bm{v}}\|_{p;\cV} : = \left \{ \begin{array}{lc} \left ( \sum_{\bm{\nu} \in \Lambda} \|{v_{\bnu}}\|^p_{\cV} \right )^{1/p} & 1 \leq p < \infty  ,
\\ \sup_{\bm{\nu} \in \Lambda} \nm{v_{\bm{\nu}}}_{\cV} & p = \infty .\end{array} \right  .
}
Note that $\ell^2(\Lambda;\cV)$ is a Hilbert space with inner product
\[
\ip{\bm{u}}{\bm{v}}_{2;\cV} = \sum_{\bnu \in \Lambda}  \ip{u_{\bnu}}{v_{\bnu}}_{\cV} .
\]
On occasion, we will consider complex, scalar-valued sequences. In this case, $\cV = (\bbC,\abs{\cdot})$ in the various definitions above. For ease of notation, we simply write $\ell^p(\Lambda)$, $\nm{\cdot}_{p}$, $\ip{\cdot}{\cdot}_2$ and so forth in this case.

\defn{[Sparsity]
\label{d:sparsity}\index{sparsity!standard}
Let $\Lambda \subseteq \cF$ and $\bm{c} = (c_{\bm{\nu}})_{\bm{\nu} \in \Lambda}$ be a $\cV$-valued sequence. The \textit{support} of $\bm{c}$ is the set
\be{
\label{sequence_support}
\supp(\bm{c}) = \{ \bm{\nu} \in \Lambda : \nm{c_{\bm{\nu}}}_{\cV} \neq 0\}.
}
A sequence is \textit{$s$-sparse} for some $s\in\mathbb{N}_0$ satisfying $s \leq |\Lambda|$ if it has at most $s$ nonzero entries, i.e.,
\bes{
| \supp(\bm{c}) | \leq s.
}
}

\defn{[best $s$-term approximation error]
\label{def:best_s_term}
Let $\Lambda \subseteq \cF$, $0 < p \leq \infty$, $\bm{c} \in \ell^p(\Lambda;\cV)$ and $s\in \mathbb{N}_0$ with $s \leq |\Lambda|$. The \textit{$\ell^p$-norm best $s$-term approximation error} of $\bm{c}$ is
\be{
\label{best_s_term}
\sigma_{s}(\bm{c})_{p;\cV} = \min \left \{ \nm{\bm{c} - \bm{z}}_{p;\cV} : \bm{z} \in \ell^p(\Lambda;\cV),\ | \supp(\bm{z}) | \leq s \right \}.
}
}

Let $\bm{c} = (c_{\bm{\nu}})_{\bm{\nu} \in \cF}$ be the coefficients of some function $f \in L^2_{\varrho}(\cU ; \cV)$, as defined in \R{f_exp}. Then, when $p = 2$, we have the following:
\bes{
\sigma_{s}(\bm{c})_{2;\cV} = \nmu{f - f_{s}}_{L^2_{\varrho}(\cU ; \cV)},
}
where $f_s$ is its best $s$-term polynomial approximation \R{f_best_s_term}. Therefore, we can study the error of $f_s$ by studying the quantity $\sigma_{s}(\bm{c})_{2;\cV}$. For notational purposes, we denote this quantity in terms of the coefficients $\bm{c}$.  However, on some occasions, this term is referred to as $\sigma_{s}(f)_{2;\cV}$.

\subsection{Rates of best $s$-term polynomial approximation}\label{ss:best-s-term}

As noted, best $s$-term polynomial approximation of holomorphic functions is a well-studied subject, especially in the context of solutions of parametric DEs. See, e.g., \cite{bonito2021polynomial,bieri2009sparse,cohen2010convergence, cohen2011analytic,hansen2013analytic,chkifa2015breaking,todor2007convergence,opschoor2022exponential,
tran2017analysis,beck2014convergence,beck2012optimal} and, in particular, \cite{cohen2015approximation} and \cite[Chpt.\ 3]{adcock2021sparse}. In this section, we recap two standard types of error decay rates for this approximation, those of \textit{algebraic} and \textit{exponential} type, respectively. Note that these results are for Chebyshev and Legendre polynomial approximations -- the main focus of the work. The latter type of decay rate holds in finite dimensions, while the former holds in both finite and infinite dimensions. In this work, these error decay rates serve as the optimal benchmark against which to compare the approximations computed from sample values.

The following two results are standard, and have appeared in various different guises in the aforementioned works. 

\thm{[Algebraic rates of convergence; finite-dimensional case]
\label{t:best_s_term_alg}
Let $0 < p   \leq 1$ and $f \in \cB(\bm{\rho})$ for some $\bm{\rho} > \bm{1}$. Let $\bm{c} = (c_{\bm{\nu}})_{\bm{\nu} \in \bbN^d_0}$ be as in \R{f_exp}. Then,  for every  $s \geq 1$ there are sets $S_1,S_2 \subset \cF$, $|S_1|, |S_2| \leq s$, such that
\be{
\label{thm:alg_rates_thresh:eq:bound_finite}
\nm{f - f_{S_1}}_{L^2_{\varrho}(\cU ; \cV)}
\leq  C \cdot  s^{1/2-1/p},
\qquad
\nm{f - f_{S_2}}_{L^{\infty}(\cU ; \cV)}
  \leq  C   \cdot  s^{1-1/p},
}
where $f_{S_i} = \sum_{\bm{\nu} \in S_i} c_{\bm{\nu}} \Psi_{\bm{\nu}}$ for $i = 1,2$ and $C = C(d,p,\bm{\rho}) > 0$ depends on $d$, $p$ and $\bm{\rho}$ only.
}

\thm{[Algebraic rates of convergence; infinite-dimensional case]
\label{t:best_s_term_alg_inf}
Let $0 < p <1$, $\varepsilon > 0$, $\bm{b} = (b_{j})_{j \in \bbN} \in \ell^p(\bbN)$ and $f \in \cB(\b,\varepsilon)$, where $\cB(\b,\varepsilon)$ is as in \R{B-b-eps-def}. Then,  for every  $s \geq 1$ there are sets $S_1,S_2 \subset \cF$, $|S_1|, |S_2| \leq s$, such that
\be{
\label{thm:alg_rates_thresh:eq:bound}
\nm{f - f_{S_1}}_{L^2_{\varrho}(\cU ; \cV)}
\leq  C \cdot  s^{1/2-1/p},
\qquad
\nm{f - f_{S_2}}_{L^{\infty}(\cU ; \cV)}
  \leq  C   \cdot  s^{1-1/p},
}
where $f_{S_i} = \sum_{\bm{\nu} \in S_i} c_{\bm{\nu}} \Psi_{\bm{\nu}}$ for $i = 1,2$ and $C = C(\b,\varepsilon,p)  > 0$ depends on $\b$, $\varepsilon$ and $p$ only.
}

 Observe that the curse of dimensionality is not avoided in the constant $C(d,p,\bm{\rho})$ in \eqref{thm:alg_rates_thresh:eq:bound_finite}, but it is avoided in the rate Conversely, \R{thm:alg_rates_thresh:eq:bound} holds in infinite dimensions.
 
We next state a result on exponential convergence in finite dimensions. Such rates have been established in various different works (see, e.g.,\cite{beck2014convergence,beck2012optimal,opschoor2022exponential,cohen2015approximation,tran2017analysis}). The following result is a minor modification of \cite[Thm.\ 3.25]{adcock2021sparse}, in which we allow arbitrary $s \geq 1$ at the expense of a constant $C$ in the error bound.

\thm{[Exponential rates of convergence; finite-dimensional case]
\label{t:best_s_term_exp_1}
Let $f \in \cB(\bm{\rho})$ for some $\bm{\rho} > \bm{1}$ and $\bm{c} = (c_{\bm{\nu}})_{\bm{\nu} \in  \bbN^d_0}$ be as in \R{f_exp}. Then, for every $s \geq 1$ there is a set $S \subset \cF$, $|S| \leq s$, such that
\be{
\label{thm:exp_rates_thresh:eq:bound}
\nm{f - f_S}_{L^2_{\varrho}(\cU ; \cV)}  \leq \nm{f - f_S}_{L^{\infty}(\cU ; \cV)}
\leq C  \cdot \exp(-\gamma s^{1/d} ),
}
for all
\be{
\label{exp_rates-gamma}
0 < \gamma < (d+1)^{-1} \left ( d! \prod_{j = 1}^d \ln(\rho_j) \right )^{1/d},
}
where $f_{S} = \sum_{\bm{\nu} \in S} c_{\bm{\nu}} \Psi_{\bm{\nu}}$ and $C = C(d,\gamma,p,\bm{\rho}) > 0$ is a constant depending on $d$, $\gamma$, $p$  and $\bm{\rho}$ only.
}

In Appendix \ref{s:best-poly-rates} we show how these three theorems can be obtained as immediate consequences of several more general results.

\rem{
It is possible to improve the rate \R{thm:exp_rates_thresh:eq:bound} by removing the $(d+1)^{-1}$ factor in \R{exp_rates-gamma} \cite{tran2017analysis}. The difficulty in doing this is that such rates are not necessarily attained in lower sets (this is, however, true if $\bm{\rho}$ is sufficiently large -- see \cite[Lem.\ 7.20]{adcock2021sparse}). As we discuss next, lower sets are a crucial ingredient in our analysis. Conversely, the rates described in Theorem \ref{t:best_s_term_exp_1} can always be attained in lower sets.
}

\subsection{Lower and anchored sets}\label{s:lower-anchored}

Our objective in this work is to construct a polynomial approximation that satisfies similar error bounds to those of the best $s$-term approximation $f_s$, for any holomorphic function $f$. Hence, ideally, we would have access to the multi-index set $S$ corresponding to the largest $s$ coefficients of $f$ (measured in the $\cV$-norm). As discussed, this is not possible in general, since the only information we have about $f$ is its values at a finite number of sample points. Another problem is that such coefficients could occur at arbitrarily-large multi-indices, thus necessitating a search over infinitely-many multi-indices. Fortunately, it is well known that near-best $s$-term polynomial approximations can be constructed using sets of multi-indices with additional structure. These are \textit{lower} sets (used in the finite-dimensional case) and \textit{anchored} sets (used in the infinite-dimensional case). Classical references for lower and anchored sets include \cite{kuntzman1959methodes,temlyakov1980approximation,lorentz1986solvability,deboor1992computational}. More recently, these structures have been used extensively in the construction of interpolation, least-squares and compressed sensing schemes for polynomial approximation with desirable sample complexity bounds (see, e.g., \cite{adcock2021sparse} and references therein).

\defn{
\label{def:lower-anchored-set}
A set $\Lambda \subseteq \cF$ is \textit{lower} if the following holds for every $\bm{\nu},\bm{\mu} \in \cF$:
\bes{
\label{lowet_set_cond}
(\bm{\nu} \in \Lambda \text{ and } \bm{\mu} \leq \bm{\nu} )\Longrightarrow \bm{\mu} \in \Lambda.
}
A set $\Lambda \subseteq \cF$ is \textit{anchored} if it is lower and if the following holds for every $j \in \mathbb{N}$:
$$
\bm{e}_j \in \Lambda \Longrightarrow 
\{\bm{e}_1,\bm{e}_2, \ldots, \bm{e}_{j} \}\subseteq \Lambda.
$$
}

Lower sets are typically used in finite-dimensional settings, with anchored sets being employed in infinite dimensions. They are a key notion we exploit in this paper. To underscore the usefulness of these structures, we remark in passing that the rates articulated in Theorems \ref{t:best_s_term_alg}--\ref{t:best_s_term_exp_1} can, up to possible changes in the constants, also be attained using $s$-term approximations in lower or anchored sets. See Appendix \ref{s:best-poly-rates}.

\section{Problem statement and main results}\label{s:prob-main-res}

In this section, we first formally define the problem we aim to solve before stating our main results.  This paper concerns algorithms for computing approximation of Hilbert-valued functions from finitely-many sample values. We  define this concept formally in a moment. For now, though, we consider that an algorithm must take a finite input and produce a finite output. Hence, in order to discuss algorithms, we first need to define what these finite inputs and outputs are in our setting.

\subsection{Samples}

Let $f \in L^2_{\varrho}(\cU;\cV)$ be the function we seek to approximate. Throughout this work, we consider $m$ \textit{sample points} $\y_1,\ldots,\y_m \in \cU$ drawn randomly and independently according to the probability measure $\varrho$. Corresponding to each sample point, we consider the noisy \textit{sample values}
\bes{
d_i = \ f(\bm{y}_i) + n_i \in \cV_h,\quad i = 1,\ldots,m,
}
where $\bm{n} = (n_i)^{m}_{i=1} \in \cV^{m}$ is an error term, referred to as the \textit{sampling error}. Observe that the samples values $d_i$ are assumed to be elements of the finite-dimensional space $\cV_h$. This is a natural assumption to make. Indeed, in the context of parametric DEs, the value $f(\bm{y})$ (the solution of the DE with parameter value $\bm{y}$) is typically computed via a (finite element) discretization of the DE, thus yielding an element of $\cV_h$, which is the corresponding discrete (finite element) space.

As a result of the assumption $d_i \in \cV_h$, the error term $n_i$ encompasses the error involved in approximating $f(\bm{y}_i) \in \cV$ by an element of $\cV_h$ (e.g., the (finite element) discretization error in the context of a parametric DE). Note that we do not specify precisely how such an approximation is performed, nor how large an error this results in. In other words, we consider the computation that evaluates $f$ at $\bm{y}_i$ as a black box. A particular case of interest is when the $d_i$ are the orthogonal projections of the exact sample values $f(\bm{y}_i)$, i.e.
\bes{
d_i = \cP_h(f(\bm{y}_i)),\qquad i = 1,\ldots,m.
}
However, we do not assume this in what follows, since in practice the numerical computation that yields the $d_i$ may not involve computing the projection $\cP_h$. Our objective is to develop algorithms for which the error scales linearly in $\nm{\bm{n}}_{2;\cV}$, the norm of the noise, thus accounting for any black box mechanism for computing the samples.

Recall that we consider a basis $\{ \varphi_k \}^{K}_{k=1}$ for $\cV_h$. We assume that the computation that evaluates $f(\bm{y}_i)$ produces the coefficients of the sample values $d_i$ in this basis (i.e.\ the finite element coefficients in the aforementioned example). Therefore, we now write the sample values as
\begin{equation}\label{data_we_get}
d_i = \ f(\bm{y}_i) + n_i = \sum^{K}_{i=1} d_{ik} \varphi_k, \quad i =1,\ldots,m,
\end{equation}
and consider the values $d_{ik} \in \bbC$ as the \textit{data} we obtain by sampling $f$.

\subsection{Problem statement}\label{ss:prob-stat}
We now formally define the  input and output of the algorithm.  The \textit{input} of the algorithm is the collection of sample points $(\bm{y}_i)^{m}_{i=1}$ and the array of $mK$ values $(d_{i,k})_{i,k=1}^{m,K} \in \bbC^{m \times K}$ defined by \eqref{data_we_get}.  

We next define the output. To this end, we first fix a multi-index set $\Lambda \subset \cF$  of size $|\Lambda| = N$ for some $N \geq 1$. This set defines a polynomial space $\cP_{\Lambda ; \cV_h}$, as in \R{span_P}, within which we shall construct the resulting polynomial approximation. Hence, we consider an approximation of the form $\hat{f} \in  \cP_{\Lambda ; \cV_h}$   given by 
\begin{equation}
\label{fhatdef}
\hat{f} : \bm{y} \mapsto \sum^{N}_{j=1} \left ( \sum^{K}_{k=1} \hat{c}_{jk}  \varphi_k \right ) \Psi_{\bnu_j}(\bm{y}),
\end{equation}
where $\hat{c}_{j,k}  \in \bbC$ for $j \in [N],k \in [K]$  and $\bm{\nu}_1,\ldots,\bm{\nu}_N$ is some indexing of the multi-indices in $\Lambda$.   In this way, we define formally the \textit{output} of the algorithm as the coefficients $(\hat{c}_{jk} )_{j,k=1}^{N,K} \in \bbC^{N \times K}$.

Finally, in order to define an algorithm we need one additional ingredient. Let
\be{
\label{Gram-matrix-def}
\bm{G} = \left ( \ip{\varphi_j}{\varphi_k}_{\cV} \right )^{K}_{j,k = 1} \in \bbC^{K \times K}
}
denote the Gram matrix of the basis $\{ \varphi_k \}^{K}_{k=1} \subset \cV_h$. Note that $\bm{G}$ is self adjoint and positive definite. However, $\bm{G}$ is only equal to the identity when $\{ \varphi_k \}^{K}_{k=1}$ is orthonormal. In what follows, we assume that it is possible to perform matrix-vector multiplications with $\bm{G}$. In other words, we have access to the function
\bes{
\cT_{\bm{G}} : \bbC^K \rightarrow \bbC^K,\ \bm{x} \mapsto \bm{G} \bm{x}.
}
For convenience, we write $F(\bm{G})$ for the maximum number of arithmetic operations and comparisons required to evaluate $\cT_{\bm{G}}(\bm{x})$ for arbitrary $\bm{x}$. Note that $F(\bm{G}) \leq K^2$ in general. However, this may be smaller when $\bm{G}$ is structured. For instance, in the case of a finite element discretization, this computation can often be performed in $\ord{K}$ operations.

\defn{
[Algorithm for polynomial approximation of Hilbert-valued functions]
\label{d:AlgPolyApprox}
Let $\Lambda \subset \cF$ of size $|\Lambda| = N$ be given, along with an indexing $\bm{\nu}_1,\ldots,\bm{\nu}_N$ of the multi-indices in $\Lambda$. An \textit{algorithm for polynomial approximation of Hilbert-valued functions from sample values} is a mapping
\bes{
\cA : \cU^m \times \bbC^{m \times K} \rightarrow \bbC^{N \times K},\ \left ( (\bm{y}_i)^{m}_{i=1} , (d_{i,k})^{m,K}_{i,k=1} \right ) \mapsto (\hat{c}_{jk})^{N,K}_{j,k=1},
}
for which the evaluation of $\cA((\bm{y}_i) , (d_{i,k}))$ 
involves only finitely-many arithmetic operations (including square roots), comparisons and evaluations of the matrix-vector multiplication function $\cT_{\bm{G}}$. If $(d_{ik})$ is as in \R{data_we_get} for some $f \in L^2_{\varrho}(\cU ; \cV)$, then the resulting approximation $\hat{f}$ of $f$ is given by \R{fhatdef}, where $(\hat{c}_{jk}) = \cA((\bm{y}_i) , (d_{i,k}))$. The \textit{computational cost} of an algorithm $\cA$ is the maximum number of arithmetic operations and comparisons (including those used in the evaluation of $\cT_{\bm{G}}$) used to compute the output from any input.
}

\rem{
\label{r:oracle-knowledge}
As formulated above, it is up to the user to choose a suitable multi-index set $\Lambda$. 
Fortunately, as we see in our main results below, this multi-index set is given simply and explicitly in terms of $m$ and another parameter $\epsilon$ (a failure probability). In particular, no `oracle' knowledge of the function being approximated is required. Thus, one can also make the stronger assertion in what follows in which the algorithm takes the same input, but outputs both the desired index set $\Lambda$ and the polynomial coefficients. For ease of presentation, we shall not do this.
}

\rem{
When $d = \infty$ each sample point $\bm{y}_i$ is an infinite sequence of real numbers. It is implicit in Definition \ref{d:AlgPolyApprox} that the algorithm only accesses finitely-many entries of this sequence. This does not cause any problems. As noted, the polynomial approximation is obtained in the index set $\Lambda$, which is a finite subset of $\cF$. Hence, the multi-indices in $\Lambda$ are nonzero only in their first $n$ entries, for some $n$. Therefore, it is only necessary to access the first $n$ entries of each sequence $\bm{y}_i$. More concretely, in our main results below, the polynomial approximation in infinite dimensions is obtained in a multi-index set $\Lambda = \Lambda^{\mathsf{HCI}}_n$ in which only the first $n$ terms can be nonzero, where $n$ is an integer given explicitly in terms of $m$ and $\epsilon$.
}

\subsection{Main results}\label{Sec_Main_results}

We now present the main results of this paper. We reiterate at this stage that these results are formulated for Chebyshev and Legendre polynomials. See \S \ref{s:conclusions} for some further discussion on other polynomial systems.

As noted above, these results employ specific choices of the index set $\Lambda$ in order to obtain the desired approximation rates. Specifically, in finite dimensions, we consider the \textit{hyperbolic cross} index set
\be{
\label{HC_index_set}
\Lambda = \Lambda^{\mathsf{HC}}_{n,d} = \left \{ \bm{\nu} = (\nu_k)^{d}_{k=1} \in \bbN^d_0 : \prod^{d}_{k=1} (\nu_k+1) \leq n \right \} \subset \bbN^d_0.
}
We term $n$ the \textit{order} of the hyperbolic cross. Note that it is common to consider \R{HC_index_set} as the hyperbolic cross of order $n-1$. We use $n$ here as it is slightly more convenient for this work. When defined this way, $\Lambda^{\mathsf{HC}}_{n,d}$ is in fact the union of all lower sets (see Definition \ref{def:lower-anchored-set}) in $d$ dimensions of size at most $n$ (see, e.g., \cite[Prop.\ 2.5]{adcock2021sparse}). Thus, this set is a natural choice for polynomial approximation.

In infinite dimensions, we define the following index set
\be{
\label{HC_index_set_inf}
\Lambda = \Lambda^{\mathsf{HCI}}_{n} = \left \{ \bm{\nu} = (\nu_k)^{\infty}_{k=1} \in \cF : \prod^{n}_{j=1} (\nu_k + 1) \leq n,\ \nu_k = 0,\ k > n \right \} \subset \cF.
}
Similarly, the union of all anchored sets (Definition \ref{def:lower-anchored-set}) of size at most $n$ in infinite dimensions is a subset  of $\Lambda^{\mathsf{HCI}}_{n}$ (see, e.g., \cite[Prop.\ 2.18]{adcock2021sparse}).
Note that $\Lambda^{\mathsf{HCI}}_{n}$ is isomorphic to $\Lambda^{\mathsf{HC}}_{n,n}$ under the restriction map $\bm{\nu} = (\nu_k)^{\infty}_{k=1} \in \cF \mapsto (\nu_k)^{n}_{k=1} \in \bbN^d_0$. For convenience, we now also define
\be{
\label{Thetadef}
N = \Theta(n,d) 
= \begin{cases} | \Lambda^{\mathsf{HC}}_{n,d} | & d < \infty ,\\ 
| \Lambda^{\mathsf{HCI}}_{n} | = |\Lambda^{\mathsf{HC}}_{n,n}| & d = \infty ,\end{cases}
}
as the cardinality of the index set employed. 
In general, the exact behaviour of $\Theta(n,d)$ is unknown. However, it admits a variety of different bounds. These are summarized as follows for $d < \infty$:
\be{
\label{N_bound}
N = | \Lambda^{\mathsf{HC}}_{n,d} | \leq \min \left \{ 2 n^3 4^d , \E n^{2 + \log(d)/\log(2)} , \frac{n (\log(n) + d \log(2))^{d-1}}{(d-1)!}  \right \}.
}
The bounds are based on \cite{kuhn2015approximation,chernov2016new}. See also \cite[Lem.\ B.3--B.5]{adcock2021sparse}. 

Finally, we also define
\be{
\label{main_alpha_def}
\alpha = \begin{cases} 1 & \mbox{Legendre},
\\
 \log(3) / \log(4) & \mbox{Chebyshev},
 \end{cases}
}
and, given $m \geq 3$ and $ \epsilon  \in (0,1)$,
\begin{equation}
\label{Ldef}
L = L(m,d,\epsilon) = \begin{cases}
\log(   {m}) \cdot \left ( \log(   {m}) \cdot \min \{ \log(   {m}) + d , \log(\E d) \cdot \log(  {m}) \} + \log(\epsilon^{-1}) \right ) & d < \infty,
\\
\log( m) \cdot \left ( \log^3( m) + \log(\epsilon^{-1}) \right ) & d = \infty.
\end{cases}
\end{equation}

\subsubsection{Algebraic rates of convergence, finite dimensions}\label{ss:main-res-alg-fin}

\thm{
[Existence of a mapping; algebraic case, finite dimensions]
\label{t:main-res-map-alg}
Let  $d \in \bbN$, $\{ \Psi_{\bm{\nu}} \}_{\bm{\nu} \in \bbN^d_0} \subset L^2_{\varrho}(\cU)$ be either the orthonormal Chebyshev or Legendre basis and $\{ \varphi_k \}^{K}_{k=1}$ be a basis for $\cV_h$. Then for every $m \geq 3$, $0 < \epsilon < 1$ and $K \geq 1$, there is a mapping 
\bes{
\cM : \cU^m \times \bbC^{m \times K} \rightarrow \bbC^{N \times K}, 
}
where  $N =\Theta(n ,d)$ is as in \R{Thetadef} with $n = \lceil m / L \rceil$ and $L = L(m,d,\epsilon)$ as in \R{Ldef},
with the following property. Let $f \in \cB(\bm{\rho})$ for arbitrary $\bm{\rho} > \bm{1}$, draw $\bm{y}_1,\ldots,\bm{y}_m$ randomly and independently according to $\varrho$ and let $(d_{ik})^{m,K}_{i,k=1} \in \bbC^{m \times K}$ be as in \R{data_we_get} for arbitrary noise terms $\bm{n} = (n_i)^{m}_{i=1} \in \cV$. Let $(\hat{c}_{jk}) = \cM((\bm{y}_i),(d_{ik}))$ and define the approximation $\hat{f}$ as in \R{fhatdef} based on the index set $\Lambda = \Lambda^{\mathsf{HC}}_{n,d}$. Then the following holds with probability at least $1-\epsilon$. The error satisfies
\ea{
\label{err-bound-map-alg_1}
\nmu{f - \hat{f}}_{L^2_{\varrho}(\cU;\cV)} \leq c_1 \cdot \zeta, \qquad \nmu{f - \hat{f}}_{L^{\infty}(\cU;\cV)}  \leq c_2 \cdot  \sqrt{\frac{m}{L}} \cdot \zeta,
}
for any $0 < p \leq 1$, where
\be{
\label{zeta-alg-def}
\zeta := C \cdot \left ( \frac{m}{c_0 L} \right )^{1/2-1/p}  +  \frac{\nm{\bm{n}}_{2;\cV} }{\sqrt{m}} +    \nmu{f - \cP_h(f)}_{L^{\infty}(\cU ; \cV)},
}
$c_0,c_1,c_2 \geq 1$ are universal constants
and 
$C = C(d,p,\bm{\rho})$ depends on $d$, $p$ and $\bm{\rho}$ only.
}

We now make several remarks about this result. The same remarks apply (with obvious modifications) to all subsequent results as well. First, notice how the index set $\Lambda$ in which the approximation is constructed is given completely explicitly in terms of $m$, $d$ and $\epsilon$. Thus, as claimed in Remark \ref{r:oracle-knowledge}, no `oracle' information about the function being approximated is required. Indeed, notice that the mapping described in this theorem is \textit{universal} in the sense that its applies equally to \textit{any} function $f \in \cB(\bm{\rho})$ and \textit{any} $\bm{\rho} > \bm{1}$.

A key aspect of this theorem is the factor $\zeta$, defined in \R{zeta-alg-def}, which determines the error bounds \R{err-bound-map-alg_1}. As claimed in \S \ref{s:main-contribs}, this incorporates three main key errors arising in the approximation process:
\bulls{
\item[(i)] \textit{The approximation error.} This is the algebraically-decaying term $E_{\textsf{app}} = C \cdot (m/(c_0 L))^{1/2-1/p}$. It is completely equivalent to the best $s$-term approximation error bound in Theorem \ref{t:best_s_term_alg}, except with $s$ replaced by $m/(c_0 L)$.
\item[(ii)] \textit{The sampling error.} This is the term $E_{\textsf{samp}} = \nm{\bm{n}}_{2;\cV} / \sqrt{m}$, where $\bm{n} = (n_i)^{m}_{i=1}$ is as in \R{data_we_get}. In other words, the effect of any errors in computing the sample values $f(\bm{y}_i)$ enters linearly in the overall error bound.
\item[(iii)] \textit{The physical discretization error.} This is the term $E_{\textsf{disc}} = \nmu{f - \cP_h(f)}_{L^{\infty}(\cU;\cV)}$. It describes the effect of working in the finite-dimensional subspace $\cV_h$, instead of the full space $\cV$. Critically, it depends on the orthogonal projection (best approximation) $\cP_h(f)$ of $f$ from $\cV_h$.
}
Notice that (i) also describes the \textit{sample complexity} of the scheme. Indeed, Theorem \ref{t:main-res-map-alg} asserts that there is a polynomial approximation that can be obtained from $m$ samples that attains the best $s$-term rate $s^{1/2-1/p}$, where $s = m/(c_0 L)$ scales like $m$ up to the polylogarithmic factor $L$.

Theorem \ref{t:main-res-map-alg} asserts the existence of a mapping that takes samples values as its input and produces the coefficients of a polynomial approximation attaining a desired error bound as its output. The mapping, as we see later, arises as a minimizer of a certain weighted $\ell^1$-minimization problem. Thus, it is not an algorithm in the sense of Definition \ref{d:AlgPolyApprox}. In the next two theorems we assert the existence of algorithms that attain the same error, plus additional algorithmic error terms.

\thm{
[Existence of an algorithm; algebraic case, finite dimensions]
\label{t:main-res-algo-alg}
Consider the setup of Theorem \ref{t:main-res-map-alg}. Then, for every $t \geq 1$, there exists an algorithm 
\bes{
\cA_t : \cU^m \times \bbC^{m \times K} \rightarrow \bbC^{N \times K}, 
}
in the sense of Definition \ref{d:AlgPolyApprox} such that the same property holds, except with \R{err-bound-map-alg_1} replaced by
\begin{equation}
\label{err-bound-algo-alg_2}
 \nmu{f - \hat{f}}_{L^2_{\varrho}(\cU;\cV)} \leq c_1 \cdot \left( \zeta +  \frac{1}{t}\right), \qquad \nmu{f - \hat{f}}_{L^{\infty}(\cU;\cV)}  \leq c_2 \cdot \sqrt{\frac{m}{L}} \cdot \left(\zeta   + \frac{1}{t}   \right) , 
\end{equation}
 where $c_1,c_2 \geq 1$ are as in \R{err-bound-map-alg_1} and $\zeta$ is as in \R{zeta-alg-def}.
The computational cost of the algorithm is bounded by
\be{
\label{comp-cost-alg-algo-fin}
c_3 \cdot \left [ m \cdot \Theta(n,d) \cdot d +  t \cdot \left ( m \cdot \Theta(n,d) \cdot K + (\Theta(n,d) + m) \cdot (F(\bm{G})+K) \right ) \cdot (\Theta(n,d))^{\alpha } \right ],
}
 where $n = \lceil m/L \rceil$ is as in Theorem \ref{t:main-res-map-alg}, $\Theta(n,d)$ is as in \R{Thetadef}, $\alpha$ is as in \R{main_alpha_def} and $c_3 > 0$  is a universal constant. 
}

The key element of this theorem is that the same error bound as in Theorem \ref{t:main-res-map-alg} is attained, up to an additional term. In particular, we have the three sources of errors (i)--(iii), plus the following:
\bulls{
\item[(iv)] \textit{The algorithmic error.} This is the error $E_{\textsf{alg}} = 1/t$ committed by the algorithm $\cA_t$ in approximately computing the mapping $\cM$ in Theorem \ref{t:main-res-map-alg}. It is given in terms of the parameter $t$, which also enters linearly into the computational cost estimate \R{comp-cost-alg-algo-fin}.
}
Unfortunately, the $1/t$ decay rate of the algorithmic error is slow. Thus, it may be computationally expensive to compute an approximation to within a desired error bound. Fortunately, as we now explain, it is possible to improve it to $\E^{-t}$ subject to an additional technical assumption.

\thm{
[Existence of an efficient algorithm; algebraic case, finite dimensions]
\label{t:main-res-effic-algo-alg}
Consider the setup of Theorem \ref{t:main-res-map-alg}. Then for every $t \geq 1$ and $\zeta' > 0$ there exists an algorithm 
\bes{
\cA_{t,\zeta'} : \cU^m \times \bbC^{m \times K} \rightarrow \bbC^{N \times K}, 
}
in the sense of Definition \ref{d:AlgPolyApprox} such that the same property holds whenever $\zeta' \geq \zeta$, except with \R{err-bound-map-alg_1} replaced by
\begin{equation}
\label{err-bound-effic-algo-alg_2}
 \nmu{f - \hat{f}}_{L^2_{\varrho}(\cU;\cV)} \leq c_1 \cdot \left( \zeta +  \zeta' + \E^{-t} \right), \qquad \nmu{f - \hat{f}}_{L^{\infty}(\cU;\cV)}  \leq c_2 \cdot \sqrt{\frac{m}{L}} \cdot\left( \zeta +  \zeta' + \E^{-t}  \right) , 
\end{equation}
where $c_1,c_2 \geq 1$ are as in \R{err-bound-map-alg_1} and $\zeta$ is as in \R{zeta-alg-def}.
The computational cost of the algorithm is bounded by
\bes{
c_3 \cdot \left [ m \cdot \Theta(n,d) \cdot d +  t \cdot \left ( m \cdot \Theta(n,d) \cdot K + (\Theta(n,d) + m) \cdot (F(\bm{G})+K) \right ) \cdot (\Theta(n,d))^{\alpha } \right ],
}
 where $n = \lceil m/L \rceil$ is as in Theorem \ref{t:main-res-map-alg}, $\Theta(n,d)$ is as in \R{Thetadef}, $\alpha$ is as in \R{main_alpha_def} and $c_3 > 0$  is a universal constant.
}

We refer to this as an `efficient' algorithm, since the parameter $t$ enters linearly in the computational cost but the algorithmic error scales like $\E^{-t}$. The main limitation of this result is that the algorithm parameter $\zeta'$ needs to be an upper bound for the true error bound $\zeta$ in order for \R{err-bound-effic-algo-alg_2} to hold. This is a technical assumption for the proof, and does not appear necessary in practice. We demonstrate this phenomenon through numerical experiment in \S \ref{s:num-exp}.

\subsubsection{Algebraic rates of convergence, infinite dimensions}

We now consider algebraic rates of convergence in the infinite-dimensional setting. The next three results should be compared against the best $s$-term approximation result, Theorem \ref{t:best_s_term_alg_inf}.

\thm{
[Existence of a mapping; algebraic case, infinite dimensions]
\label{t:main-res-map-alg_infty}
Let  $d = \infty$, $\{ \Psi_{\bm{\nu}} \}_{\bm{\nu} \in \bbN^d_0} \subset L^2_{\varrho}(\cU)$ be either the orthonormal Chebyshev or Legendre basis and $\{ \varphi_k \}^{K}_{k=1}$ be a basis for $\cV_h$. Then for every $m \geq 3$, $0 < \epsilon < 1$ and $K \geq 1$, there is a mapping
\bes{
\cM : \cU^m \times \bbC^{m \times K} \rightarrow \bbC^{N \times K}, 
}
where  $N =\Theta(n ,d)$ is as in \R{Thetadef} with $n = \lceil m / L \rceil$, where $L = L(m,d,\epsilon)$ is as in \R{Ldef},
with the following property. Let $\varepsilon > 0$, $0 < p < 1$ and $\bm{b} \in \ell^p(\bbN)$, $\bm{b} > \bm{0}$, be monotonically nonincreasing. Let $f \in \cB(\bm{b},\varepsilon)$, draw $\bm{y}_1,\ldots,\bm{y}_m$ randomly and independently according to $\varrho$ and let $(d_{ik})^{m,K}_{i,k=1} \in \bbC^{m \times K}$ be as in \R{data_we_get} for arbitrary noise terms $\bm{n} = (n_i)^{m}_{i=1} \in \cV$. Let $(\hat{c}_{jk}) = \cM((\bm{y}_i),(d_{ik}))$ and define the approximation $\hat{f}$ as in \R{fhatdef} based on the index set $\Lambda = \Lambda^{\mathsf{HCI}}_{n}$. Then the following holds with probability at least $1-\epsilon$. The error satisfies
\ea{
\label{err-bound-map-alg-inf_1}
\nmu{f - \hat{f}}_{L^2_{\varrho}(\cU;\cV)} \leq c_1 \cdot \zeta, \qquad \nmu{f - \hat{f}}_{L^{\infty}(\cU;\cV)}  \leq c_2 \cdot  \sqrt{\frac{m}{L}} \cdot \zeta,
}
where
\be{
\label{zeta-alg-inf-def}
\zeta := C \cdot \left ( \frac{m}{c_0 L} \right )^{1/2-1/p}  +  \frac{\nm{\bm{n}}_{2;\cV} }{\sqrt{m}} +    \nmu{f - \cP_h(f)}_{L^{\infty}(\cU ; \cV)},
}
$c_0,c_1,c_2 \geq 1$ are universal constants
and 
$C = C(\bm{b},\varepsilon,p)$ depends on $\bm{b}$, $\varepsilon$ and $p$ only.
}

\thm{
[Existence of an algorithm; algebraic case, infinite dimensions]
\label{t:main-res-algo-alg_infty}
Consider the setup of Theorem \ref{t:main-res-map-alg_infty}. Then, for every $t \geq 1$, there exists an algorithm 
\bes{
\cA_t : \cU^m \times \bbC^{m \times K} \rightarrow \bbC^{N \times K}, 
}
in the sense of Definition \ref{d:AlgPolyApprox} such that the same property holds, except with \R{err-bound-map-alg-inf_1} replaced by
\begin{equation}
\label{err-bound-algo-alg-inf_2}
 \nmu{f - \hat{f}}_{L^2_{\varrho}(\cU;\cV)} \leq c_1 \cdot \left( \zeta +  \frac{1}{t}\right), \qquad \nmu{f - \hat{f}}_{L^{\infty}(\cU;\cV)}  \leq c_2 \cdot \sqrt{\frac{m}{L}} \cdot \left(\zeta   + \frac{1}{t}   \right) , 
\end{equation}
 where $c_1,c_2 \geq 1$ are as in \R{err-bound-map-alg-inf_1} and $\zeta$ is as in \R{zeta-alg-inf-def}.
The computational cost of the algorithm is bounded by
\bes{
c_3 \cdot \left [ m \cdot \Theta(n,\infty) \cdot n +  t \cdot \left ( m \cdot \Theta(n,\infty) \cdot K + (\Theta(n,\infty) + m) \cdot (F(\bm{G})+K) \right ) \cdot (\Theta(n,\infty))^{\alpha } \right ],
}
 where $n = \lceil m/L \rceil$ is as in Theorem \ref{t:main-res-map-alg_infty}, $\Theta(n,\infty)$ is as in \R{Thetadef}, $\alpha$ is as in \R{main_alpha_def} and $c_3 > 0$  is a universal constant.
}

In finite dimensions, the computational cost estimate \R{comp-cost-alg-algo-fin} is somewhat difficult to interpret, since its behaviour depends on the relative sizes of $m$ and $d$. Fortunately, in infinite dimensions we can give a more informative assessment. Suppose, for simplicity, that $K$ is fixed (for example, $K = 1$ in the case of a scalar-valued function approximation problem). Then the computational cost is bounded by
\bes{
c \cdot m \cdot \Theta(n,\infty) \cdot n  + c_K \cdot t \cdot m \cdot \Theta(n,\infty)^{\alpha+1},
}
where $c > 0$ is a universal constant $c_K > 0$ is a constant depending on $K$ only.
Recall from \R{Thetadef} that $ \Theta(n,\infty) = | \Lambda^{\mathsf{HCI}}_{n} | = | \Lambda^{\mathsf{HC}}_{n,n} |$. Now, when $d = n$ and $n$ is sufficiently large, the minimum in \R{N_bound} is attained by the second term $\E n^{2 + \log(n)/\log(2)} $. Substituting this into the above expression and recalling that $n = \lceil m / L \rceil$, where $L = L(m,\infty,\epsilon)$, we deduce that the computational cost is bounded by
\bes{
c_K \cdot  t \cdot m \cdot g(m)^{(\alpha + 1)\log (4g(m))/ \log(2) },\qquad g(m) : = \left \lceil \frac{m}{\log(m) \cdot ( \log^3(m) + \log(\epsilon^{-1})) } \right \rceil. 
}
Since $m \geq 3$ by assumption, we have $\log(m) \geq 1$ and therefore $g(m) \leq m$. Hence, this admits the slightly looser upper bound
\bes{
 c_K \cdot t \cdot m^{1+(\alpha+1) \log (4m)/\log(2)}.
}
We conclude that the computational cost (for fixed $K$ and $t$) is \textit{subexponential} in $m$. Further, if we choose $t = m^{1/p-1/2}$ in accordance with the algebraically-decaying term in \R{zeta-alg-inf-def}, then we conclude the following: \textit{it is possible to approximate a holomorphic function of infinitely-many variables with error decaying algebraically fast in $m$ via an algorithm whose computational cost is subexponential in $m$.} Whether this can be reduced to an algebraic cost is an open problem.

\thm{
[Existence of an efficient algorithm; algebraic case, infinite dimensions]
\label{t:main-res-effic-algo-alg_infty}
Consider the setup of Theorem \ref{t:main-res-map-alg_infty}. Then, for every $t \geq 1$ and $\zeta' > 0$ there exists an algorithm
\bes{
\cA_{t,\zeta'} : \cU^m \times \bbC^{m \times K} \rightarrow \bbC^{N \times K}, 
}
in the sense of Definition \ref{d:AlgPolyApprox} such that the same property holds whenever $\zeta' \geq \zeta$, except with \eqref{err-bound-map-alg-inf_1} replaced by
\begin{equation}
 \nmu{f - \hat{f}}_{L^2_{\varrho}(\cU;\cV)} \leq c_1 \cdot \left( \zeta +  \zeta' + \E^{-t} \right), \qquad \nmu{f - \hat{f}}_{L^{\infty}(\cU;\cV)}  \leq c_2 \cdot \sqrt{\frac{m}{L}} \cdot\left( \zeta +  \zeta' + \E^{-t}  \right) , 
\end{equation}
 where $c_1,c_2 \geq 1$ are as in  \R{err-bound-map-alg-inf_1} and and $\zeta \leq \zeta'$ is as in \R{zeta-alg-inf-def}.
The computational cost of the algorithm is bounded by
\bes{
c_3 \cdot \left [ m \cdot \Theta(n,\infty) \cdot n +  t \cdot \left ( m \cdot \Theta(n,\infty) \cdot K + (\Theta(n,\infty) + m) \cdot (F(\bm{G})+K) \right ) \cdot (\Theta(n,\infty))^{\alpha } \right ],
}
 where $n = \lceil m/L \rceil$ is as in Theorem \ref{t:main-res-map-alg_infty}, $\Theta(n,\infty)$ is as in \R{Thetadef}, $\alpha$ is as in \R{main_alpha_def} and $c_3 > 0$  is a universal constant.
}

\subsubsection{Exponential rates of convergence, finite dimensions}\label{ss:main-res-exp}

Finally, we consider exponential rates of convergence in finite dimensions. The following results should be compared against Theorem \ref{t:best_s_term_exp_1}.

\thm{
[Existence of a mapping; exponential case, finite dimensions]
\label{t:main-res-map-alg_exp}
Let $d \in \bbN$, $\{ \Psi_{\bm{\nu}} \}_{\bm{\nu} \in \bbN^d_0} \subset L^2_{\varrho}(\cU)$ be either the orthonormal Chebyshev or Legendre basis and $\{ \varphi_k \}^{K}_{k=1}$ be a basis for $\cV_h$. Then for every $m \geq 3$, $0 < \epsilon < 1$ and $K \geq 1$, there is a mapping
\bes{
\cM : \cU^m \times \bbC^{m \times K} \rightarrow \bbC^{N \times K}, 
}
where  $N =\Theta(n,d)$ is as in \R{Thetadef} with  
\be{
\label{n-exp-case}
n = \begin{cases} \lceil \sqrt{m/L} \rceil& \mbox{Legendre,} \\ \lceil m/(2^d L)\rceil & \mbox{Chebyshev,} \end{cases} 
}
 and $L$ as in \R{Ldef}, with the following property. Draw $\bm{y}_1,\ldots,\bm{y}_m$ randomly and independently according to $\varrho$. Then, with probability at least $1-\epsilon$, the following holds. Let $f \in \cB(\bm{\rho})$ for arbitrary $\bm{\rho} > \bm{1}$, $(d_{ik})^{m,K}_{i,k=1} \in \bbC^{m \times K}$ be as in \R{data_we_get} for arbitrary noise terms $\bm{n} = (n_i)^{m}_{i=1} \in \cV$, $(\hat{c}_{jk})^{N,K}_{j,k=1} = \cM((\bm{y}_i)^{m}_{i=1},(d_{ik})^{m,k}_{i,k=1})$ and define the approximation $\hat{f}$ as in \R{fhatdef} based on the index set $\Lambda = \Lambda^{\mathsf{HC}}_{n,d}$. Then the error satisfies
\be{
\label{err-bound-map-alg_exp}
\nmu{f - \hat{f}}_{L^2_{\varrho}(\cU;\cV)} \leq  c_1 \cdot \zeta  ,\qquad 
\nmu{f - \hat{f}}_{L^{\infty}(\cU;\cV)}  \leq c_2  \cdot \sqrt{\frac{m}{L}} \cdot \zeta,
}
for any
\bes{
0 < \gamma < (d+1)^{-1} \left ( d! \prod^{d}_{j=1} \log(\rho_j) \right )^{1/d},
}
where
\be{
\label{zeta-alg-def_exp}
\zeta := C \cdot \left \{ \begin{array}{cc} \exp\left (-\frac{\gamma}{2} \left ( \frac{m}{c_0 L} \right )^{\frac1d} \right ) & \mbox{Chebyshev} \\ \exp \left (-\gamma \left (\frac{m}{c_0 L} \right )^{\frac{1}{2d}} \right ) & \mbox{Legendre}  \end{array} \right . +  \frac{\nm{\bm{n}}_{2;\cV} }{ \sqrt{m}} +  \nmu{f - \cP_h(f)}_{L^{\infty}(\cU ; \cV)},
}
$c_0,c_1,c_2 \geq 1$ are universal  constants and $C = C(d,\gamma,\bm{\rho})$ depends on $d$, $\gamma$ and $\bm{\rho}$ only.
}

\thm{
[Existence of an algorithm; exponential case, finite dimensions]
\label{t:main-res-algo-alg_exp}
Consider the setup of Theorem \ref{t:main-res-map-alg_exp}. Then, for every $t \geq 1$, there exists an algorithm
\bes{
\cA_t : \cU^m \times \bbC^{m \times K} \rightarrow \bbC^{N \times K}, 
}
in the sense of Definition \ref{d:AlgPolyApprox} such that the same property holds, except with \R{err-bound-map-alg_exp} replaced by
\begin{equation}
\label{err-bound-algo-alg_2_inf}
\nmu{f - \hat{f}}_{L^2_{\varrho}(\cU;\cV)} \leq c_1 \cdot \left( \zeta +  \frac{1}{t} \right), \qquad \nmu{f - \hat{f}}_{L^{\infty}(\cU;\cV)}  \leq c_2\cdot \sqrt{\frac{m}{L}} \cdot \left(\zeta   + \frac{1}{t}  \right), 
\end{equation}
 where $c_1,c_2 \geq 1$ are as in \R{err-bound-map-alg_exp} and $\zeta$ is as in \R{zeta-alg-def_exp}. 
 The computational cost of the algorithm is bounded by
\bes{
c_3 \cdot \left [ m \cdot \Theta(n,d) \cdot n +  t \cdot \left ( m \cdot \Theta(n,d) \cdot K + (\Theta(n,d) + m) \cdot (F(\bm{G})+K) \right ) \cdot (\Theta(n,d))^{\alpha } \right ],
}
 where $n $ is as in \R{n-exp-case}, $\Theta(n,d)$ is as in \R{Thetadef}, $\alpha$ is as in \R{main_alpha_def} and $c_3 > 0$  is a universal constant.
}

\thm{
[Existence of an efficient algorithm; exponential case, finite dimensions]
\label{t:main-res-effic-algo-alg_exp}
Consider the setup of Theorem \ref{t:main-res-map-alg_exp}. Suppose that there is a known upper bound $\zeta' \geq \zeta$, where $\zeta$ is as in \R{zeta-alg-def_exp}. Then, for every $t \geq 1$ and $\zeta' >0$ there exists an algorithm
\bes{
\cA_{t,\zeta'} : \cU^m \times \bbC^{m \times K} \rightarrow \bbC^{N \times K}, 
}
in the sense of Definition \ref{d:AlgPolyApprox}  for which the same property holds whenever $\zeta' \geq \zeta$, except with \eqref{err-bound-map-alg_exp} replaced by
\begin{equation}
\label{err-bound-algo-alg_3_exp}
\nmu{f - \hat{f}}_{L^2_{\varrho}(\cU;\cV)} \leq c_1 \cdot \left(\zeta + \zeta' +  \E^{-t} \right) , \qquad
\nmu{f - \hat{f}}_{L^{\infty}(\cU;\cV)}  \leq c_2 \cdot  \sqrt{\frac{m}{L}}\left (\zeta+ \zeta'  +   \E^{-t} \right ),
\end{equation}
where  $c_1,c_2 \geq 1$ are as in \R{err-bound-map-alg_exp} and $\zeta$ is as in \R{zeta-alg-def_exp}. The computational cost of the algorithm is bounded by
\bes{
c_3 \cdot \left [ m \cdot \Theta(n,d) \cdot n +  t \cdot \left ( m \cdot \Theta(n,d) \cdot K + (\Theta(n,d) + m) \cdot (F(\bm{G})+K) \right ) \cdot (\Theta(n,d))^{\alpha } \right ],
}
 where $n $ is as in \R{n-exp-case}, $\Theta(n,d)$ is as in \R{Thetadef}, $\alpha$ is as in \R{main_alpha_def} and $c_3 > 0$  is a universal constant.
}

As before,  suppose  that $K$ is fixed and, since we consider exponential rates, that $d$ is also fixed. Then, using the third estimate in \R{N_bound}, we deduce that the computational cost  of this algorithm is bounded by
\bes{
c_{K,d} \cdot ( m \cdot  n^2  \cdot (\log(n))^{d-1}+ t \cdot m \cdot (n \cdot (\log(n))^{d-1})^{\alpha+1}).
}
Using the crude bound $n \leq m$, we deduce the bound
\bes{
c_{K,d} \cdot \left(  t \cdot m^{\alpha+2} (\log(m))^{(d-1)(\alpha+1)} \right).
}
Thus, for fixed $t$, the computational cost is polynomial in $m$ as $m \rightarrow \infty$. In particular, with the efficient algorithm of Theorem \ref{t:main-res-effic-algo-alg_exp} (subject to the caveat that an upper bound for the error is known) we deduce the following: \textit{in fixed dimension $d$, it is possible to approximate a holomorphic function with error decaying exponentially fast in $m$ via an algorithm whose computational cost is polynomial in $m$.} Whether the polynomial growth rate described above is sharp is an open problem.

\rem{
\label{rem:uniform-vs-nonuniform}
There is a subtle difference between the algebraic and exponential results. The former are \textit{nonuniform} in the sense that a single draw of the sample points $\bm{y}_1,\ldots,\bm{y}_m$ is sufficient for recovery of a fixed function $f$ with high probability up to the specified error bound. The latter are \textit{uniform}, since a single draw of  the sample points $\bm{y}_1,\ldots,\bm{y}_m$ is sufficient for recovery of any function with high probability up to the specified error bound. The reason for this difference stems from bounding a \textit{discrete} error term \R{discrete-error}, which is a random variable dependent on $f$ and the sample points. In the algebraic case, in order to obtain the desired algebraic exponent $1/2-1/p$ we bound this term with high probability for each fixed $f$. See Step 4 of the proof of Theorem \ref{t:main-res-map-alg_inex}. This renders the ensuing result nonuniform. Conversely, in the exponential case (where the appearance of small algebraic factors is not a concern, since they can be absorbed into the exponentially-decaying term) we bound this term with probability one for \textit{any} $f$. See Step 4 of the proof of Theorem \ref{t:main-res-map-alg_inex_exponential}. Note that one could also derive uniform guarantees in the algebraic case by considering a fixed value of $p$ and letting $\cM$ and $\cA$ depend on $p$, or by considering a restricted range $0 < p \leq p^* < 1$. Both strategies involve a larger value of $n$, with its size depending on $p$ or $p^*$. See \cite[\S 7.6.2]{adcock2021sparse} for further discussion.
}

\section{Construction of the algorithms}\label{s:constr-algs}

In this section, we describe the construction of the algorithms asserted in our main results. These are based on techniques from compressed sensing \cite{adcock2021compressive,adcock2021sparse,foucart2013mathematical} on the premise that the polynomial coefficients of a holomorphic function are approximately sparse. There are several main differences between standard compressed sensing and what we develop below. First, following \cite{chkifa2018polynomial,adcock2018infinite,adcock2018compressed,adcock2021sparse,rauhut2016interpolation,rauhut2017compressive}, we work in a weighted setting in order to promote sparsity in lower or anchored sets (recall \S \ref{s:lower-anchored}). Second, following \cite{dexter2019mixed}, we work with Hilbert-valued vectors, whose entries take values in the Hilbert space $\cV$. Finally, so as to avoid unrealistic assumptions on the functions being approximated, we use consider \textit{noise-blind} decoders, as in \cite{adcock2019correcting}. See also Remark \ref{r:noise-blind}. 

\subsection{Recovery via Hilbert-valued, weighted $\ell^1$-minimization}\label{SRecovery}

We first require some additional notation. Given $N \in \bbN$ we let $\cV^{N}$ be the vector space of Hilbert-valued vectors of length $N$, i.e.\ $\bm{\v} = ( v_i)^{N}_{i=1}$ where $v_i \in \cV$, $i=1,\ldots,N$. Next, given $\Lambda \subseteq \cF$ and a vector of positive weights $\bm{w} = (w_{\bm{\nu}})_{\bm{\nu} \in \Lambda}$, where $\bm{w} > \bm{0}$, we define the weighted $\ell_w^p(\Lambda;\cV)$ space, $0 < p \leq 2$, as the set of $\cV$-valued sequences $\bm{v} = (v_{\bm{\nu}})_{\bm{\nu} \in \Lambda}$ for which  
\bes{
\nmu{\bm{v}}_{p,\w;\cV} : = \left ( \sum_{\bnu \in \Lambda} w^{2-p}_{\bnu} \nm{v_{\bnu}}^p_{\cV} \right )^{1/p} < \infty.
}
Notice that $\ell^2_{\bm{w}}(\Lambda ; \cV)$ coincides with the unweighted space $\ell^2(\Lambda ; \cV)$. 

Now, let $\Lambda \subset \cF$ be a finite multi-index set of size $|\Lambda | = N$ and consider the ordering $\Lambda = \{ \bm{\nu}_1,\ldots,\bm{\nu}_N \}$. Note that we will, in practice, choose either $\Lambda = \Lambda^{\mathsf{HC}}_{n,d}$ when $d < \infty$ or $\Lambda = \Lambda^{\mathsf{HCI}}_n$ when $d = \infty$, where the order $n$ is as described in the corresponding theorem (Theorems \ref{t:main-res-map-alg}--\ref{t:main-res-effic-algo-alg_exp}). With this in mind,  given $f \in L^2_{\varrho}(\cU ; \cV)$, define
\be{
\label{f_exp_trunc}
f_{\Lambda} = \sum_{\bm{\nu} \in \Lambda} c_{\bm{\nu}} \Psi_{\bm{\nu}}
}
as the truncated expansion of $f$ based on the index set $\Lambda$ and
\be{
\label{f_coeff_trunc}
\bm{c}_{\Lambda} = (c_{\bm{\nu}_j})^{N}_{j=1} \in \cV^N
}
as the finite vector of coefficients of $f$ with indices in $\Lambda$. As explained in \S \ref{ss:prob-stat}, our objective is, in effect, to approximate these coefficients.

We do this as follows. Given $\bm{y}_1,\ldots,\bm{y}_m \in \cU$, we define the normalized measurement matrix
\begin{equation}\label{def-measMatrix}
\bm{A} =\left(\frac{\Psi_{\bnu_j} (\y_i)}{\sqrt{m}} \right)^{m,N}_{i,j=1} \in \bbC^{m \times N}
\end{equation}
and the normalized measurement and error vectors
\be{
\label{def-measVec}
\bm{b} = \frac{1}{\sqrt{m}} \left ( f(\bm{y}_i) + n_i \right )^{m}_{i=1} \in \cV^m_h,\qquad \bm{e} = \frac{1}{\sqrt{m}} (n_i)^{m}_{i=1} \in \cV^m.
}
Notice that any $m \times N$ matrix $\bm{A} = (a_{ij})^{m,N}_{i,j=1}$ extends to a bounded linear operator $\cV^N \rightarrow \cV^m$ (or $\cV^N_h \rightarrow \cV^m_h$) in the obvious way, i.e.,
\be{
\label{def-inducedBLO}
\bm{x} = (x_i)^{N}_{i=1} \in \cV^N \mapsto \bm{A} \bm{x} = \left ( \sum^{N}_{j=1} a_{ij} x_j \right )^{m}_{i=1} \in \cV^m.
}
For ease of notation, we make no distinction between the matrix $\bm{A} \in \bbC^{m \times N}$ and the linear operator $\bm{A} \in \cB(\cV^N,\cV^m)$ (or $\bm{A} \in \cB(\cV^N_h,\cV^m_h)$) in what follows. Using this, we obtain
\bes{
\bm{A} \bm{c}_{\Lambda} = \frac{1}{\sqrt{m}} \left ( f_{\Lambda}(\bm{y}_i) \right )^{m}_{i=1} =  \frac{1}{\sqrt{m}} (f(\bm{y}_i))^{m}_{i=1} - \frac{1}{\sqrt{m}} \left ( f(\bm{y}_i) - f_{\Lambda}(\bm{y}_i) \right )^{m}_{i=1},
}
and therefore
\be{
\label{linsys_for_cLambda}
\bm{A} \bm{c}_{\Lambda} + \bm{e} + \bm{e'} = \bm{b},
}
where
\bes{
\bm{e'} = \frac{1}{\sqrt{m}}  \left ( f(\bm{y}_i) - f_{\Lambda}(\bm{y}_i) \right )^{m}_{i=1}.
}
We have now formulated the recovery of $\bm{c}_{\Lambda}$ as the solution of a noisy linear system \R{linsys_for_cLambda}, where the noise term $\bm{e} + \bm{e}'$ encompasses both the noise $\bm{e} = (n_i)^{m}_{i=1} / \sqrt{m}$ in the sample values and the error $\bm{e}'$ due to the truncation \R{f_exp_trunc} of the infinite expansion \R{f_exp} via the index set $\Lambda$. 

Due to the discussion in \S \ref{ss:seq-space}--\ref{s:lower-anchored}, we expect the coefficients $\bm{c}_{\Lambda}$ to not only be approximately sparse, but also well approximated by a subset of $s$ coefficients whose indices define a lower or anchored set. In classical compressed sensing, one exploits sparse structure via minimizing an $\ell^1$-norm. To exploit sparse and lower structure, we follow ideas of \cite{chkifa2018polynomial,adcock2018infinite,adcock2018compressed,adcock2021sparse} and use a weighted $\ell^1$-norm penalty. Specifically, we now compute an approximate solution via the Hilbert-valued, \textit{weighted Square-Root LASSO (SR-LASSO)} optimization problem
\be{
\label{wsr-LASSO}
\min_{\bm{z} \in \cV^N_h} \cG(\z),\qquad \cG(\z) : = \lambda \nm{\bm{z}}_{1,\w;\cV} + \nmu{\bm{A} \bm{z} - \bm{b} }_{2;\cV} .
}
Here $\lambda>0$ is a tuning hyperparameter. 

\rem{
\label{r:noise-blind}
As an alternative to solve this Hilbert-valued compressed sensing problem, we could use a formulation based on a constrained basis pursuit or unconstrained LASSO problem. However, we consider the  SR-LASSO problem \eqref{wsr-LASSO}  instead. While other approaches are arguably more common, based on \cite{adcock2019correcting} the SR-LASSO has the desirable property that the optimal values of its hyperparameter $\lambda$ is independent of the noise term (in this case $\bm{e} + \bm{e}'$). This is not the case for other formulations, whose hyperparameters need to be chosen in terms of the (unknown) magnitude of the noise in order to ensure good theoretical and practical performance (see, e.g., \cite[Chpt.\ 6]{adcock2021compressive}). This is particularly problematic in the setting of function approximation, where such terms are function dependent (for instance, the term $\bm{e}'$ depends on the expansion tail $f - f_{\Lambda}$) and therefore generally unknown. See \cite{adcock2019correcting} and \cite[\S 6.6]{adcock2021sparse} for further discussion.
}

Notice that \R{wsr-LASSO} is solved over $\cV^N_h$ not $\cV^N$, since the latter would not be numerically solvable in general. As we see below, it can be reformulated an optimization problem over $\bbC^{N \times K}$, where $K = \dim(\cV_h)$. However, since the true coefficients of $f$ are elements of $\cV$ and not $\cV_h$, this discretization inevitably results in an additional error, 
which must also be accounted for in the analysis. This leads precisely to the physical discretization error (term (iii) in \S \ref{ss:main-res-alg-fin}).

Finally, we now also specify the weights. Following \cite{adcock2018infinite,chkifa2018polynomial,adcock2018compressed} (see also \cite[Rem.\ 2.14]{adcock2021sparse}), a good choice of weights (for promoting lower or anchored structure) is given by the so-called \textit{intrinsic weights}
\be{
\label{weights_def}
\bm{w} = \bm{u} = (u_{\bm{\nu}})_{\bm{\nu} \in \Lambda},\quad u_{\bm{\nu}} = \nm{\Psi_{\bm{\nu}}}_{L^{\infty}(\cU)}, \ \bnu \in \Lambda. 
}
In particular, for Chebyshev and Legendre polynomials these are given explicitly by
\begin{equation*}
u_{\bm{\nu}} =  \nm{\Psi_{\bm{\nu}}}_{L^{\infty}(\cU)} = \begin{cases} 
      \prod^{d}_{j=1} \sqrt{2 \nu_j +1}  , & \mathrm{Legendre,} \\
      2^{\nm{\bnu}_0/2} ,& \mathrm{Chebyshev,} \\
   \end{cases}
\end{equation*}
where $\nmu{\bm{\nu}}_0:= |\mathrm{supp}(\nu)|$. Typically, we index these weights over the multi-indices $\bm{\nu} \in \Lambda$. However, we will, for convenience, write $w_{i}$ instead of $w_{\bm{\nu}_i}$ in what follows, where, as above $\{ \bm{\nu}_1,\ldots,\bm{\nu}_N \}$ is an ordering of $\Lambda$.

\subsection{Reformulation as a matrix recovery problem and the mappings in Theorems \ref{t:main-res-map-alg}, \ref{t:main-res-map-alg_infty} and \ref{t:main-res-map-alg_exp}}
We now describe the mappings whose existence is asserted in Theorems \ref{t:main-res-map-alg}, \ref{t:main-res-map-alg_infty} and \ref{t:main-res-map-alg_exp}. These maps all arise via exact solutions of weighted SR-LASSO optimization problems. However, since \R{wsr-LASSO} yields a vector in $\cV^{N}_{h}$ and the mappings should yield outputs in $\bbC^{N \times K}$, we first need to reformulate \R{wsr-LASSO} using the basis $\{ \varphi_i \}^{K}_{i=1}$ for $\cV_h$.

Notice first that any vector of coefficients $\bm{c} = (c_{\bm{\nu}_i})^{N}_{i=1} \in \cV^N_h$ is equivalent to a matrix of coefficients
\bes{
\bm{C} = (c_{ik})^{N,K}_{i,k = 1} \in \bbC^{N \times K},
}
via the relation
\bes{
c_{\bm{\nu}_i} = \sum^{K}_{k=1} c_{ik} \varphi_k,\quad i \in [N].
}
Next, observe that if $g = \sum^{K}_{k=1} d_k \varphi_k \in \cV_h$ then
\be{
\label{write_Vnorm_as_G}
\nm{g}_{\cV} = \nm{\bm{d}}_{\bm{G}} = \sqrt{\bm{d}^{*} \bm{G} \bm{d}},
}
where $\bm{d} = (d_k)^{K}_{k=1} \in \bbC^K$ and $\bm{G} \in \bbC^{K \times K}$ is the Gram matrix for $\{ \varphi_k \}^{K}_{k=1}$, given by \R{Gram-matrix-def}. Since $\bm{G}$ is positive definite, it has a unique positive definite square root matrix $\bm{G}^{1/2}$. Hence we may write
\bes{
\nm{g}_{\cV} = \nmu{\bm{G}^{1/2} \bm{d}}_2.
}
We now use some additional notation. Given $1 \leq p \leq \infty$ and $1 \leq q \leq 2$, we define the weighted $\ell^{p,q}_{\bm{w}}$-norm of a matrix $\bm{C} = (c_{ik})^{N,K}_{i,k = 1} \in \bbC^{N \times K}$ as
\bes{
\nm{\bm{C}}_{p,q,\bm{w}} = \left ( \sum^{N}_{i=1} w^{2-p}_i \left ( \sum^{K}_{k=1} |c_{ik} |^q \right )^{p/q} \right )^{1/p}.
}
Note that this is precisely the weighted $\ell^p_{\bm{w}}$-norm of the vector of $( \nm{\bm{c}_i}_q )^{N}_{i=1}$, where $\bm{c}_i = (c_{ik})^{K}_{k=1} \in \bbC^K$ is the $i$th row of $\bm{C}$. Further, if $p = q = 2$, then this is just the unweighted $\ell^{2,2}$-norm of a matrix (which is simply its Frobenius norm). In this case, we typically write $\nm{\cdot}_{2,2}$.

Now let $\bm{z} \in \cV^N_h$ be arbitrary, $\bm{Z} \in \bbC^{N \times K}$ be the corresponding matrix and $\bm{z}_i \in \bbC^K$ be the $i$th row of $\bm{Z}$. Then
\bes{
\nm{\bm{z}}_{1,\bm{w} ; \cV} = \sum^{N}_{i=1} w_{i} \nm{z_{\bm{\nu}_i}}_{\cV} = \sum^{N}_{i=1} w_i \nmu{\bm{G}^{1/2}\bm{z}_i}_{2} = \nmu{\bm{Z} \bm{G}^{1/2}}_{2,1,\bm{w}}.
}
Similarly, let $\bm{A} = (a_{ij})^{m,N}_{i,j=1} \in \bbC^{m \times N}$ and $\bm{b} = (b_i)^{m}_{i=1} \in \cV^m_h$ be as in \R{def-measMatrix} and \R{def-measVec}, respectively, and let $\bm{B} \in \bbC^{m \times K}$ be the matrix corresponding to $\bm{b}$. Then
\bes{
\nmu{\bm{A} \bm{z} - \bm{b}}^2_{2;\cV} = \sum^{m}_{i=1} \nm{\sum^{N}_{j=1} a_{ij} z_{\bm{\nu}_i} - b_i }^2_{\cV} = \nmu{(\bm{A} \bm{Z} - \bm{B}) \bm{G}^{1/2}}^2_{2,2}.
}
Therefore, we now consider the minimization problem
\be{
\label{wsr-LASSO-coeff}
\min_{\bm{Z} \in \bbC^{N \times K}} \left \{ \lambda \nm{\bm{Z}}_{2,1,\bm{w}} + \nmu{(\bm{A} \bm{Z} - \bm{B}) \bm{G}^{1/2}}_{2,2} \right \}.
}
This is equivalent to \R{wsr-LASSO} in the following sense. A vector $\hat{\bm{c}} = (\hat{c}_{\bm{\nu}_i})^{N}_{i=1} \in \cV^N_{h}$ is a minimizer of \R{wsr-LASSO} if and only if the matrix $\widehat{\bm{C}} = (\hat{c}_{ik})^{N,K}_{i,k = 1} \in \bbC^{N \times K}$ with entries defined by the relation
\bes{
\hat{c}_{\bm{\nu}_i} = \sum^{K}_{k=1} \hat{c}_{ik} \varphi_k,\quad i \in [N],
}
is a minimizer of \R{wsr-LASSO-coeff}. 

With this in hand, we are now ready to define the mappings used in Theorems \ref{t:main-res-map-alg}, \ref{t:main-res-map-alg_infty} and \ref{t:main-res-map-alg_exp}. These are described in Table \ref{tab:the-mappings}.
Note that these are indeed well-defined mappings, since the minimizer of \R{wsr-LASSO-coeff} with smallest $\ell^{2,2}$-norm is unique (this follows from the facts that \R{wsr-LASSO-coeff} is a convex problem, therefore its set of minimizers is a convex set, and the function $\bm{Z} \mapsto \nm{\bm{Z}}^2_{2,2}$ is strongly convex). This particular choice is arbitrary, and is made solely so as to have a well-defined mapping. It is of no consequence whatsoever. Indeed, the various error bounds we prove later hold for any minimizer of \R{wsr-LASSO-coeff}.

\begin{table}
\benbox{
\begin{itemize}
\item Let $m$, $\epsilon$ and $n$ be as given in the particular theorem and set $\Lambda = \Lambda^{\mathsf{HC}}_{n,d}$ (Theorem \ref{t:main-res-map-alg} and \ref{t:main-res-map-alg_exp}) or $\Lambda = \Lambda^{\mathsf{HCI}}_{n}$ (Theorem \ref{t:main-res-map-alg_infty}).

\item Set $\lambda = (4 \sqrt{m/L})^{-1}$, where $L = L(m,d,\epsilon)$ is as in \R{Ldef}.

\item Let $\bm{D} = (d_{ik})^{m,K}_{i,k=1} \in \bbC^{m \times K}$ and $\bm{Y} = (\bm{y}_i)^{m}_{i=1}$ be an input, as in \R{data_we_get}, and set $\bm{B} = \frac{1}{\sqrt{m}} \bm{D}$.

\item Let $\bm{G}$, $\bm{A}$ and $\bm{w}$ be as in \R{Gram-matrix-def}, \R{def-measMatrix} and \R{weights_def}, respectively.

\item Define the output $\widehat{\bm{C}} = \cM(\bm{Y},\bm{D})$ as the minimizer of \R{wsr-LASSO-coeff} with smallest $\ell^{2,2}$-norm.
\end{itemize}
}
\vspace{-1pc}

\caption{The mappings $\cM : \cU^m \times \bbC^{m \times K} \rightarrow \bbC^{N \times K}$ used in Theorems \ref{t:main-res-map-alg}, \ref{t:main-res-map-alg_infty} and \ref{t:main-res-map-alg_exp}}
\label{tab:the-mappings}
\end{table}

\subsection{The primal-dual iteration}

To derive the algorithms described in the other main theorems, we need methods for approximately solving the optimization problems \R{wsr-LASSO} and \R{wsr-LASSO-coeff}. We use the \textit{primal-dual iteration} \cite{ChambollePock2011} (also known as the Chambolle--Pock algorithm) to this end.
We first briefly describe the primal-dual iteration in the general case (see \cite{ChambollePock2011,ChambolleEtAl2016,ChamPockAN}, as well as \cite[\S 7.5]{adcock2021compressive}) for more detailed treatments), before specializing to the weighted SR-LASSO problem in the next subsection.

Let $(\cX, \ip{\cdot}{\cdot}_{\cX})$ and $(\cY,\ip{\cdot}{\cdot}_{\cY})$ be (complex) Hilbert spaces, $g: \cX \rightarrow \bbR\cup\lbrace \infty \rbrace$, $h: \cY \rightarrow \bbR\cup\lbrace \infty \rbrace$ be proper, lower semicontinuous and convex functions and $ A \in  \cB(\cX,\cY)$ be a bounded linear operator satisfying $\mathrm{dom}(h) \cap  A (\mathrm{dom}(g))\neq \emptyset$. The primal-dual iteration is a general method for solving the convex optimization problem
\begin{equation}\label{primal1}
\min_{x\in \cX} \left \{ g(x) + h(A(x)) \right \}.
\end{equation}
\GR{}
Under this setting the (Fenchel--Rockafeller) dual problem is
\begin{equation}\label{dual1}
\min_{\xi \in \cY} \left \{  g^*{A^*(\xi))} + h^*(-\xi)  \right \},
\end{equation}
where $g^*$ and $h^*$ are the convex conjugate functions of $g$ and $h$, respectively. Recall that, for a function $f: \cX \rightarrow \bbR \cup \{\infty\}$, its convex conjugate is defined by
\begin{equation}\label{conjugate}
f^*(z) = \sup_{x \in \cX} \left( \Re\ip{x}{z}_{\cV}-f(x)\right),\quad z \in \cX.
\end{equation}
The Lagrangian of \R{primal1} is defined by
\begin{equation}\label{Lag1}
\cL (x,\xi)= g(x)+  \Re\ip{A(x)}{\xi}_{\cY}-h^*(\xi),\qquad x \in \mathrm{dom}(g),\ \xi \in \mathrm{dom}(h^*),
\end{equation}
and $\cL(x,\xi)=\infty$ if $x\not\in \mathrm{dom}(g)$ or $\cL(x,\xi)=-\infty$ if $\xi \not\in \mathrm{dom}(h^*)$. This in turn leads to the saddle-point formulation of the problem
\bes{
\min_{x \in \cX} \max_{\xi \in \cY} \cL(x,\xi).
}
The primal-dual iteration seeks a solution $(\hat{x},\hat{\xi})$ of the saddle-point problem by solving the following fixed-point equation
\begin{equation}\label{GeneralPrimal-Dual}
\begin{split}
\hat{x} &= \mathrm{prox}_{\tau g}(\hat{x}-\tau A^*(\hat{\xi})),
\\
\hat{\xi} &= \mathrm{prox}_{\sigma h^*}(\hat{\xi}+\sigma A(\hat{x}) ),
\end{split}
\end{equation}
where $ \tau,\sigma>0$ are stepsize parameters and  $\mathrm{prox}$ is the proximal operator, which is defined by
\bes{
\mathrm{prox}_{f}(z)= \arg\min_{x \in \cX} \left \{ f(x)+\dfrac{1}{2}\nm{x-z}_{\cX}^2 \right \}, \qquad z \in \mathrm{dom}(f). 
}
To be precise, given initial values $(x^{(0)},\xi^{(0)}) \in \cX \times \cY$ the primal-dual iteration defines a sequence $\{ (x^{(n)},\xi^{(n)}) \}^{\infty}_{n=1} \subset \cX \times \cY$ as follows:
\be{
\label{PDI}
\begin{split}
x^{(n+1)} &= \mathrm{prox}_{\tau g}(x^{(n)} - \tau A^*(\xi^{(n)})),
\\
\xi^{(n+1)} &= \mathrm{prox}_{\sigma h^*} (\xi^{(n)} + \sigma A (2 x^{(n+1)}-x^{(n)})).
\end{split}
}

\subsection{The primal-dual iteration for the weighted SR-LASSO problem}\label{Sec:PDI-A}

 We now apply this scheme to \eqref{wsr-LASSO} and \R{wsr-LASSO-coeff}. We first describe an algorithm to approximately solve the Hilbert-valued problem \R{wsr-LASSO}, before using the equivalence between elements of $\cV^N_h$ and $\bbC^{N \times K}$ to obtain an algorithm for approximately solving \R{wsr-LASSO-coeff}. 

Consider \R{wsr-LASSO}. We define $\cX = (\cV^N_h,\ip{\cdot}{\cdot}_{2;\cV})$, $\cY = (\cV^m_h,\ip{\cdot}{\cdot}_{2;\cV})$ and $g: \cX \rightarrow \bbR\cup\lbrace \infty \rbrace$, $h: \cY \rightarrow \bbR\cup\lbrace \infty \rbrace$ as  the  proper, lower semicontinuous and convex functions 
\bes{
g(\bm{x}) = \lambda \nmu{\bm{x}}_{1,\w;\cV},\qquad h(\bm{y}) = \nmu{\bm{y} - \bm{b}}_{2;\cV},\qquad \x \in \cV^N_h,\ \y \in \cV^m_h.
} 
We first find the proximal maps of $g$ and $h^*$.
Using \eqref{conjugate}, we see that
\begin{equation*}
\begin{split}
h^*(\bm{\xi})  = \sup_{\v \in \cV^m_h} \left( \Re\ip{\v}{\bxi}_{\cV}-\nmu{\v - \bm{b}}_{2;\cV}\right)
=  \Re\ip{\bm{b}}{\bm{\xi}}_{\cV} +\sup_{\v \in \cV^m_h} \left( \Re\ip{\v}{\bxi}_{\cV}-\nmu{\v}_{2;\cV}\right), \quad \forall \bxi \in \cV^m_h.\\
\end{split}
\end{equation*}
From \cite[Ex. 13.3 \& 13.4]{BauschkeCombettes2017} it follows that
\begin{equation*}
(\nm{\cdot}_{\cV})^* = \delta_{B},\qquad B: = \{ \bm{\xi} \in \cV^m_h : \nm{\bm{\xi}}_{2;\cV} \leq 1 \},
\end{equation*}
where $\delta_{B}$ is the indicator function of the set $B$, taking value $\delta_{B}(\bm{\xi}) = 0$ when $\bm{\xi} \in B$ and $+ \infty$ otherwise.
Hence
\begin{equation}\label{eq-fstar}
h^*(\bm{\xi}) = \Re\ip{\bm{b}}{\bm{\xi}}_{\cV} + \delta_{B}(\bm{\xi}).
\end{equation}
Using this, we obtain  
\begin{equation*}
\begin{split}
\mathrm{prox}_{\sigma h^*}(\bm{\xi}) 
&= \arg\min_{\bm{z} \in \cV^m_h} \left\lbrace \sigma \delta_{B}(\bm{\z})+ \sigma \Re\ip{\bm{b}}{\bm{\z}}_{\cV} + \dfrac{1}{2} \nm{\z-\bm{\xi}}_{2;\cV}^2\right\rbrace \\
&= \arg\min_{\bm{z}: \nm{\z}_{2;\cV}\leq 1 } \left\lbrace \dfrac{1}{2} \nm{\z-(\bm{\xi}-\sigma\b)}_{2;\cV}^2\right\rbrace \\
&= \mathrm{proj}_{B}(\bm{\xi} - \sigma \bm{b}),\\ 
\end{split}
\end{equation*}
where $\mathrm{proj}_{B}$ is the projection onto $B$, which is given explicitly by
\begin{equation*}
\mathrm{proj}_{B}(\bm{\xi}) = \min \left \{ 1 , \frac{1}{\nm{\bm{\xi}}_{2;\cV} } \right \} \bm{\xi}.
\end{equation*}
On the other hand, applying the definition of the proximal operator to the function $\tau g$ with parameter  $\tau>0$, we deduce that
\bes{
\left ( \mathrm{prox}_{\tau g}(\bm{x}) \right )_{i} = \mathrm{prox}_{\tau w_i \lambda \nm{\cdot}_{\cV}}(x_i),\ i = 1,\ldots,N,\qquad \bm{x} = (x_i)^{N}_{i=1} \in \cV^N_h.
}
Moreover, a simple adaptation of \cite[Ex.\ 14.5]{BauschkeCombettes2017} with the $\nm{\cdot}_{\cV}$-norm gives
\begin{equation*}
 \mathrm{prox}_{\tau \nm{\cdot}_{\cV}}(x) = \max \lbrace \nm{x}_{\cV}-\tau, 0 \rbrace \dfrac{x}{\nm{x}_{\cV}},\qquad \forall x\in \cV_h\setminus \lbrace 0\rbrace.
\end{equation*}
Hence, 
\begin{equation*}
 \mathrm{prox}_{\tau g}(\bm{x})  = \left( \max \{ \nm{x_i}_{\cV} - \tau\lambda w_i , 0 \} \frac{x_i}{\nm{x_i}_{\cV}} \right)_{i=1}^N,\qquad \bm{x} = (x_i)^{N}_{i=1} \in \cV^N_h \setminus \lbrace \0\rbrace.
\end{equation*}
With this in hand, we are now ready to define the primal-dual iteration for \R{wsr-LASSO}. As we see later, the analysis of convergence for the primal-dual iteration is given in terms of the \textit{ergodic} sequence
\bes{
\bar{\bm{c}}^{(n)} = \frac1n \sum^{n}_{i=1} \bm{c}^{(i)},\qquad n = 1,2,\ldots,
}
where $\bm{c}^{(i)} \in \cV^N_h$ is the primal variable obtained at the $i$th step of the iteration.
Hence, we now include the computation of these sequences in the primal-dual iteration for the weighted SR-LASSO problem \R{wsr-LASSO}, and take this as the output. The resulting procedure is described in Algorithm  \ref{a:primal-dual-wSRLASSO}.

\begin{algorithm}[t]
\caption{\texttt{primal-dual-wSRLASSO} -- the primal-dual iteration for the weighted SR-LASSO problem \eqref{wsr-LASSO}}
\label{a:primal-dual-wSRLASSO}
\SetKwInOut{input}{inputs}
\SetKwInOut{initialize}{initialize}
\SetKwInOut{output}{output}
\input{measurement matrix $\bm{A} \in \bbC^{m \times N}$, measurements $\bm{b} \in \cV^N_h$, positive weights $\bm{w} = (w_i)^{N}_{i=1}$, parameter $\lambda>0$, stepsizes $\tau, \sigma>0$, maximum number of iterations $T \geq 1$, initial values $\bm{c}^{(0)} \in \cV^N_h$, $\bm{\xi}^{(0)} \in \cV^m_h$}
\output{$\bar{\bm{c}} = \text{\texttt{primal-dual-wSRLASSO}}(\bm{A},\bm{b},\bm{w},\lambda,\tau,\sigma,T,\bm{c}^{(0)}, \bm{\xi}^{(0)})$, an approximate minimizer of \eqref{wsr-LASSO}}
\initialize{$\bar{\bm{c}}^{(0)} = \bm{0} \in \cV^N_h$ }
\BlankLine

\For{$n = 0,1,\ldots T-1$}{
$ \bm{p} = (p_i)^{N}_{j=1}= \bm{c}^{(n)}-\tau \bm{A}^*\bm{\xi}^{(n)}$
\\
$\bm{c}^{(n+1)}  =\left( \max \{ \nm{p_{i}}_{\cV} - \tau\lambda w_i , 0 \} \frac{p_{i}}{\nm{p_{i}}_{\cV}} \right)_{i=1}^N$
\\
$ \bm{q}  =\bm{\xi}^{(n)} + \sigma \bm{A} (2 \bm{c}^{(n+1)}-\bm{c}^{(n)})-\sigma \b$
\\
$ \bm{\xi}^{(n+1)} = \min \left \{ 1, \frac{1}{\nm{\bm{q}}_{2;\cV}} \right \} \bm{q}$
\\
$\bar{\bm{c}}^{(n+1)} =  \frac{n}{n+1} \bar{\bm{c}}^{(n)} + \frac{1}{n+1} \bm{c}^{(n+1)} $
}
$\bar{\bm{c}} = \bar{\bm{c}}^{(T)}$
\end{algorithm}

Having done this, we next adapt Algorithm \ref{a:primal-dual-wSRLASSO} in the way mentioned previously to obtain an algorithm for \R{wsr-LASSO-coeff}. This is given in Algorithm \ref{a:primal-dual-wSRLASSO-coeff}.

\begin{algorithm}[t]
\caption{\texttt{primal-dual-wSRLASSO-C} -- the primal-dual iteration for the weighted SR-LASSO problem \eqref{wsr-LASSO-coeff}}
\label{a:primal-dual-wSRLASSO-coeff}
\SetKwInOut{input}{inputs}
\SetKwInOut{initialize}{initialize}
\SetKwInOut{output}{output}
\input{measurement matrix $\bm{A} \in \bbC^{m \times N}$, measurements $\bm{B} \in \bbC^{m \times K}$, positive weights $\bm{w} = (w_i)^{N}_{i=1}$, Gram matrix $\bm{G} \in \bbC^{K \times K}$, parameter $\lambda>0$, stepsizes $\tau, \sigma>0$, maximum number of iterations $T \geq 1$, initial values $\bm{C}^{(0)} \in \bbC^{N \times K}$, $\bm{\Xi}^{(0)} \in \bbC^{m \times K}$}
\output{$\widebar{\bm{C}} = \text{\texttt{primal-dual-wSRLASSO-C}}(\bm{A},\bm{b},\bm{w},\bm{G},\lambda,\tau,\sigma,T,\bm{C}^{(0)}, \bm{\Xi}^{(0)})$, an approximate minimizer of \eqref{wsr-LASSO-coeff}}
\initialize{$\widebar{\bm{C}}^{(0)} = \bm{0} \in \bbC^{N \times K}$ }
\BlankLine
\For{$n = 0,1,\ldots T-1$}{
$ \bm{P} = (p_{ik})^{N,K}_{j,k=1}= \bm{C}^{(n)}-\tau \bm{A}^*\bm{\Xi}^{(n)}$
\\
\For{$i = 1,\ldots,N$}{
$\bm{p}_i = (p_{ik})^{K}_{k=1}$
\\
$(c^{(n+1)}_{ik})^{K}_{k=1} =  \max \{ \nmu{\bm{G}^{1/2} \bm{p}_{i}}_2 - \tau\lambda w_i , 0 \} \frac{\bm{p}_{i}}{\nmu{\bm{G}^{1/2} \bm{p}_{i}}_2}$
}
$\bm{C}^{(n+1)} = (c^{(n+1)}_{ik})^{N,K}_{i,k=1}$
\\
$ \bm{Q}  =\bm{\Xi}^{(n)} + \sigma \bm{A} (2 \bm{C}^{(n+1)}-\bm{C}^{(n)})-\sigma \bm{B}$
\\
$ \bm{\Xi}^{(n+1)} = \min \left \{ 1, \frac{1}{\nm{\bm{Q} \bm{G}^{1/2}}_{2,2}} \right \} \bm{Q}$
\\
$\widebar{\bm{C}}^{(n+1)} =  \frac{n}{n+1} \widebar{\bm{C}}^{(n)} + \frac{1}{n+1} \bm{C}^{(n+1)} $
}
$\widebar{\bm{C}} = \widebar{\bm{C}}^{(T)}$
\end{algorithm}

\rem{
Note that even though the square-root matrix $\bm{G}^{1/2}$ is used in Algorithm \ref{a:primal-dual-wSRLASSO-coeff}, this matrix does not need to be computed. Indeed,
\bes{
\nmu{\bm{G}^{1/2} \bm{d}}_2 = \sqrt{\bm{d}^* \bm{G} \bm{d}},\quad \bm{d} \in \bbC^K,
}
and for a matrix $\bm{C} \in \bbC^{N \times K}$, we have
\bes{
\nmu{\bm{C} \bm{G}^{1/2}}_{2,2} = \sqrt{\sum^{N}_{i=1} \nm{\bm{G}^{1/2} \bm{c}_i}^2_2} = \sqrt{\sum^{N}_{i=1} \bm{c}^*_i \bm{G} \bm{c}_i },
}
where $\bm{c}_i \in \bbC^K$ is the $i$th row of $\bm{C}$. In particular, computing $\nmu{\bm{G}^{1/2} \bm{d}} $ involves $c(F(\bm{G}) + K)$ arithmetic operations, and computing $\nmu{\bm{C} \bm{G}^{1/2}}_{2,2} $ involves $c m (F(\bm{G}) + K)$ arithmetic operations, for some universal constant $c>0$.
}

To conclude this section, we now state and prove a lemma on the computational cost of Algorithm \ref{a:primal-dual-wSRLASSO-coeff}. This will be used later when proving the main theorems:

\lem{
[Computational cost of Algorithm \ref{a:primal-dual-wSRLASSO-coeff}] 
\label{l:comp-cost-PDI}
The computational cost of Algorithm \ref{a:primal-dual-wSRLASSO-coeff} is bounded by 
\bes{
c \cdot \left ( m \cdot N \cdot K + (m+N) \cdot (F(\bm{G})+K) \right ) \cdot T,
}
where $c> 0$ is a universal constant.
}
\prf{
We proceed line-by-line. Line $2$ involves a matrix-matrix multiplication and matrix subtraction, for a total of at most
\bes{
c \cdot m \cdot N \cdot K\qquad \mbox{(line 2)}
}
arithmetic operations for some universal constant $c$. Now consider lines 3--5. By the previous remark, we may calculate $\nmu{\bm{G}^{1/2} \bm{p}_i}_{2} = \sqrt{\bm{p}^*_i \bm{G} \bm{p}_i}$ using one multiplication with the matrix $\bm{G}$, one inner product of vectors of length $K$ and one square root (recall from Definition \ref{d:AlgPolyApprox} that we count square roots as arithmetic operations). This involves at most $c \cdot(F(\bm{G}) + K)$ arithmetic operations. Hence the cost of line 5 is at most
\bes{
c \cdot (F(\bm{G}) + K)\qquad \mbox{(line 5)},
}
for a possibly different universal constant $c$. Therefore, the total cost of lines 3--5 is
\bes{
c \cdot (F(\bm{G}) + K) \cdot N \qquad \mbox{(lines 3--5)}.
}
Line 7 involves no arithmetic operations and line 8 involves at most
\bes{
c \cdot m \cdot N \cdot K\qquad \mbox{(line 8)}
}
operations. Consider line 9. Due to the previous remark, the computation of $\nmu{\bm{Q} \bm{G}^{1/2}}_{2,2}$ can be performed in at most $c \cdot m \cdot (F(\bm{G}) + K)$ operations (since $\bm{Q}$ is of size $m \times K$). Hence line 9 involves at most
\bes{
c \cdot m \cdot(F(\bm{G}) + K)\qquad \mbox{(line 9)}
}
operations. Finally, line 10 involves at most
\bes{
c \cdot N \cdot K\qquad \mbox{(line 10)}
}
operations. After simplifying, we deduce that lines 2--10 involve at most
\bes{
c  \cdot\left ( m  \cdot N \cdot K + (K+F(\bm{G}) ) \cdot (N+m) \right )
}
operations. The result now follows by multiplying this by the number of iterations $T$.
}

\subsection{The algorithms in Theorems \ref{t:main-res-algo-alg}, \ref{t:main-res-algo-alg_infty} and \ref{t:main-res-algo-alg_exp}}\label{ss:the-algorithms}

We are now almost ready to specify the algorithms used in Theorems \ref{t:main-res-algo-alg}, \ref{t:main-res-algo-alg_infty} and \ref{t:main-res-algo-alg_exp}. 
Notice that Algorithms \ref{a:primal-dual-wSRLASSO} and \ref{a:primal-dual-wSRLASSO-coeff} require the measurement matrix $\bm{A}$ as an input. Hence, we first describe the computation of this matrix for Chebyshev and Legendre polynomials. This is summarized in Algorithm \ref{a:construct-A}. Notice that line 5 of this algorithm involves evaluating the first $k$ one-dimensional Chebyshev or Legendre polynomials. This can be done efficiently via the three-term recurrence relation, as explained in the proof of the following result:

\begin{algorithm}[t]
\caption{\texttt{construct-A} -- constructing the measurement matrix \R{def-measMatrix}}
\label{a:construct-A}
\SetKwInOut{input}{inputs}
\SetKwInOut{initialize}{initialize}
\SetKwInOut{output}{output}
\input{sample points $\bm{y}_1,\ldots,\bm{y}_m \in \cU^d$, finite index set $\Lambda = \{ \bm{\nu}_1,\ldots,\bm{\nu}_N \} \subset \cF$}
\output{$\bm{A} = \texttt{construct-A}((\bm{y}_i)^{m}_{i=1},\Lambda) \in \bbC^{m \times N}$, the measurement matrix \R{def-measMatrix}}
\initialize{$\widebar{\bm{C}}^{(0)} = \bm{0} \in \bbC^{N \times K}$ }
\BlankLine
$k = \max \{ j : (\bm{\nu}_i)_j \neq 0,\ i = 1,\ldots,N,\ j = 1,\ldots,d \}$
\\ 
$n = \max \{ (\bm{\nu}_i)_j : i = 1,\ldots,N,\ j = 1,\ldots,n \}$
\\
\For{$i = 1,\ldots ,m$}{
Set $\bm{z} = (z_j)^{k}_{j=1} = ((\bm{y}_i)_j)^{k}_{j=1}$
\\
$b_{ij} = \Psi_j(z_i)$, $i = 1,\ldots,k$, $j = 0,\ldots,n$, 
\\
\For{$j = 1,\ldots,N$}{
$a_{ij} = \prod^{n}_{l=1} b_{l,(\bm{\nu}_j)_l }$
}
}
$\bm{A} = \frac{1}{\sqrt{m}} ( a_{ij} )^{m,N}_{i,j=1}$
\end{algorithm}

\lem{[Computational cost of Algorithm \ref{a:construct-A}]
\label{l:comp-cost-A}
The computational cost of Algorithm \ref{a:construct-A} is bounded by
\bes{
c \cdot m \cdot (n +N ) \cdot k,
}
where $c > 0$ is a universal constant and $k$ and $n$ are as in lines 1 and 2 of the algorithm.
}
\prf{
Consider line 5 of the algorithm. Evaluation of the first $k+1$ Chebyshev or Legendre polynomials can be done via the three-term recurrence relation. In the Chebyshev case, this is
\bes{
\Psi_0(z) = 1,\quad \Psi_1(z) = \sqrt{2} z,\qquad \Psi_{j+1}(z) = 2 z \Psi_j(z) - c_j \Psi_{j-1}(z),\quad j = 1,\ldots,k,
}
where $c_j = 1$ if $j \geq 1$ and $1/\sqrt{2}$ otherwise, and in the Legendre case, it is
\eas{
\Psi_0(z) &= 1,\quad \Psi_1(z) = \sqrt{3} z,
\\
 \Psi_{j+1}(z) & = \frac{\sqrt{j+3/2}}{j+1}\left(\frac{2j+1}{\sqrt{j+1/2}} z \Psi_j(z) - \frac{j}{\sqrt{j-1/2}} \Psi_{j-1}(z) \right ),\quad j = 2,\ldots,k,
}
(recall that these polynomials are normalized with respect to their respective probability measures). Hence the computational cost for line 5 is bounded by $c \cdot n \cdot k$. The computational cost for lines 6--8 is precisely $N \cdot (k-1)$. Hence, the computational cost for forming each row of $\bm{A}$ is bounded by $c \cdot (n \cdot k + N \cdot k )$. The result now follows.
}

With this in hand, we are now ready to specify the algorithms used in Theorem \ref{t:main-res-algo-alg}, Theorem \ref{t:main-res-algo-alg_infty} and \ref{t:main-res-algo-alg_exp}. These are given in Table \ref{tab:the-algorithms}.

\begin{table}
\benbox{
\begin{itemize}
\item Let $m$, $\epsilon$, $n$ and $t$ be as given in the particular theorem and set:
\bull{
\item $\Lambda = \Lambda^{\mathsf{HC}}_{n,d}$ (Theorems \ref{t:main-res-map-alg} and \ref{t:main-res-map-alg_exp}) or $\Lambda = \Lambda^{\mathsf{HCI}}_{n}$ (Theorem \ref{t:main-res-map-alg_infty}),
\item $\lambda = (4 \sqrt{m/L})^{-1}$, where $L = L(m,d,\epsilon)$ is as in \R{Ldef},
\item $\tau = \sigma =(\Theta(n,d))^{-\alpha}$, where $\Theta(n,d)$ and $\alpha$ are as in \R{Thetadef} and \R{main_alpha_def}, respectively,
\item $T = \lceil 2 (\Theta(n,d))^{\alpha} t \rceil$.
}
\item Let $\bm{D} = (d_{ik})^{m,K}_{i,k=1} \in \bbC^{m \times K}$ and $\bm{Y} = (\bm{y}_i)^{m}_{i=1}$ be an input, as in \R{data_we_get}, and set $\bm{B} = \frac{1}{\sqrt{m}} \bm{D}$.

\item Compute $\bm{A} = \texttt{construct-A}(\bm{Y},\Lambda)$.

\item Let $\bm{G}$ and $\bm{w}$ be as in \R{Gram-matrix-def} and \R{weights_def}, respectively.

\item Define the output $\widebar{\bm{C}} = \cA(\bm{D})$, where
\bes{
\cA(\bm{D}) = \text{\texttt{primal-dual-wSRLASSO-C}}\left (\bm{A},\bm{B},\bm{w},\bm{G},\lambda,\tau,\sigma,T,\bm{0},\bm{0} \right )
}
\end{itemize} 
}
\vspace{-1pc}

\caption{The algorithms $\cA : \cU^m \times \bbC^{m \times K } \rightarrow \bbC^{N \times K}$ used in Theorem \ref{t:main-res-algo-alg}, Theorem \ref{t:main-res-algo-alg_infty} and \ref{t:main-res-algo-alg_exp}. }
\label{tab:the-algorithms}
\end{table}

\subsection{An efficient restarting procedure for the primal-dual iteration and the algorithms used in Theorems \ref{t:main-res-effic-algo-alg}, \ref{t:main-res-effic-algo-alg_infty} and \ref{t:main-res-effic-algo-alg_exp}}\label{Section_restart}

While the primal-dual iteration converges under very general conditions, it typically does so very slowly, with the error in the objective function decreasing like $\ord{1/n}$, where $n$ is the iteration number. To obtain exponential convergence (down to some controlled tolerance) we employ a restarting procedure. This is based on recent work of \cite{colbrook2021can,colbrook2021warpd}.

Restarting is a general concept in optimization, where the output of an algorithm after a fixed number of steps is then fed into the algorithm as input, after suitably scaling the parameters of the algorithm \cite{renegar2022simple,roulet2020sharpness,roulet2020computational}. In the case of the primal-dual iteration for the weighted SR-LASSO problem, this procedure involves three hyperparameters: a \textit{tolerance} $\zeta' > 0$ and \textit{scale} parameters $0 < r < 1$ and $s > 0$. After applying one step of the primal-dual iteration (Algorithm \ref{a:primal-dual-wSRLASSO} or \ref{a:primal-dual-wSRLASSO-coeff}) yielding an output $\bm{c}^{(1)}$, it then scales this vector and the right-hand side vector $\bm{b}$ by an exponentially-decaying factor $a_l$ (defined in terms of $\zeta'$, $r$ and $s$), before feeding in these values into the primal-dual iteration as input.

We explain the motivations behind the specific form of the restart procedure for the primal-dual iteration later in \S \ref{ss:restart-scheme}. For now, we simply state the procedures in the case of the weighted SR-LASSO problems \R{wsr-LASSO} and \R{wsr-LASSO-coeff}. These are given in Algorithms \ref{a:primal-dual-wSRLASSO-restart} and \ref{a:primal-dual-wSRLASSO-coeff-restart}, respectively. With these in hand, we can also give the algorithms used in  Theorems \ref{t:main-res-effic-algo-alg}, \ref{t:main-res-effic-algo-alg_infty} and \ref{t:main-res-effic-algo-alg_exp}. See Table \ref{tab:the-effic-algorithms}.

\begin{algorithm}[t]
\caption{\texttt{primal-dual-rst-wSRLASSO} -- the restarted primal-dual iteration for the weighted SR-LASSO problem \eqref{wsr-LASSO}}
\label{a:primal-dual-wSRLASSO-restart}
\SetKwInOut{input}{inputs}
\SetKwInOut{initialize}{initialize}
\SetKwInOut{output}{output}
\input{measurement matrix $\bm{A} \in \bbC^{m \times N}$, measurements $\bm{b} \in \cV^N_h$, positive weights $\bm{w} = (w_i)^{N}_{i=1}$, parameter $\lambda>0$, stepsizes $\tau, \sigma>0$, number of primal-dual iterations $T \geq 1$, number of restarts $R \geq 1$, tolerance $\zeta' > 0$, scale parameter $0 < r < 1$, constant $s > 0$, initial values $\bm{c}^{(0)}= \0 \in \cV^N_h$ $\bm{\xi}^{(0)}=\0 \in \cV^m_h$.}
\output{$\tilde{\bm{c}} = \text{\texttt{primal-dual-rst-wSRLASSO}}(\bm{A},\bm{b},\bm{w},\lambda,\tau,\sigma,T,R,\zeta',r,s)$, an approximate minimizer of \eqref{wsr-LASSO}}
\initialize{$\bar{\bm{c}}^{(0)} = \bm{0} \in \cV^N_h$  , $\varepsilon_0 = \nm{\bm{b}}_{2;\cV}$ }
\BlankLine

\For{$l = 0,\ldots,R-1$}{
$\varepsilon_{l+1} = r (\varepsilon_l + \zeta')$
\\
$a_l = s \varepsilon_{l+1}$
\\
$\tilde{\bm{c}}^{(l+1)} = a_l \cdot \text{\texttt{primal-dual-wSRLASSO}}(\bm{A},\bm{b}/a_l,\bm{w},\lambda,\tau,\sigma,T,\tilde{\bm{c}}^{(l)}/a_l, \bm{0})$ 
}
$\tilde{\bm{c}} = \tilde{\bm{c}}^{(R)}$ 
\end{algorithm}

\begin{algorithm}[t]
\caption{\texttt{primal-dual-rst-wSRLASSO-C} -- the restarted primal-dual iteration for the weighted SR-LASSO problem \eqref{wsr-LASSO-coeff}}
\label{a:primal-dual-wSRLASSO-coeff-restart}
\SetKwInOut{input}{inputs}
\SetKwInOut{initialize}{initialize}
\SetKwInOut{output}{output}
\input{measurement matrix $\bm{A} \in \bbC^{m \times N}$, measurements $\bm{B} \in \bbC^{N \times K}$, positive weights $\bm{w} = (w_i)^{N}_{i=1}$, Gram matrix $\bm{G} \in \bbC^{K \times K}$, parameter $\lambda>0$, stepsizes $\tau, \sigma>0$, number of primal-dual iterations $T \geq 1$, number of restarts $R \geq 1$, tolerance $\zeta' > 0$, scale parameter $0 < r < 1$, constant $s > 0$, initial values $\bm{C}^{(0)}= \0 \in \bbC^{N \times K}$, $\bm{\Xi}^{(0)}= \0 \in \bbC^{m \times K}$}
\output{$\widetilde{\bm{C}} = \text{\texttt{primal-dual-rst-wSRLASSO-C}}(\bm{A},\bm{b},\bm{w},\bm{G},\lambda,\tau,\sigma,T,R,\zeta',r,s)$, an approximate minimizer of \eqref{wsr-LASSO-coeff}}
\initialize{$\widetilde{\bm{C}}^{(0)} = \bm{0} \in \bbC^{N\times K}$, $\varepsilon_0 = \nmu{\bm{B}\bm{G}^{1/2}}_{2;2}$ }
\BlankLine

\For{$l = 0,\ldots,R-1$}{
$\varepsilon_{l+1} = r (\varepsilon_l + \zeta)$
\\
$a_l = s \varepsilon_{l+1}$
\\
$\widetilde{\bm{C}}^{(l+1)} = a_l \cdot \text{\texttt{primal-dual-wSRLASSO-C}}(\bm{A},\bm{B}/a_l,\bm{w},\bm{G},\lambda,\tau,\sigma,T,\widetilde{\bm{C}}^{(l)}/a_l, \bm{0})$
}
$\widetilde{\bm{C}} = \widetilde{\bm{C}}^{(R)}$
\end{algorithm}

\begin{table}
\benbox{
\begin{itemize}
\item Let $m$, $\epsilon$, $n$, $t$ and $\zeta'$ be as given in the particular theorem and set:
\bull{
\item $\Lambda = \Lambda^{\mathsf{HC}}_{n,d}$ (Theorems \ref{t:main-res-effic-algo-alg} and \ref{t:main-res-effic-algo-alg_exp}) or $\Lambda = \Lambda^{\mathsf{HCI}}_{n}$ (Theorem \ref{t:main-res-effic-algo-alg_infty}),
\item $\lambda = (4 \sqrt{m/L})^{-1}$, where $L = L(m,d,\epsilon)$ is as in \R{Ldef},
\item $\tau = \sigma = (\Theta(n,d))^{-\alpha}$, where $\Theta(n,d)$ and $\alpha$ are as in \R{Thetadef} and \R{main_alpha_def}, respectively,
\item $T= \lceil (\Theta(n,d))^{\alpha} c^{\star} \rceil $, where $c^{\star}$ is a universal constant,
\item $R = t$
\item $r = \E^{-1}$
\item $s = \frac{(\Theta(n,d))^{\alpha} T}{2}$
}
\item Let $\bm{D} = (d_{ik})^{m,K}_{i,k=1} \in \bbC^{m \times K}$ and $\bm{Y} = (\bm{y}_i)^{m}_{i=1}$ be an input, as in \R{data_we_get},  and set $\bm{B} = \frac{1}{\sqrt{m}} \bm{D}$.

\item Compute $\bm{A} = \texttt{construct-A}(\bm{Y},\Lambda)$.
\item Let $\bm{G}$, $\bm{A}$ and $\bm{w}$ be as in \R{def-measMatrix}, \R{Gram-matrix-def} and \R{weights_def}, respectively.

\item Define the output $\widetilde{\bm{C}} = \cA(\bm{D})$, where
\bes{
  \cA(\bm{D}) =   \text{\texttt{primal-dual-rst-wSRLASSO-C}}(\bm{A},\bm{B},\bm{w},\bm{G},\lambda,\tau,\sigma,T,R,\zeta,r,c)
}
\end{itemize}
}
\vspace{-1pc}

\caption{The algorithms $\cA : \cU^m \times \bbC^{m \times K } \rightarrow \bbC^{N \times K}$ used in Theorems \ref{t:main-res-effic-algo-alg}, \ref{t:main-res-effic-algo-alg_infty} and \ref{t:main-res-effic-algo-alg_exp}.}
\label{tab:the-effic-algorithms}
\end{table}

Note that these algorithms involve a number $c^{\star}$, which is a universal constant. It is possible to provide a precise numerical value of this constant by carefully tracking the constants in several of the proof steps. Since doing so is not especially illuminative, we forgo this additional effort. Instead, we now give a little more detail on this constant:

\rem{
\label{cstar-no}
From \R{cstar-def} we see that $c^{\star} = 3296 \sqrt{c_0}$, where $c_0$ is the universal constant that arises in \R{zeta-alg-def}. As shown in the proof of Theorem \ref{t:main-res-map-alg_inex}, the constant $c_0$ needs to be chosen sufficiently large so that the measurement matrix $\bm{A}$ satisfies the so-called \textit{weighted RIP}. In particular, it is related to the universal constant $c > 0$ defined in Lemma \ref{l:LegMat_RIP}. See, in particular, \R{m-cond-for-wRIP}. A numerical value for this constant can indeed be found using results shown in \cite{chkifa2018polynomial}. With this in hand, one can then keep track of the constant $c_0$ in the proof of Theorem \ref{t:main-res-map-alg_inex} to find its numerical value. This discussion also highlights why tracking the value of $c^{\star}$ is non particularly illuminative. Indeed, it is well-known that universal constants appearing in RIP estimates in compressed sensing are generally very pessimistic \cite{adcock2021sparse,foucart2013mathematical,adcock2021compressive}.
}

\section{Numerical experiments}\label{s:num-exp}

\subsection{Experimental setup}

We first describe the experimental setup.

\subsubsection{Hyperparameter values}\label{ss:hyperparam-values}

\begin{table}
\begin{center}
\begin{tabular}{c|cc}
Parameter & Value & Notes
\\
\hline
$\lambda$ & $( \sqrt{25 m})^{-1}$ & Based on  \cite[App.\ A]{adcock2021sparse}
\\
$\sigma$ & $\nm{\bm{A}}^{-1}_2$ & Based on Lemma \ref{lemma-subopti}
\\
$\tau$ & $\nm{\bm{A}}^{-1}_2$ & Based on Lemma \ref{lemma-subopti}
\\
$r$ & $\E^{-1}$ & Based on Theorem \ref{thm:restart-scheme-bound}
\\
$T$ & $\left \lceil \frac{2 \nm{\bm{A}}_2}{r} \right \rceil$ & Based on Theorem \ref{thm:restart-scheme-bound}, assuming $C = 1$
\\
$s$ & $\frac{T}{2 \nm{\bm{A}}_2}  $ & Based on Theorem \ref{thm:restart-scheme-bound}
\end{tabular}
\end{center}

\caption{Hyperparameter values used in the numerical experiments. The first three parameters are used in both the unrestarted and restarted primal-dual iterations. The final three parameters are used in the restarted scheme only.}
\label{tab:the-parameters}
\end{table}

The algorithms used in the main theorems (see Tables \ref{tab:the-algorithms} and \ref{tab:the-effic-algorithms}) are designed to ensure the desired error bounds. In our numerical experiments, we deviate from these values in a number of minor ways. However, our hyperparameter choices are still closely based on theory. We now discuss the hyperparameter choices used in the experiments. These choices are summarized in Table \ref{tab:the-parameters}.

First, we take the parameter $\lambda$ to be $\lambda = ( \sqrt{25 m})^{-1}$. This differs somewhat from the value $\lambda = (4 \sqrt{m/L})^{-1}$ used in the theoretical algorithms. The rationale behind doing this is that $L$ is, in practice, a polylogarithmic factor that arises from the compressed sensing theory. It is well known that logarithmic factors appearing in compressed sensing theory are generally quite pessimistic \cite{adcock2021sparse,foucart2013mathematical,adcock2021compressive}. Therefore, we avoid using $L$. The choice $\lambda = (5 \sqrt{m})^{-1}$ was obtained in \cite[App.\ A]{adcock2021sparse} after manual tuning.

As shown later, the primal-dual iteration converges subject to the condition $\nm{\bm{A}}^2_{2} \leq (\tau \sigma)^{-1}$. See Lemma \ref{lemma-subopti}. Since the error bound \R{subopti-error} scales linearly in $\tau^{-1}$ and $\sigma^{-1}$, a standard choice for these parameters is
\be{
\label{tau-sigma-practical}
\tau = \sigma = 1/\nm{\bm{A}}_2.
}
In Tables \ref{tab:the-algorithms} and \ref{tab:the-effic-algorithms} we choose $\tau = \sigma = (\Theta(n,d))^{-\alpha}$, since the latter is an upper bound for $\nm{\bm{A}}_{2}$, i.e., $\nm{\bm{A}}_{2} \leq (\Theta(n,d))^{\alpha}$. See \R{eq:Bound_A}. This bound is arguably quite crude. The reason for using it in our main theorems is to avoid having to compute $\nm{\bm{A}}_{2}$, since this generally cannot be done in finitely-many arithmetic operations. However, in our numerical experiments we simply use \R{tau-sigma-practical} instead, since it is simpler and $\nm{\bm{A}}_2$ can approximated efficiently in practice.

For the restarting scheme, we also have the scale parameter $0 < r < 1$, the constant $s > 0$ and the number of inner iterations $T$. These parameters are inferred from Theorem \ref{thm:restart-scheme-bound}. This result shows that the error in the restarted primal dual iteration after $l$ restarts is bounded by
\be{
\label{pdi-restart-bd}
r^l \nmu{\bm{b}}_{2;\cV} + \frac{r}{1-r} \zeta',
}
provided
\bes{
T = \left \lceil \frac{2 C}{r \sqrt{\sigma \tau}} \right \rceil,\qquad a_l = \frac12 \sigma \varepsilon_{l+1} T,\ l = 0,2,\ldots .
}
Here, as discussed in Theorem \ref{thm:restart-scheme-bound}, $C > 0$ is a numerical constant that arises from the compressed sensing theory.
This and the choice \R{tau-sigma-practical} leads immediately to the following value for $s$:
\bes{
s = \frac{T}{2 \nm{\bm{A}}_2}  .
}
Unfortunately, the constant $C$ is difficult to determine exactly (it is closely related to the constant $c^{\star}$ discussed in Remark \ref{cstar-no}). In our experiments, we simply pick the value $C = 1$. This immediately yields
\bes{
T = \left \lceil \frac{2 \nm{\bm{A}}_2}{r} \right \rceil.
}
Finally, to determine a value of $r$ we consider the error bound \R{pdi-restart-bd}. This is based on \cite{colbrook2021can}. After $l$ restarts, the total number of iterations $t = T l$. Substituting the value of $T$, we see that
\be{
\label{restarted-error-iterations}
r^{l} = \exp \left ( \log(r) t / T \right ) = \exp \left ( \log(r) \left \lceil \frac{2 \nm{\bm{A}}_2}{r} \right \rceil^{-1} t \right ).
}
Ignoring the ceiling function, it therefore makes sense to choose $0 < r < 1$ to minimize the function $r \mapsto r \log(r)$. This attains its minimum value of $-\E^{-1}$ at $r = \E^{-1}$. Hence we use this value.

\subsubsection{Test functions}

We consider four test functions. The first two are scalar-valued functions, given by
\be{
\label{f1-def}
f_1(\bm{y}) = \exp\left( - \frac{1}{2d} \sum_{k=1}^d y_k \right),\ \forall \bm{y} \in \cU,\quad \mbox{with $d = 2$},
}
and
\be{
\label{f2-def}
f_2(\bm{y}) = \exp\left( - \frac{2}{d} \sum_{k=1}^d (y_k - w_k)^2 \right),\ \forall \bm{y} \in \cU,\qquad 
 \mbox{with $w_k = \frac{(-1)^k}{k+1}$, $\forall k \in [d]$ and $d = 16$}.
}
These are standard test functions (see, e.g., \cite[\S A.1]{adcock2021sparse}).
The first function varies very little with respect to $\bm{y}$. Hence it is expected to be very well-approximated by a sparse polynomial approximation. The second has more variation in $\bm{y}$, therefore we expect a larger approximation error. 

We also consider two Hilbert-valued functions. These both arise as solutions of the parametric elliptic diffusion equation
\be{
\label{PDE}
-\nabla \cdot (a(\bm{x},\bm{y}) \nabla u(\bm{x},\bm{y})) =g(\bm{x}),\  \forall \bm{x} \in D,\, \bm{y} \in \cU,\quad u(\bm{x},\bm{y})= 0,\ \forall \bm{x} \in \partial D,\, \bm{y} \in \cU,
}
which is a standard problem in the parametric PDE literature. We take the physical domain $D$ as $D = (0,1)^2$. For simplicity, we also choose $g(\bm{x}) = 10$ 
to be constant. In this case, the solution map
\bes{
\cU \rightarrow \cV,\ \bm{y} \mapsto u(\cdot,\bm{y}),\qquad \cV = H^1_0(D),
}
is a Hilbert-valued function with codomain being the Sobolev space $H^1_0(D)$. We consider two different setups, leading to smooth and less smooth Hilbert-valued functions, which we denote as $f_3$ and $f_4$, respectively. The first is is a simple two-dimensional problem with lognormal diffusion coefficient:
\be{
\label{f3-def}
f_3:\quad d = 2,\ a(\bm{x},\bm{y}) = 5 + \exp(x_1 y_1 + x_2 y_2).
}
For the second, we consider the diffusion coefficient from \cite[Eqn.\ (24)]{adcock2021deep}, modified from an earlier example from \cite[Eqn.\ (5.2)]{nobile2008sparse}, with 30-dimensional parametric dependence and one-dimensional (layered) spatial dependence given by
\be{
\label{f4-def}
\begin{split}
f_4:\qquad d &= 30, \ a(\bm{x},\bm{y}) 
 = \exp\left(1+ y_1 \left(\frac{\sqrt{\pi}\beta}{2}\right)^{1/2} + \sum_{i=2}^d \; \zeta_i \; \vartheta_i(\bm{x}) \; y_i\right), \\
\zeta_i 
& := (\sqrt{\pi} \beta)^{1/2} \exp\left( \frac{-\left( \left\lfloor \frac{i}{2} \right\rfloor \pi \beta\right)^2}{8} \right), 
\quad
\vartheta_i(\bm{x}) 
 := \begin{cases} 
 \sin\left( \left\lfloor\frac{i}{2}\right\rfloor \pi x_1/\beta_p \right) & \mbox{$i$ even,}
 \\
 \cos\left(\left\lfloor\frac{i}{2}\right\rfloor \pi x_1/\beta_p \right) &  \mbox{$i$ odd,} 
 \end{cases}
 \\
 \beta_c &= 1/8 ,\ \beta_p = \max\{1, 2 \beta_c\},\ \beta = \beta_c/\beta_p.
 \end{split}
}

\subsubsection{Error metrics and finite element discretization}

In our experiments, we consider the relative $L^2_{\varrho}(\cU)$-norm error
\be{
\label{rel-error-scalar}
\frac{\nmu{f - \hat{f}}_{L^2_{\varrho}(\cU)}}{\nmu{f}_{L^2_{\varrho}(\cU)}},
}
for the scalar-valued functions $f_1$ and $f_2$ and the relative $L^2_{\varrho}(\cU ; H^1_0(D))$-norm error
\be{
\label{rel-error-Hilbert}
\frac{\nmu{f - \hat{f}}_{L^2_{\varrho}(\cU ; H^1_0(D))}}{\nmu{f}_{L^2_{\varrho}(\cU ; H^1_0(D))}},
}
for the Hilbert-valued functions $f_3$ and $f_4$. To (approximately) compute this error we use a high-order isotropic Smolyak sparse grid quadrature rule based on Clenshaw--Curtis points. This rule is generated using the TASMANIAN software package \cite{stoyanov2015tasmanian}. We set the level of the quadrature rule in each experiment as large as possible within the constraints of computational time and memory. 

We now describe the discretization $\cV_h$ for the Hilbert-valued functions $f_3$ and $f_4$. This is obtained via the finite element method as implemented by Dolfin \cite{LoggWells2010a}, and accessed through the python FEniCS project \cite{AlnaesBlechta2015a}. We generate a regular triangulation $\cT_h$ of $\overline{D}$ composed of triangles $T$ of equal diameter $h_T = h$. We consider a conforming discretization, which results in a finite-dimensional subspace $\cV_h \subset \cV = H_0^1(D)$, where $\cV_h$ is the space spanned by the usual Lagrange finite elements $\{ \varphi_i \}^{K}_{i=1}$ of order $k=1$. 
We rely on the Dolfin {\tt UnitSquareMesh} method to generate a mesh with 33 nodes per side, corresponding to a finite element triangulation with $K=1089$ nodes, 2048 elements and meshsize $h = \sqrt{2}/32$.
 See \cite{dexter2019mixed,adcock2021deep} for further implementation details. 

Explicit forms of the Hilbert-valued functions $f_3$ and $f_4$ are not available. Therefore, computing the relative error requires first computing a reference solution. This is usually done by using a finite element discretization with meshsize an order of magnitude smaller than that used to compute the various approximations. However, our main focus in these experiments is on the polynomial approximation and algorithmic errors $E_{\textsf{app}}$ and $E_{\textsf{alg}}$. Since our theoretical results assert that the approximations are robust to physical discretization error, we do not perform this additional (and costly) computational step. Instead, we compute reference solutions using the same finite element discretization as that used to construct the various approximations. In other words, there is no physical discretization error present in these experiments.

\subsection{Numerical results 1: the optimization error}

Our first experiments, Figures \ref{fig:SV_f1}--\ref{fig:HV_f4}, compare the behaviour of the unrestarted primal-dual iteration to the restarted primal-dual iteration with several different values of the tolerance parameter $\zeta'$. In all cases, we observe a consistent improvement from the restarted scheme. This is particularly noticeable for the functions $f_1$ and $f_3$, since the underlying approximation error $\zeta$ is smaller in these cases. Recall that these functions are well-approximated by polynomials. As predicted by our theoretical results, the error for the restarted scheme decays exponentially fast with respect to the number of iterations to this limiting accuracy. For example, in the case of $f_1$ the restarted scheme (with sufficiently small $\zeta'$) achieves a relative error of less than $10^{-6}$ using only 500 iterations. However, the unrestarted scheme only achieves an error of around $10^{-3}$ after 1000 iterations.

An important takeaway from these experiments is the insensitivity of the algorithm to the parameter $\zeta'$. Our theoretical results only show exponential convergence (with respect to iteration number) when $\zeta' \geq \zeta$, where $\zeta$ is a certain upper bound for the error. This appears unnecessary in practice. For instance, in Figures \ref{fig:SV_f2} and \ref{fig:HV_f4} we expect the underlying error $\zeta$ to be roughly $10^{-2}$ in magnitude, since this is the limiting error achieved by the unrestarted scheme. Yet setting $\zeta' = 10^{-10}$ has no noticeable effect on the performance of the restarted scheme. Moreover, for $\zeta'\in\{10^{-4},10^{-6},10^{-8},10^{-10}\}$ the results are nearly identical in both Figures \ref{fig:SV_f2} and \ref{fig:HV_f4} and hence the plot lines are overlayed for the restarted scheme.

Another noticeable feature of these experiments is the close agreement between the theorized rate of exponential decay of the restarted scheme, which is given by the right-hand side of \R{restarted-error-iterations} and what is observed in practice. Since the value $r = \E^{-1}$ is used in these experiments, in Figures \ref{fig:SV_f1}--\ref{fig:HV_f4} we also plot the function
\be{
\label{theory-error-curve}
\exp \left ( - c t \right ),\qquad c : = \left \lceil 2 \E \nmu{\bm{A}}_2 \right \rceil^{-1}
}
versus the iteration number $t$. This theoretical curve exactly predicts the observed rate of exponential decay of the restarted schemes. 

Finally, in all four figures we also show the error of the (restarted) primal-dual iterates, as well as the ergodic sequences. Despite the theoretical results only holding for the latter, we see similar error decay for the iterates. In fact, the iterates give slightly better performance in the case of the unrestarted scheme. As expected, the ergodic sequence reduces the variation in the error for the restarted scheme. Moreover, plotting the ergodic sequence we can see more clearly the benefit of using restarts over not restarting.

\begin{figure}
\centering
\includegraphics[width=0.35\paperwidth, clip=true, trim=0mm 0mm 0mm 0mm]{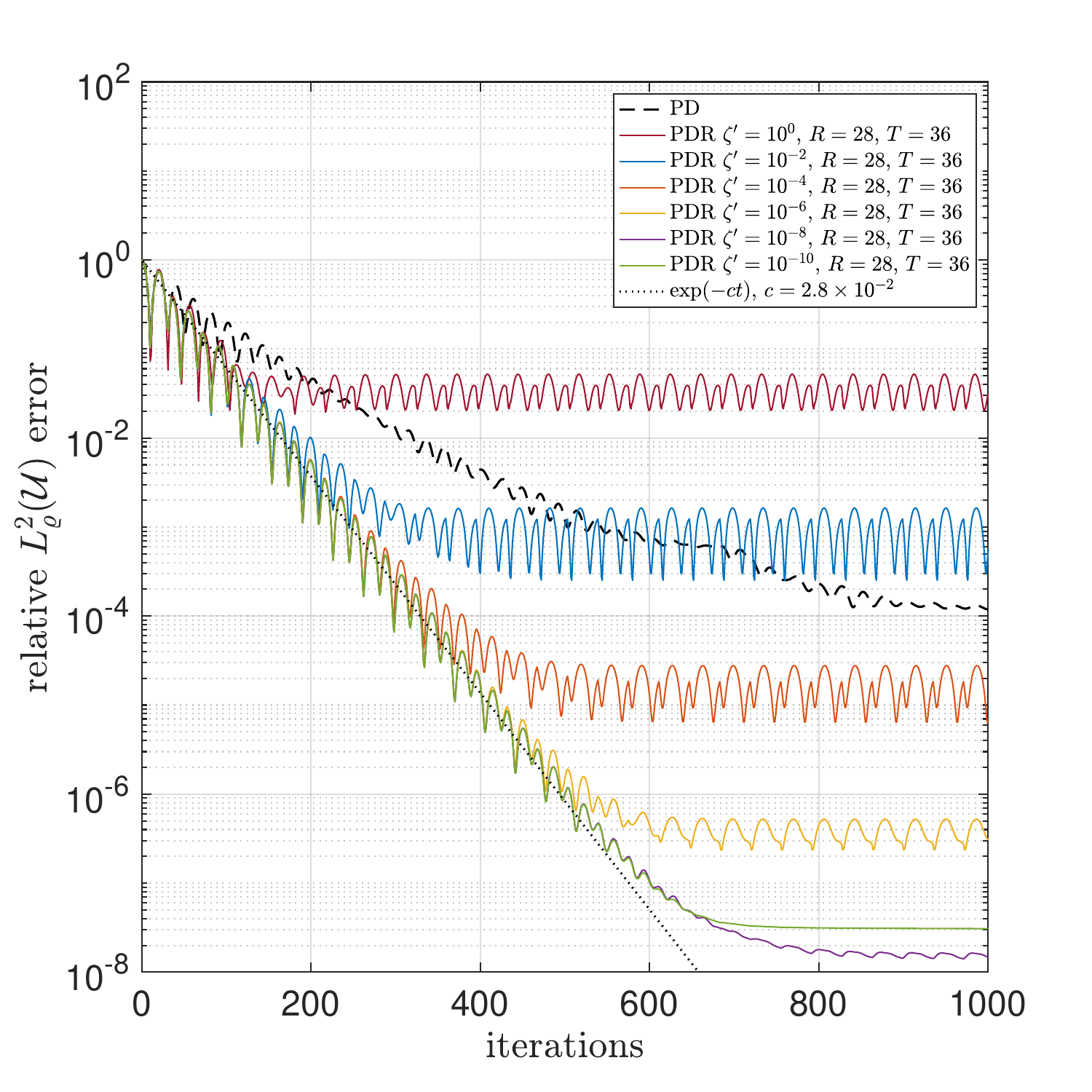}
\includegraphics[width=0.35\paperwidth, clip=true, trim=0mm 0mm 0mm 0mm]{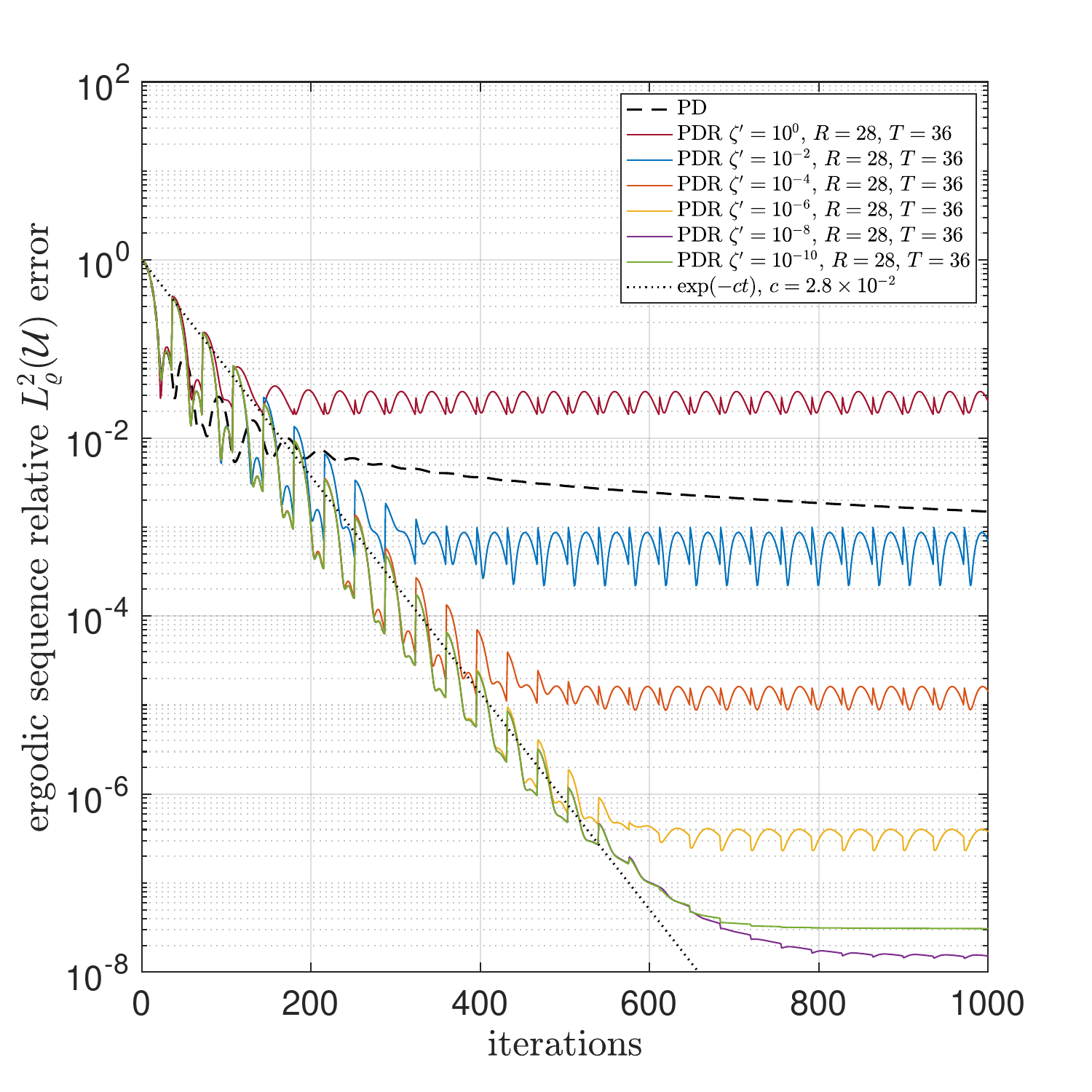}
\caption{Approximation error versus iteration number for the function $f_1$ from \R{f1-def}. This figure shows the 
relative $L^2$ errors of the polynomial approximations obtained from {\bf (left)} the iterates $\bm{c}^{(n)}$ and {\bf (right)} the ergodic sequence $\bar{\bm{c}}^{(n)}$. These approximations are constructed using the Legendre polynomial basis and $m = 250$ sample points drawn randomly and independently from the uniform measure. The index set $\Lambda = \Lambda^{\textsf{HC}}_{n,d}$, where $d = 2$ and $n = 184$, which gives a basis of cardinality $N = |\Lambda| = 997$. We compare the primal dual iteration ``PD'' and the restarted primal dual iteration ``PDR" for various values of the tolerance $\zeta'$. We also plot the theoretical error curve \R{theory-error-curve}, where $t$ is the iteration number.
The quadrature rule used to compute the relative error is a sparse grid rule of level 11 consisting of $M = 7169$ points. 
}
\label{fig:SV_f1}
\end{figure}

\begin{figure}
\centering
\includegraphics[width=0.35\paperwidth, clip=true, trim=0mm 0mm 0mm 0mm]{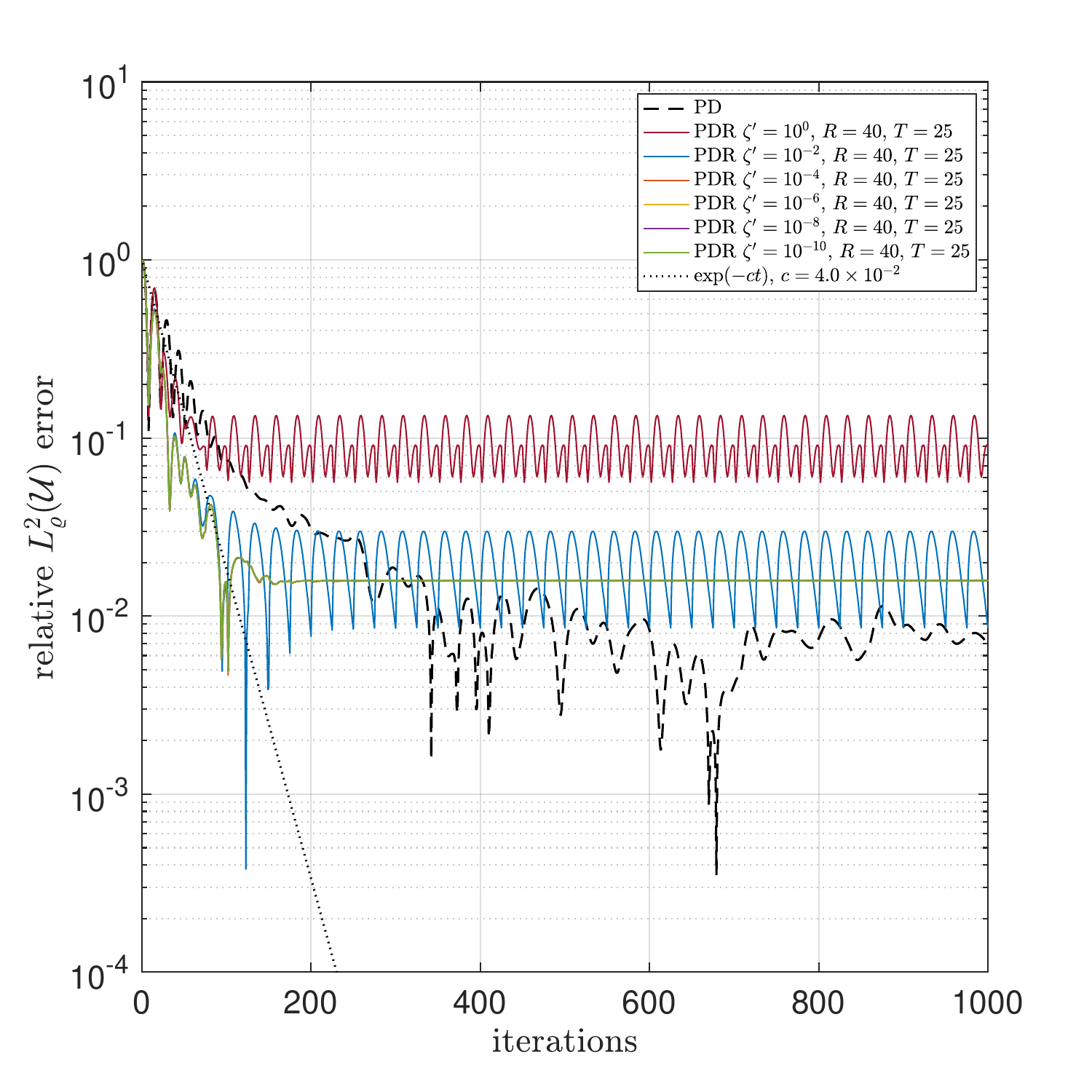}
\includegraphics[width=0.35\paperwidth, clip=true, trim=0mm 0mm 0mm 0mm]{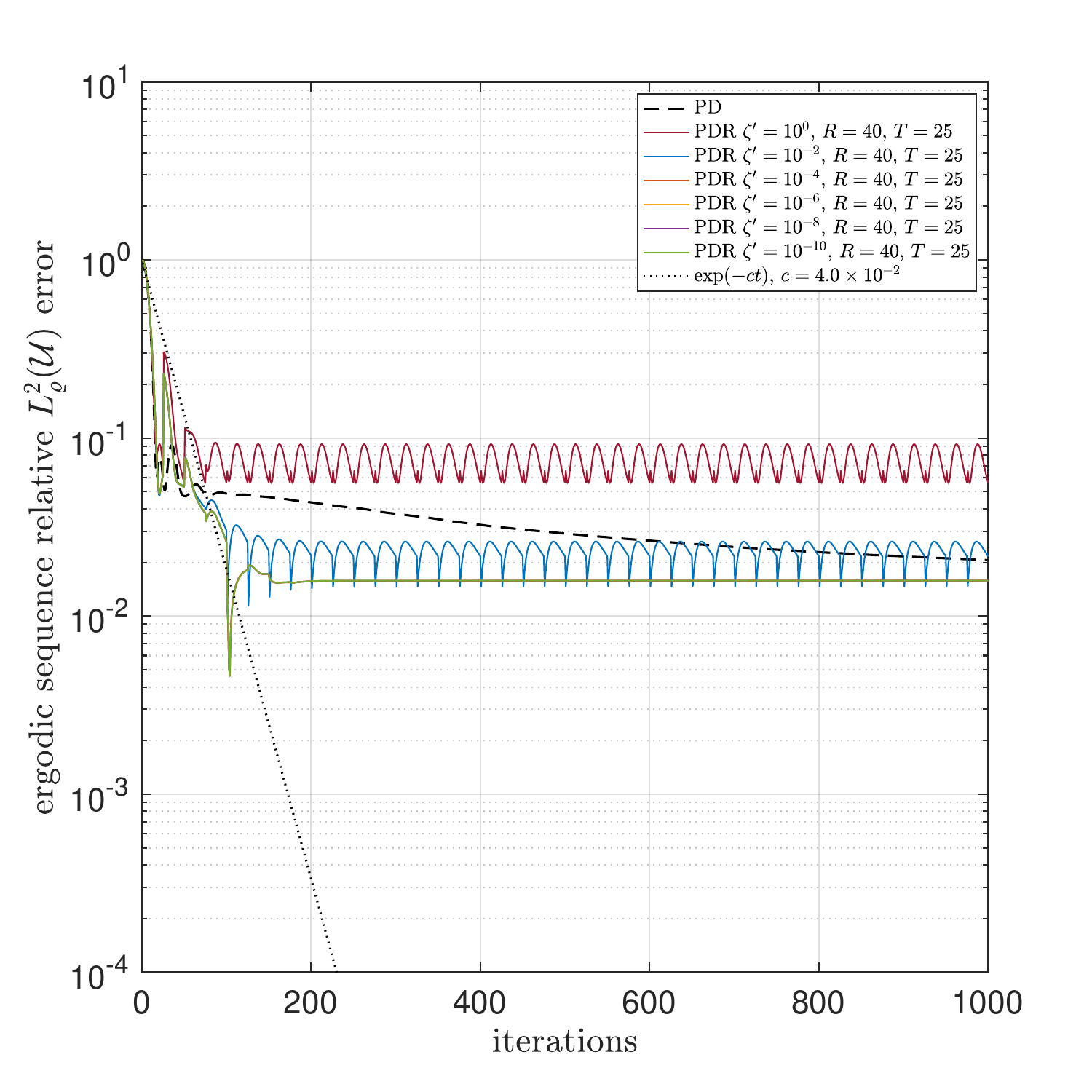}
\caption{Approximation error versus iteration number for the function $f_2$ from \R{f2-def}. This figure shows the 
relative $L^2$ errors of the polynomial approximations obtained from {\bf (left)} the iterates $\bm{c}^{(n)}$ and {\bf (right)} the ergodic sequence $\bar{\bm{c}}^{(n)}$. These approximations are constructed using the Legendre polynomial basis and $m = 2000$ sample points drawn randomly and independently from the uniform measure. The index set $\Lambda = \Lambda^{\textsf{HC}}_{n,d}$, where $d = 16$ and $n = 16$, which gives a basis of cardinality $N = |\Lambda| = 8277$. We compare the primal dual iteration ``PD'' and the restarted primal dual iteration ``PDR" for various values of the tolerance $\zeta'$. We also plot the theoretical error curve \R{theory-error-curve}, where $t$ is the iteration number. The quadrature rule used to compute the relative error is a sparse grid rule of level 5 consisting of $M = 51137$ points.}
\label{fig:SV_f2}
\end{figure}

\begin{figure}
\centering
\includegraphics[width=0.35\paperwidth, clip=true, trim=0mm 0mm 0mm 0mm]{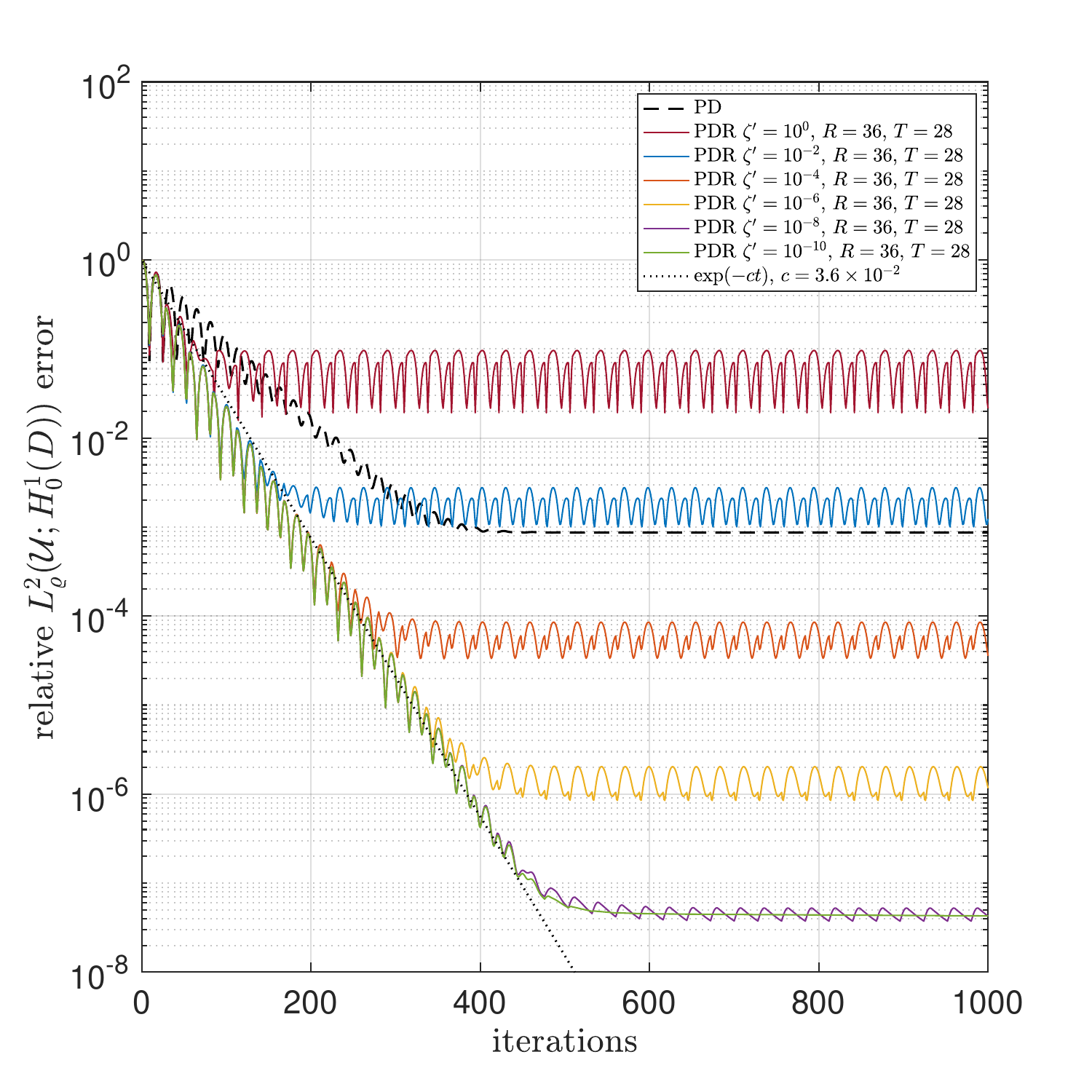}
\includegraphics[width=0.35\paperwidth, clip=true, trim=0mm 0mm 0mm 0mm]{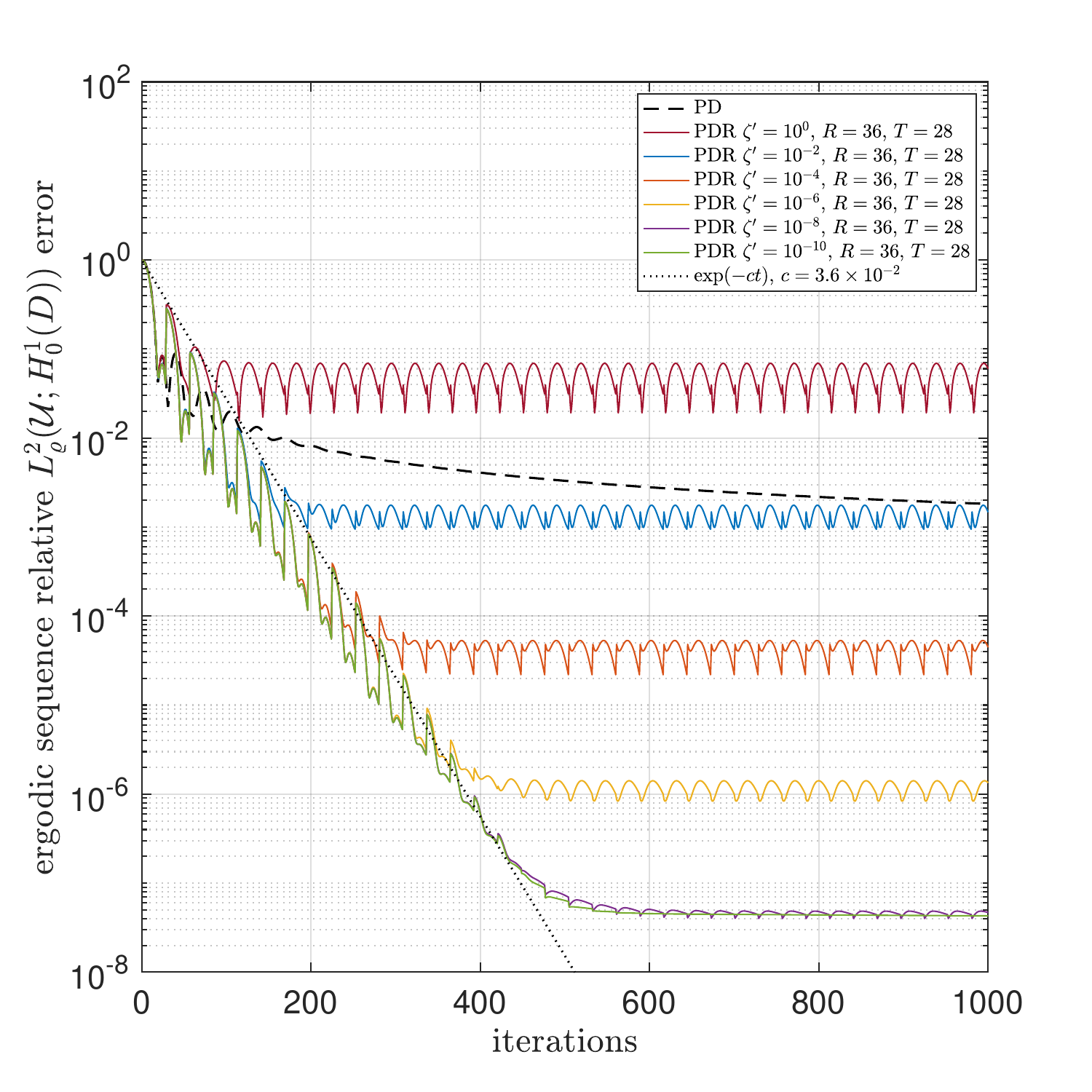}

\caption{Approximation error versus iteration number for the function $f_3$ from \R{f3-def}. This figure shows the 
relative $L^2$ errors of the polynomial approximations obtained from {\bf (left)} the iterates $\bm{c}^{(n)}$ and {\bf (right)} the ergodic sequence $\bar{\bm{c}}^{(n)}$. These approximations are constructed using the Legendre polynomial basis and $m = 250$ sample points drawn randomly and independently from the uniform measure. The index set $\Lambda = \Lambda^{\textsf{HC}}_{n,d}$, where $d = 2$ and $n = 184$, which gives a basis of cardinality $N = |\Lambda| = 997$. We compare the primal dual iteration ``PD'' and the restarted primal dual iteration ``PDR" for various values of the tolerance $\zeta'$. We also plot the theoretical error curve \R{theory-error-curve}, where $t$ is the iteration number. The quadrature rule used to compute the relative error is a sparse grid rule of level 9 consisting of $M = 1537$ points.}
\label{fig:HV_f3}
\end{figure}

\begin{figure}
\centering
\includegraphics[width=0.35\paperwidth, clip=true, trim=0mm 0mm 0mm 0mm]{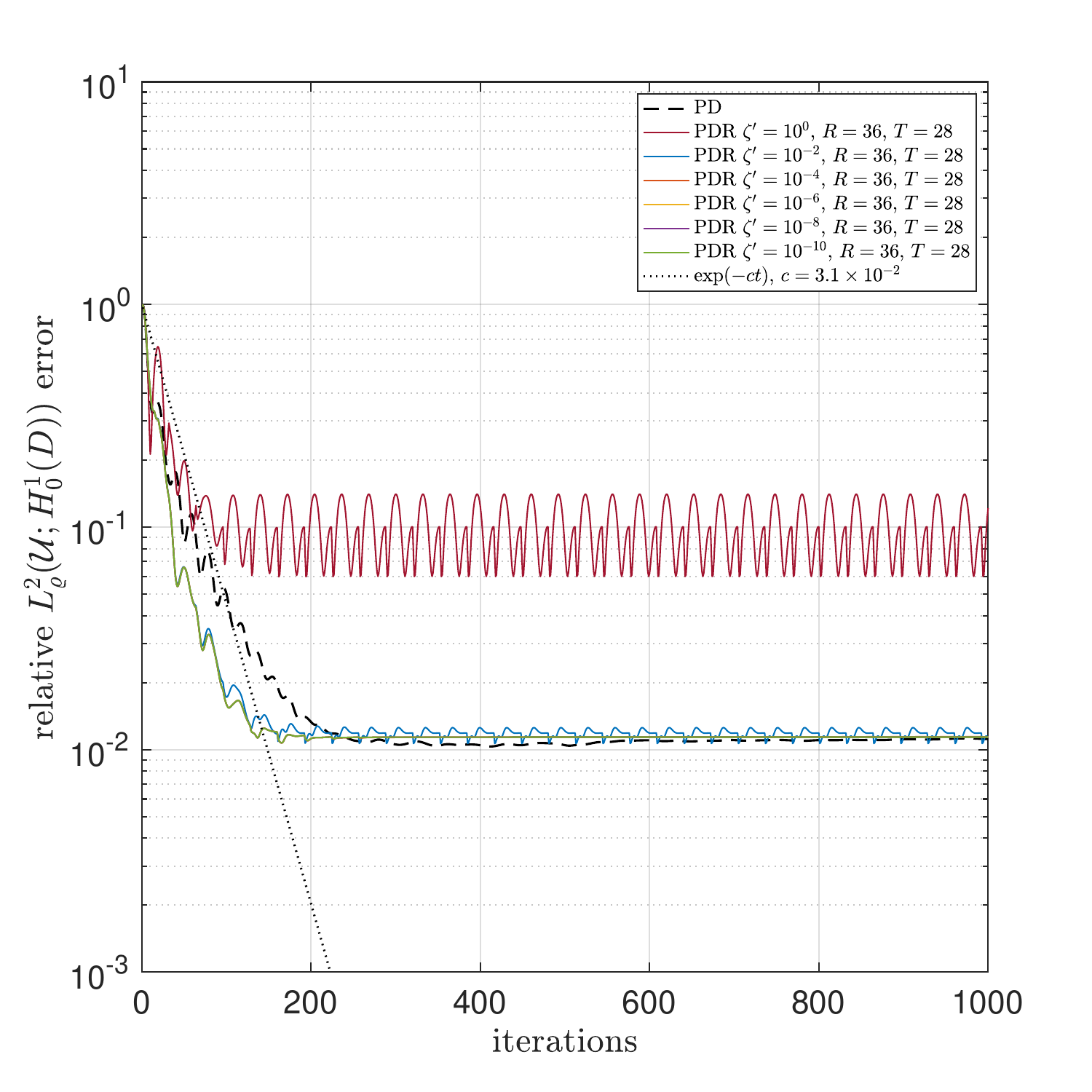}
\includegraphics[width=0.35\paperwidth, clip=true, trim=0mm 0mm 0mm 0mm]{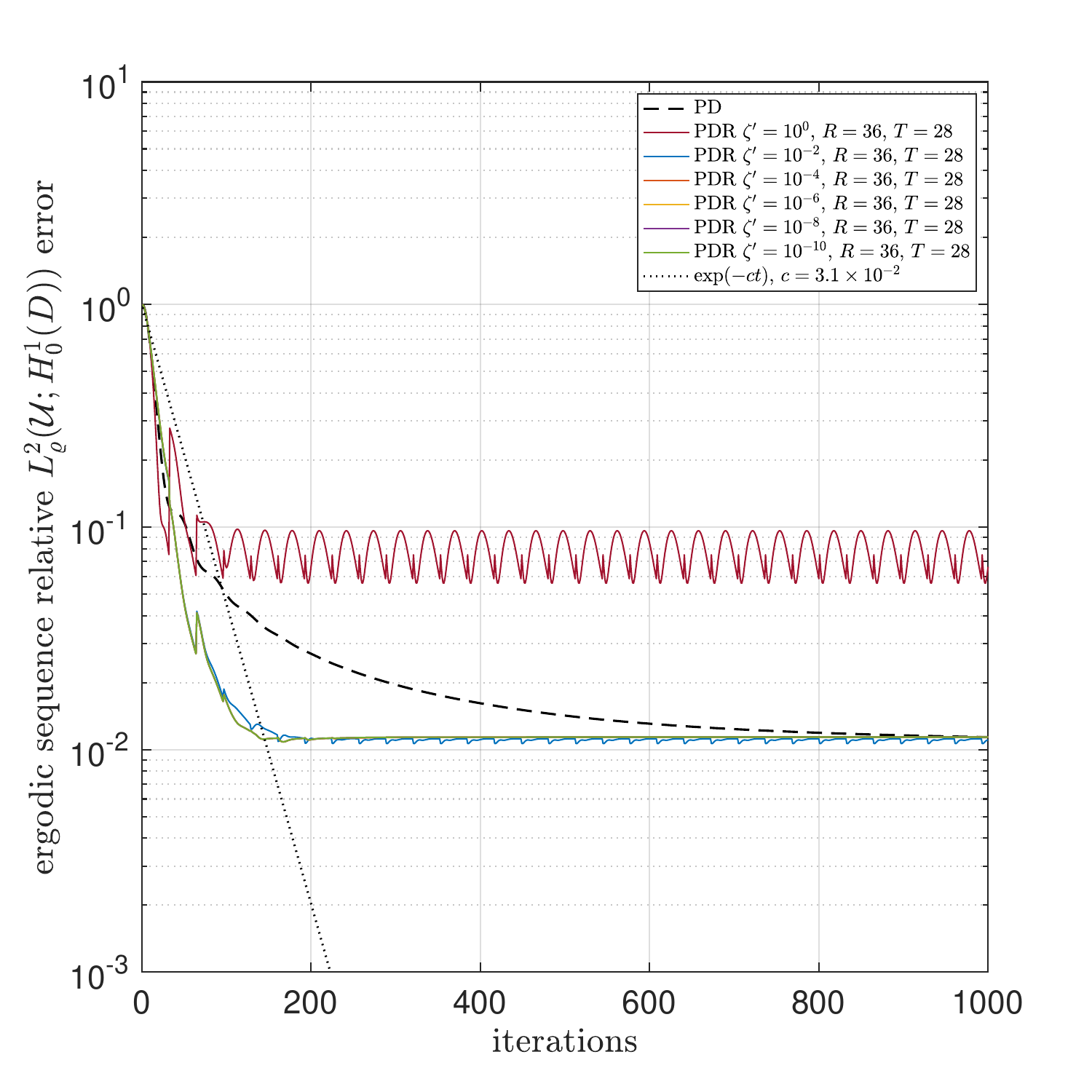}

\caption{Approximation error versus iteration number for the function $f_4$ from \R{f4-def}. This figure shows the 
relative $L^2$ errors of the polynomial approximations obtained from {\bf (left)} the iterates $\bm{c}^{(n)}$ and {\bf (right)} the ergodic sequence $\bar{\bm{c}}^{(n)}$. These approximations are constructed using the Legendre polynomial basis and $m = 1000$ sample points drawn randomly and independently from the uniform measure. The index set $\Lambda = \Lambda^{\textsf{HC}}_{n,d}$, where $d = 30$ and $n = 10$, which gives a basis of cardinality $N = |\Lambda| = 7841$. We compare the primal dual iteration ``PD'' and the restarted primal dual iteration ``PDR" for various values of the tolerance $\zeta'$. We also plot the theoretical error curve \R{theory-error-curve}, where $t$ is the iteration number. The quadrature rule used to compute the relative error is a sparse grid rule of level 3 consisting of $M = 1861$ points.}
\label{fig:HV_f4}
\end{figure}

\subsection{Numerical results 2: approximation error and run time}

In the second set of experiments, our aim is to study the approximation error versus the number of samples $m$. Having compared different solvers in the previous experiments, we now limit our attention to the restarted primal-dual iteration. The only modification we make is to introduce a stopping criterion for the number of restarts. Specifically, given a tolerance $\zeta'$, we halt the iteration if the difference between two consecutive iterates is less than $5 \cdot \zeta'$. Specifically, if 
\bes{
\nmu{\tilde{\bm{c}}^{(l)} - \tilde{\bm{c}}^{(l-1)} }_{2} \leq 5 \cdot \zeta',
}
in the scalar-valued case or
\bes{
\nmu{\tilde{\bm{c}}^{(l)} - \tilde{\bm{c}}^{(l-1)} }_{2;\cV} \leq 5 \cdot \zeta',
}
in the Hilbert-valued case, 
where $\tilde{\bm{c}}^{(l)}$ is the output of the restarted primal-dual iteration after $l$ restarts, then we halt and take $\tilde{\bm{c}}^{(l)}$ as the polynomial coefficients of the resulting approximation.

In the following experiments, we perform multiple trials for each value of $m$.  For each trial, we generate a set of sample Monte Carlo points $\bm{y}_1,\ldots,\bm{y}_m$, then compute the relative error \R{rel-error-scalar} or \R{rel-error-Hilbert} of the approximation using a sparse grid quadrature as before. Having done this, we then compute the sample mean and (corrected) sample standard deviation after a log transformation. See \cite[\S A.1.3]{adcock2021sparse} for further discussion and rationale behind this computation.

The results for the four functions $f_1,f_2,f_3,f_4$ are shown in Figures \ref{fig:SV_f1_m_comp}--\ref{fig:HV_f4_m_comp}. Figure \ref{fig:SV_f1_m_comp} shows the average approximation error and run times for $f_1$. As discussed, this function is expected to be well-approximated by polynomials. In accordance, the error decreases rapidly, achieving roughly $10^{-7}$ relative $L^2$ error when $m \approx 200$. This is in broad agreement with the exponential decay rate of the error shown in our main theorems. In Figure \ref{fig:SV_f2_m_comp} we consider the more challenging, higher-dimensional function $f_2$, plotting the average approximation error and run time. Here, as expected, the error decreases significantly more slowly.  Both figures exhibit a linear scaling of the run time with the number of samples $m$. This is consistent with our analysis, since each algorithm iteration involves dense matrix-vector multiplications with an $m \times N$ matrix. Also, comparing Figure \ref{fig:SV_f1_m_comp} and Figure \ref{fig:SV_f2_m_comp} when $m = 250$, we notice the run time is roughly 16 times larger for the latter. This is also in agreement with our analysis. Indeed, $N \approx 1000$ in Figure \ref{fig:SV_f1_m_comp} while $N \approx 8000$ in Figure \ref{fig:SV_f2_m_comp}. However, the number of inner iterations $T = \lceil 2 \nm{\bm{A}}_2 / r \rceil$ is roughly twice as large in Figure \ref{fig:SV_f2_m_comp}, where $\nm{\bm{A}}_2 \approx 13$ when $m = 250$, as it is in Figure \ref{fig:SV_f1_m_comp}, where $\nm{\bm{A}}_2 \approx 7$. The combination of these two factors accounts for the roughly 16-fold increase in run time.

Figure \ref{fig:HV_f3_m_comp} displays the performance of the restarted scheme on the Hilbert-valued function $f_3$. Here we also observe rapid decrease in the error with respect to increasing number of samples $m$, with relative $L^2$ error approximately $10^{-6}$ when $m \approx 200$. Finally, Figure \ref{fig:HV_f4_m_comp} shows the results for the less smooth high-dimensional Hilbert-valued function $f_4$. For this function, we expect slower decrease in the error with respect to $m$, which is reflected in this set of results. Nonetheless, despite its high dimensionality ($d = 30$) we still achieve two digits of relative accuracy using only $m \approx 1000$ samples.

\begin{figure}
\centering
\includegraphics[width=0.35\paperwidth, clip=true, trim=0mm 0mm 0mm 0mm]{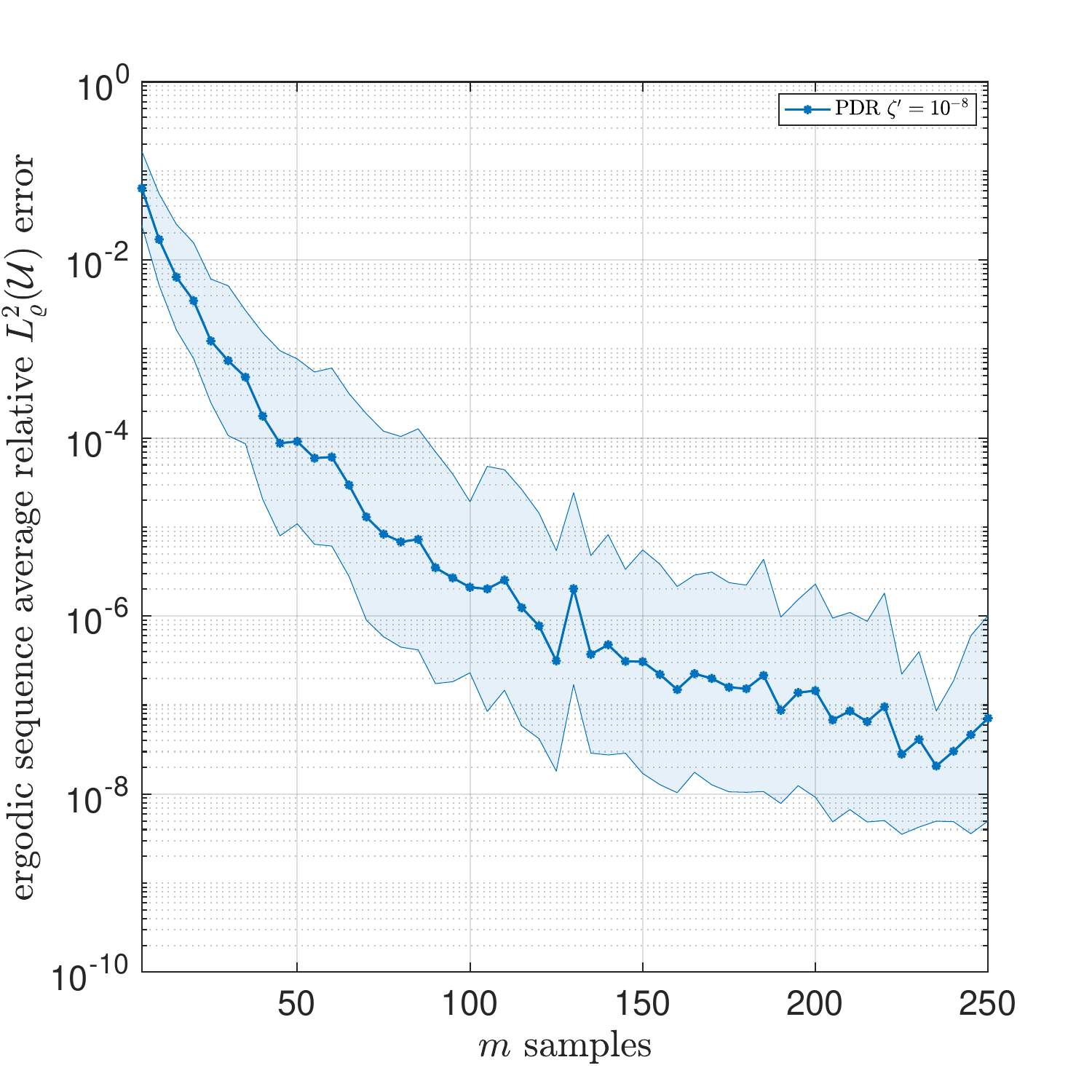} 
\includegraphics[width=0.35\paperwidth, clip=true, trim=0mm 0mm 0mm 0mm]{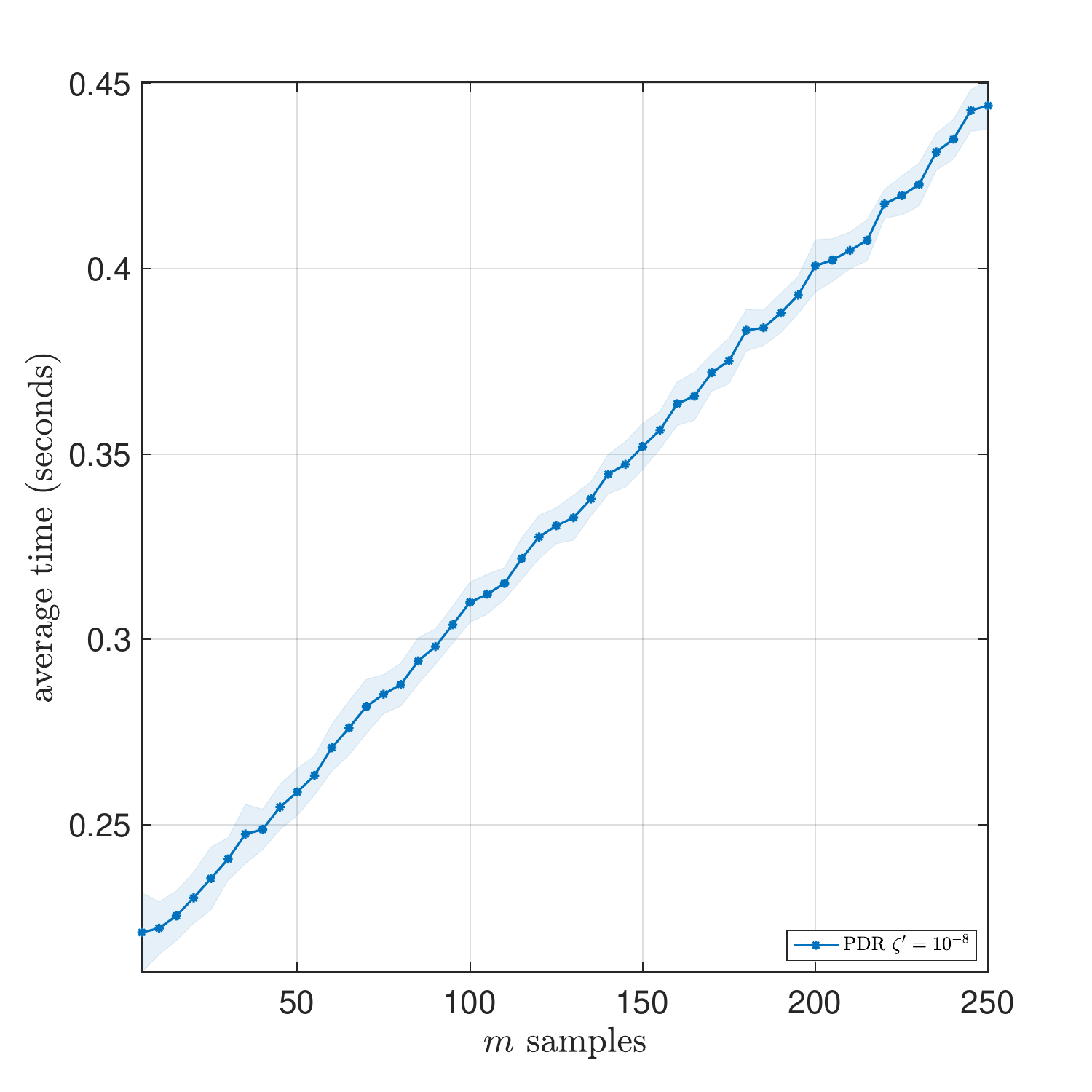}
\caption{\textbf{(left)} Approximation error and \textbf{(right)} average run time versus number of samples $m$ for the function $f_1$ from \R{f1-def}. This figure shows the 
relative $L^2$ errors of the polynomial approximations obtained from the ergodic sequence $\bar{\bm{c}}^{(n)}$. These approximations are constructed using the Legendre polynomial basis and various sets of $m$ sample points drawn randomly and independently from the uniform measure for each trial. The index set $\Lambda = \Lambda^{\textsf{HC}}_{n,d}$, where $d = 2$ and $n = 184$, which gives a basis of cardinality $N = |\Lambda| = 997$. We use the restarted primal dual iteration ``PDR" with $\zeta' = 10^{-8}$, and display the average error over 50 trials measured in the sample mean in blue and the corrected sample standard deviation after a log transformation in shaded blue, see \cite[Appendix A.1.3]{adcock2021sparse} for more details.
The quadrature rule used to compute the relative error is a sparse grid rule of level 11 consisting of $M = 7169$ points.
}
\label{fig:SV_f1_m_comp}
\end{figure}

\begin{figure}
\centering
\includegraphics[width=0.35\paperwidth, clip=true, trim=0mm 0mm 0mm 0mm]{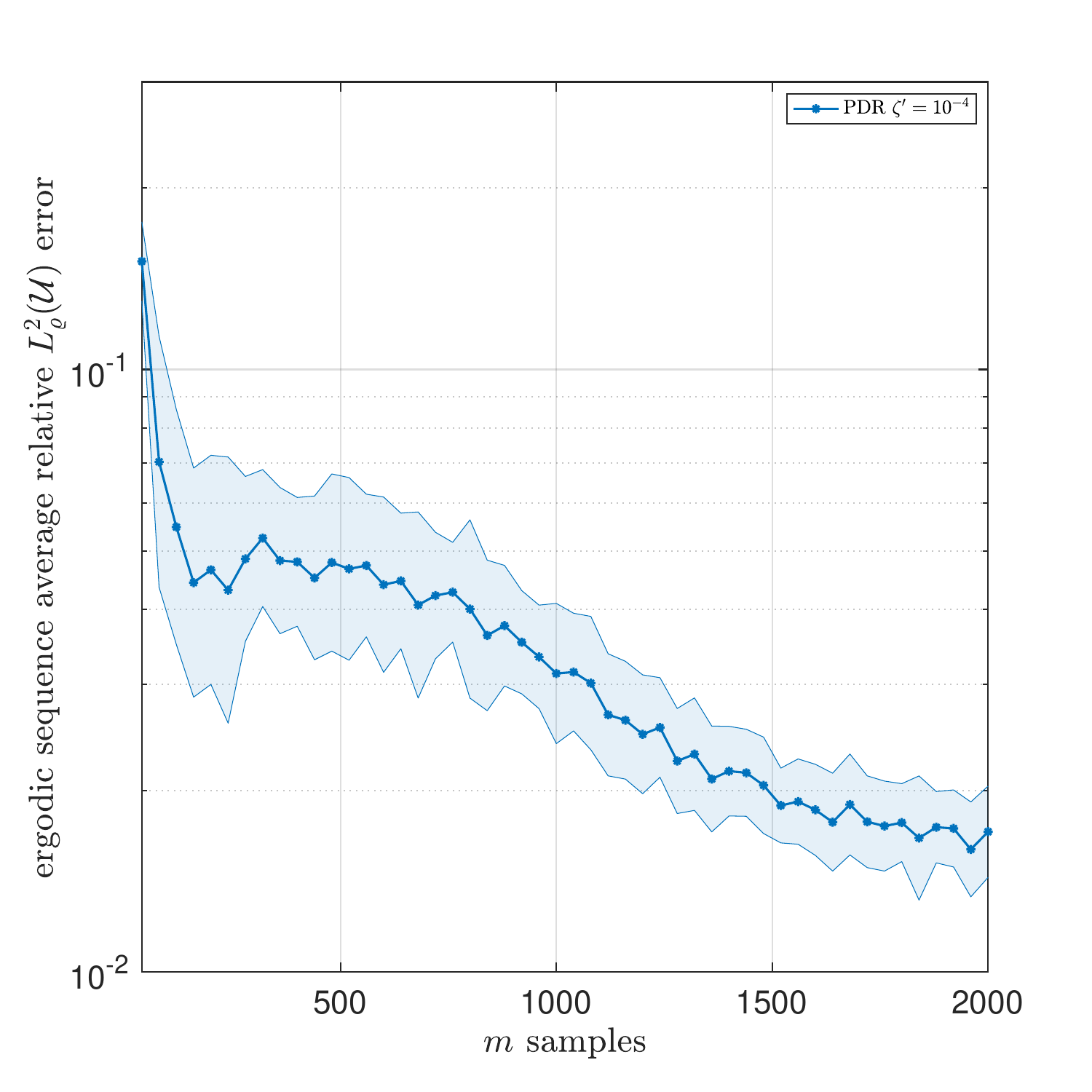}
\includegraphics[width=0.35\paperwidth, clip=true, trim=0mm 0mm 0mm 0mm]{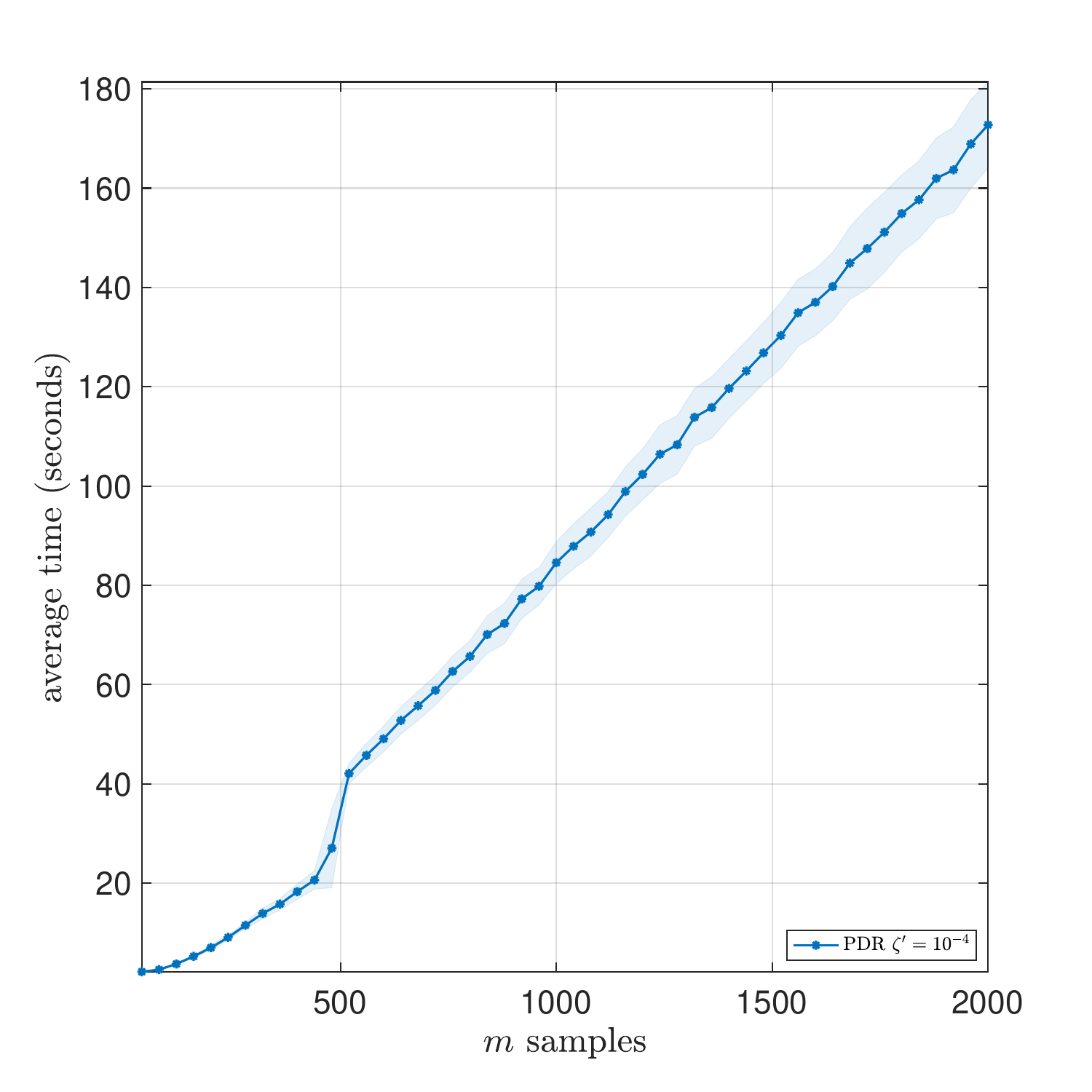}
\caption{\textbf{(left)} Approximation error and \textbf{(right)} average run time versus number of samples $m$ for the function $f_2$ from \R{f2-def}. This figure shows the 
relative $L^2$ errors of the polynomial approximations obtained from the ergodic sequence $\bar{\bm{c}}^{(n)}$. These approximations are constructed using the Legendre polynomial basis and various sets of $m$ sample points drawn randomly and independently from the uniform measure for each trial. The index set $\Lambda = \Lambda^{\textsf{HC}}_{n,d}$, where $d = 16$ and $n = 16$, which gives a basis of cardinality $N = |\Lambda| = 8277$. We compare the restarted primal dual iteration ``PDR" with $\zeta' = 10^{-4}$ with the average performance over 50 trials measured in the sample mean in blue and the corrected sample standard deviation after a log transformation in shaded blue, see \cite[Appendix A.1.3]{adcock2021sparse} for more details.
The quadrature rule used to compute the relative error is a sparse grid rule of level 5 consisting of $M = 51137$ points.}
\label{fig:SV_f2_m_comp}
\end{figure}

\begin{figure}
\centering
\includegraphics[width=0.5\paperwidth, clip=true, trim=0mm 0mm 0mm 0mm]{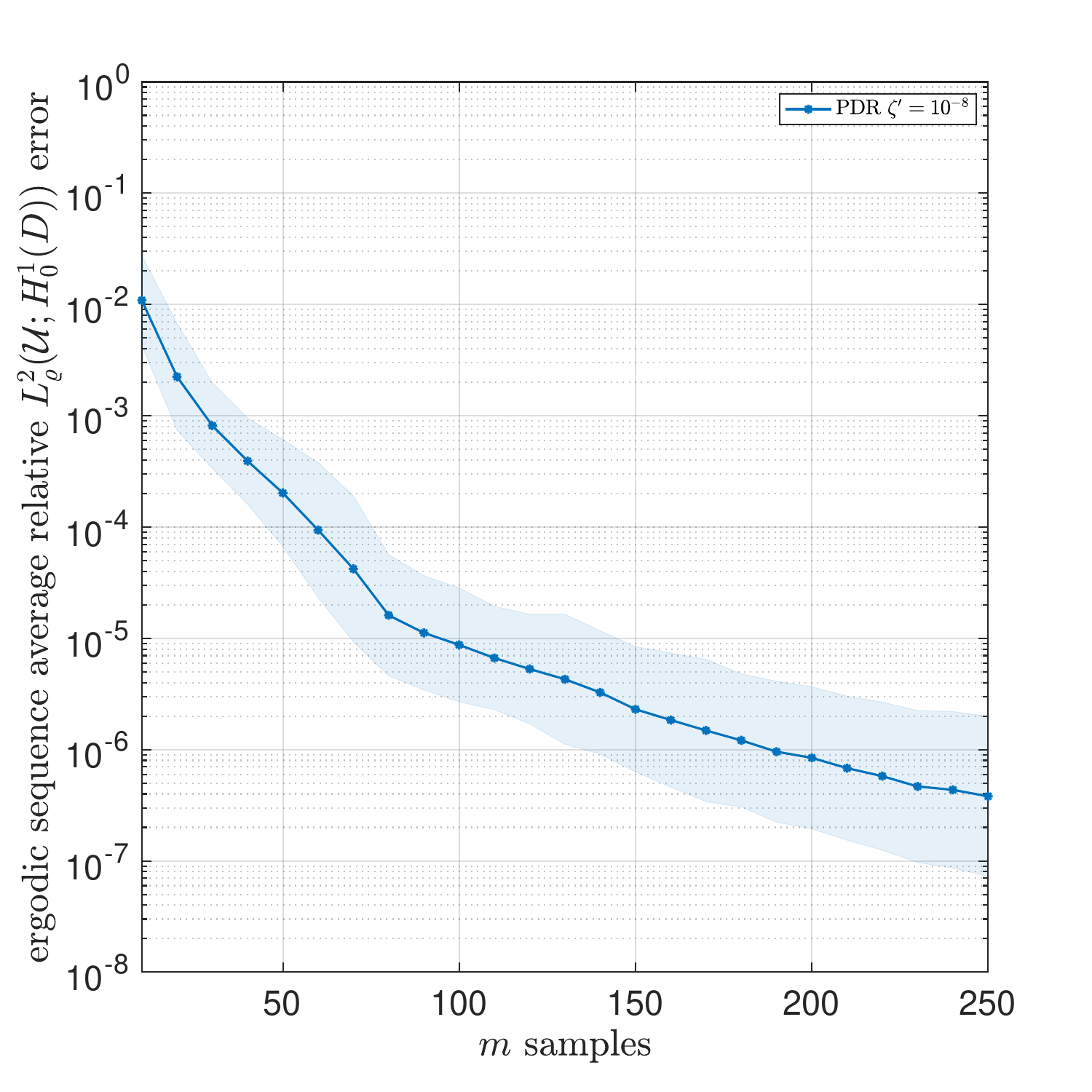}
\caption{Approximation error versus number of samples $m$ for the function $f_3$ from \R{f3-def}. This figure shows the 
relative $L^2$ errors of the polynomial approximations obtained from the ergodic sequence $\bar{\bm{c}}^{(n)}$. These approximations are constructed using the Legendre polynomial basis and various sets of $m$ sample points drawn randomly and independently from the uniform measure for each trial. The index set $\Lambda = \Lambda^{\textsf{HC}}_{n,d}$, where $d = 2$ and $n = 184$, which gives a basis of cardinality $N = |\Lambda| = 997$. We compare the restarted primal dual iteration ``PDR" with $\zeta' = 10^{-8}$ with the average performance over 50 trials measured in the sample mean in blue and the corrected sample standard deviation after a log transformation in shaded blue, see \cite[Appendix A.1.3]{adcock2021sparse} for more details.
The quadrature rule used to compute the relative error is a sparse grid rule of level 11 consisting of $M = 7169$ points.}
\label{fig:HV_f3_m_comp}
\end{figure}

\begin{figure}
\centering
\includegraphics[width=0.5\paperwidth, clip=true, trim=0mm 0mm 0mm 0mm]{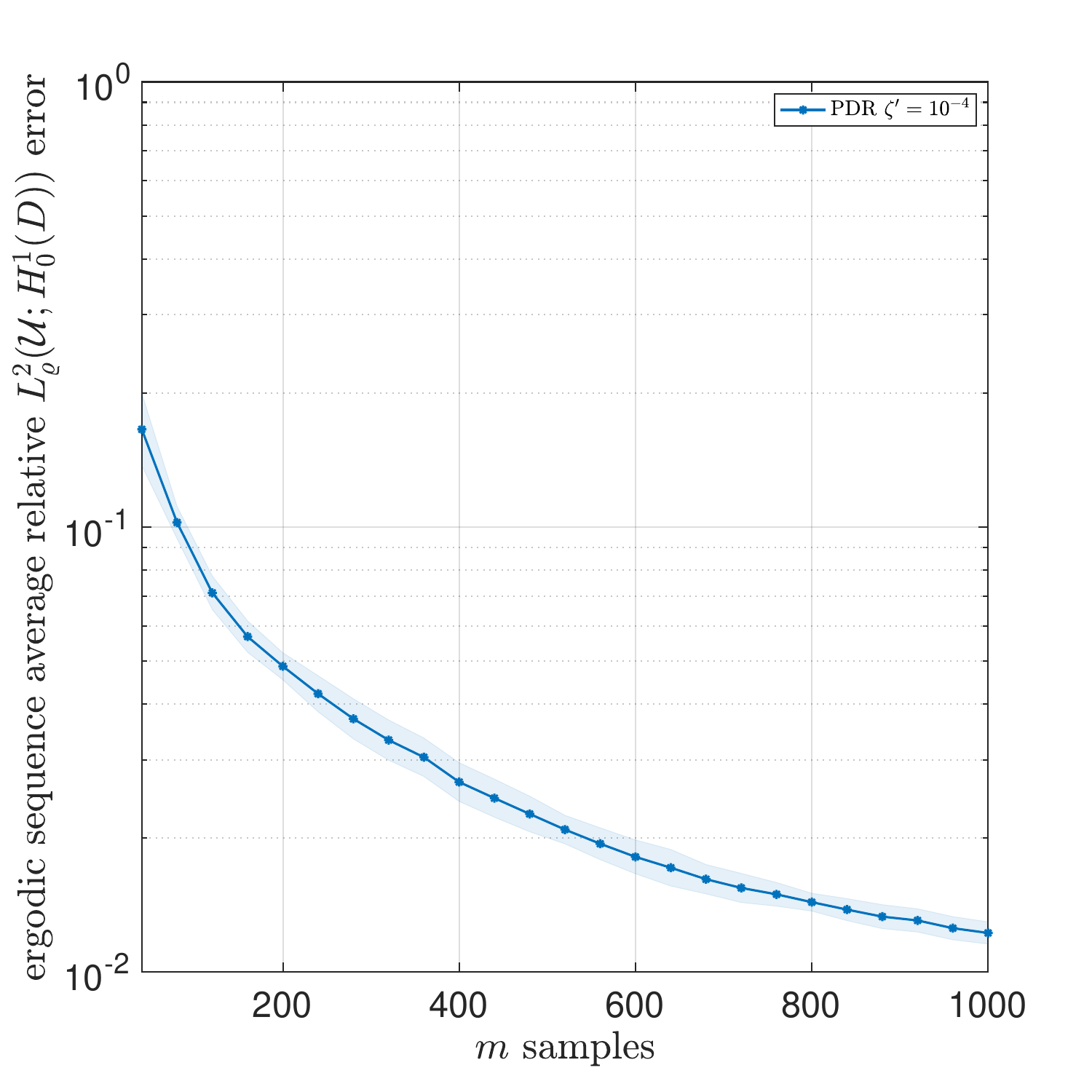}
\caption{Approximation error versus number of samples $m$ for the function $f_4$ from \R{f4-def}. This figure shows the 
relative $L^2$ errors of the polynomial approximations obtained from the ergodic sequence $\bar{\bm{c}}^{(n)}$. These approximations are constructed using the Legendre polynomial basis and various sets of $m$ sample points drawn randomly and independently from the uniform measure for each trial. The index set $\Lambda = \Lambda^{\textsf{HC}}_{n,d}$, where $d = 30$ and $n = 10$, which gives a basis of cardinality $N = |\Lambda| = 7841$. We compare the restarted primal dual iteration ``PDR" with $\zeta' = 10^{-4}$ with the average performance over 50 trials measured in the sample mean in blue and the corrected sample standard deviation after a log transformation in shaded blue, see \cite[Appendix A.1.3]{adcock2021sparse} for more details.
The quadrature rule used to compute the relative error is a sparse grid rule of level 3 consisting of $M = 1861$ points.}
\label{fig:HV_f4_m_comp}
\end{figure}

\section{Overview of the proofs}\label{s:proof-overview}

The rest of this paper is devoted to proving the main results. Since these involve a number of technical steps, we now give a brief overview of how these proofs proceed.

We commence in \S \ref{ss:HCS_recovery} by developing compressed sensing theory for Hilbert-valued vectors. We introduced the so-called \textit{weighted robust Null Space Property (rNSP) over $\cV$}, and then show in Lemma \ref{lemma-musuboptimal} that it implies certain error bounds for inexact minimizers of the Hilbert-valued, weighted SR-LASSO problem. Next, we introduced the \textit{weighted Restricted Isometry Property (RIP)} and then in Lemma \ref{implies-RIP} we show that this property over $\bbC$ implies the weighted rNSP over $\cV$.

In \S \ref{s:poly-app-CS} we focus on the polynomial approximation problem. We first give a sufficient condition in terms of $m$ for the measurement matrix \R{def-measMatrix} to satisfy the weighted RIP with high probability (Lemma \ref{l:LegMat_RIP}). Next, we state and prove three general results (Theorems \ref{t:main-res-map-alg_inex}--\ref{t:main-res-map-alg_inex_exponential}) that give error bounds for polynomial approximations obtained as inexact minimizers of the Hilbert-valued, weighted SR-LASSO problem. These results are split into the three cases considered in our main results, i.e., the algebraic and finite-dimensional case, the algebraic and infinite-dimensional case, and the exponential case. The error bounds in these results split into terms corresponding to the polynomial approximation error, the physical discretization error, the sampling error, and the error in the objective function at the inexact minimizer.

With this in mind, in the next section, \S \ref{s:errbds_PDI}, we first present error bounds for inexact minimizers obtained by finitely-many iterations of the primal-dual iteration. See Lemma \ref{lemma-subopti}. Having done this, we then have the ingredients needed to derive the restarting scheme. We derive this scheme and present an error bound for it in Theorem \ref{thm:restart-scheme-bound}.

We conclude with in \S \ref{s:final-args} with the final arguments. We use the three key theorems (Theorems \ref{t:main-res-map-alg_inex}--\ref{t:main-res-map-alg_inex_exponential}) and then proceed to estimate each of the aforementioned error terms. For the polynomial approximation error we applied to several results that are given in Appendix \ref{s:best-poly-rates}. For the error in the objective function we use the results shown in \S \ref{s:errbds_PDI}. After straightforwardly bounding the other two error terms, we finally obtain the main results.

\section{Hilbert-valued compressed sensing}\label{ss:HCS_recovery}

In this section, we develop Hilbert-valued compressed sensing theory. Here, rather than the classical setting of a vector in $\bbC^N$, one seeks to recover an Hilbert-valued vector in $\cV^N$.  This was considered in \cite{dexter2019mixed} in the for the classical sparsity model with $\ell^1$-minimization. Here, we consider the weighted sparsity model and weighted $\ell^1$-minimization. This model was first developed in \cite{rauhut2016interpolation}. See also \cite{adcock2018infinite,chkifa2018polynomial} and \cite[Chpt.\ 6]{adcock2021sparse}. Note that in this section, we shall write $\cV$ rather than $\cV_h$, as is done in \R{wsr-LASSO}. Of course, all the results shown below for $\cV$ will apply in the case of $\cV_h$.

\subsection{Weighted sparsity and weighted best approximation}

Let $\Lambda \subseteq \cF$ and $\w = (w_{\bm{\nu}})_{\bm{\nu} \in \Lambda}  > \bm{0}$ be positive weights. Given a set $S \subseteq \Lambda$, we define its weighted cardinality as
\bes{
|S|_{\w} : = \sum_{i \in S} w^2_i.
}
The following two definitions extend Definitions \ref{d:sparsity} and \ref{def:best_s_term} to the weighted setting:

\defn{[Weighted sparsity]
\label{def:weighted_sparsity}\index{sparsity!weighted}
Let $\Lambda \subseteq \cF$. A $\cV$-valued sequence $\bm{c} = (c_{\bm{\nu}})_{\bm{\nu} \in \Lambda}$ is \textit{weighted $(k,\bm{w})$-sparse} for some $k \geq 0$ and weights $\bm{w} = (w_{\bm{\nu}})_{\bm{\nu} \in \Lambda} > \bm{0}$ if 
\bes{
| \supp(\bm{c}) |_{\bm{w}} \leq k,
}
where $\supp(\bm{z}) = \{ \bm{\nu} : \nm{z_{\bm{\nu}}}_{\cV} \neq 0 \}$ is the \textit{support} of $\bm{z}$. The set of such vectors is denoted by $\Sigma_{k,\bm{w}}$.
}

\defn{[Weighted best $(k,\bm{w})$-term approximation error]
\label{def:best_s_term_weighted}
Let $\Lambda \subseteq \cF$ $0 < p \leq 2$, $\bm{w} > \bm{0}$, $\bm{c} \in \ell^p_{\bm{w}}(\Lambda;\cV)$ and $k \geq 0$. The \textit{$\ell^p_{\bm{w}}$-norm weighted best $(k,\bm{w})$-term approximation error} of $\bm{c}$ is
\be{
\label{weighted_best_s_term}
\sigma_{k}(\bm{c})_{p,\bm{w};\cV} = \min \left \{ \nm{\bm{c} - \bm{z}}_{p,\bm{w};\cV} : \bm{z} \in \Sigma_{k,\bm{w}} \right \}.
}
}
Notice that this is equivalent to
\be{
\label{sigma-k-w-equiv}
\sigma_{k}(\bm{c})_{p,\w;\cV} = \inf \left \{ \nm{\bm{c} - \bm{c}_S}_{p,\w ; \cV} : S \subseteq \Lambda,\ |S|_{\w} \leq k \right \}.
}
Here and elsewhere, for a sequence $\bm{c} = (c_{\bm{\nu}})_{\bm{\nu} \in \Lambda}$ and a set $S \subseteq \Lambda$, we define $\bm{c}_{S}$ as the sequence with $\bm{\nu}$th entry equal to $c_{\bm{\nu}}$ if $\bm{\nu} \in S$ and zero otherwise.

\subsection{The weighted robust null space property}
 
For the rest of this section, we consider the index set $\Lambda = \{1,\ldots,N\}$ for some $N \in \bbN$. Our analysis of the weighted SR-LASSO problem is presented in terms of the so-called weighted robust null space property.
Let $\w > \bm{0}$ and $k > 0$. A bounded linear operator $\bm{A} \in \cB(\cV^N,\cV^m)$ has the \textit{weighted robust Null Space Property (rNSP) over $\cV$ of order $(k,\bm{w})$  with constants $0<\rho<1$ and $\gamma>0$} if 
\begin{equation*}
\nm{\x_S}_{2;\cV} \leq  \dfrac{\rho \nm{\x_{S^c}}_{1,\w;\cV} }{\sqrt{k}}+\gamma \nm{\bm{A} \x}_{2;\cV} ,\quad \forall \x \in \cVN,
\end{equation*}
for any $S \subseteq [N]$ with $|S|_{\bm{w}} \leq k$.

Importantly, the weighted rNSP implies distance bounds in the $\ell^1_{\bm{w}}$- and $\ell^2$-norms. The following lemma is standard in the scalar case (see, e.g., \cite[Lem.\ 6.24]{adcock2021sparse}). We omit the proof of its extension to the Hilbert-valued case, since it follows almost exactly the same arguments. 

\begin{lemma}[Weighted rNSP implies $\ell^1_{\bm{w}}$ and $\ell^2$ distance bounds]\label{lemma-wrNSP-l1}
Suppose that $\bm{A} \in \cB(\cV^N,\cV^m)$ has the weighted rNSP over $\cV$ of order $(k,\w)$ with constants  $0<\rho<1$ and $\gamma>0$. Let $\x,\z \in \cVN$. Then 
\begin{equation}\label{w1nm-bound}
\nm{\z-\x}_{1,\w;\cV}\leq  C_1 \left( 2\sigma_{k}(\x)_{1,\w;\cV} +\nm{\z}_{1,\w;\cV}-\nm{\x}_{1,\w;\cV}\right)+ C_2\sqrt{k}  \nm{\bm{A}(\z-\x)}_{2;\cV},
\end{equation}
\begin{equation}\label{boundl2}
\begin{split}
\nm{\z-\x}_{2;\cV} \leq   & \dfrac{C_1'}{\sqrt{k}}  \left( 2\sigma_{k}(\x)_{1,\w;\cV} +\nm{\z}_{1,\w;\cV}-\nm{\x}_{1,\w;\cV}\right) +C_2' \nm{\bm{A}(\z-\x)}_{2;\cV},
\end{split}
\end{equation} 
where the constants are given by
\begin{equation*}
C_1 = \dfrac{(1+\rho)}{(1-\rho)}, \quad C_2 =  \dfrac{2\gamma }{(1-\rho)}, \quad C_1' = \left(\dfrac{(1+\rho)^2}{1-\rho} \right) \quad \text{ and } \quad C_2' = \left(\dfrac{(3+\rho) \gamma}{1-\rho}\right) .
\end{equation*}

\end{lemma}

Lemma \ref{lemma-wrNSP-l1} can be used to show distance bounds for exact minimizers of the Hilbert-valued weighted SR-LASSO problem
\be{
\label{SRLASSO-CS-sec}
\min_{\z \in \cV^N} \cG(\bm{z}),\qquad \cG(\bm{z}) : = \lambda \nm{\z}_{1,\w;\cV} + \nmu{\bm{A} \bm{z} - \bm{b}}_{2;\cV}.
}
Fortunately, it also implies bounds for approximate minimizers, such as those obtained by a finite number of steps of the primal-dual iteration.

\begin{lemma}
[Weighted rNSP implies error bounds for inexact minimizers]
\label{lemma-musuboptimal}
Suppose that $\bm{A} \in \cB(\cV^N,\cV^m)$ has the weighted rNSP  over $\cV$ of order $(k,\w)$ with constants  $0<\rho<1$ and $\gamma>0$. Let $\x \in \cVN$, $\b \in \cVM$ and $\bm{e} = \bm{A}\x-\b \in \cVM$, and consider the problem \R{SRLASSO-CS-sec} with parameter
\be{
\label{lambda-bound-rNSP-err}
0 < \lambda \leq  \dfrac{(1+ \rho)^2}{(3+\rho) \gamma } {k}^{-1/2}.
}
Then, for any $\tilde{\x} \in \cVN$, 
\eas{
\nm{\tilde{\x}-\x}_{1,\w;\cV} & \leq C_1 \left( 2 \sigma_{k}(\x)_{1,\w;\cV}
+\frac{\cG(\tilde{\x}) - \cG(\x) }{\lambda}  \right)+ \left( \dfrac{C_1}{\lambda} +C_2 \sqrt{k} \right) \nm{\e}_{2;\cV} ,
\\
\nm{\tilde{\x}-\x}_{2;\cV} & \leq  \dfrac{C_1'}{\sqrt{k}} \left( 2 {\sigma_{k}(\x)_{1,\w;\cV}} 
+\dfrac{\cG(\tilde{\x}) - \cG(\x) }{ \lambda} \right) + \left( \dfrac{C'_1}{\sqrt{k}\lambda} + C'_2 \right)\nm{\e}_{2;\cV} ,
}
where $C_1 $, $C_2 $, $C'_1$ and  $C'_2$ are as in Lemma \ref{lemma-wrNSP-l1}. 
\end{lemma}
\begin{proof}
First notice that $C'_1/C'_2 \leq C_1/C_2$ since $0 < \rho < 1$, where $C_1 $, $C_2 $, $C'_1$ and  $C'_2$ are as in Lemma \ref{lemma-wrNSP-l1}. Hence the condition on $\lambda$ implies that
\begin{equation}
\lambda \leq \min \lbrace C_1/C_2, C_1'/C_2'\rbrace k ^{-1/2},
\end{equation}
Using this lemma and this bound, we deduce that
\begin{equation*}
\nm{\tilde{\x}-\x}_{1,\w;\cV}\leq  2C_1 \sigma_{k}(\x)_{1,\w;\cV} + \dfrac{C_1}{\lambda}\left( \lambda \nm{\tilde{\x}}_{1,\w;\cV}+\nm{\bm{A}\tilde{\x}-\b}_{2;\cV}- \lambda\nm{\x}_{1,\w;\cV}\right)+C_2 \sqrt{K} \nm{\e}_{2;\cV}.
\end{equation*}
The definition of $\cG$ in \eqref{SRLASSO-CS-sec} gives
\begin{equation*}
\nm{\tilde{\x}-\x}_{1,\w;\cV}\leq  2C_1 \sigma_{k}(\x)_{1,\w;\cV} + \dfrac{C_1}{\lambda}\left(  \cG( \tilde{\x}) -\cG(  {\x})+ \nm{\e}_{2;\cV} \right)+C_2 \sqrt{k} \nm{\e}_{2;\cV},
\end{equation*}
which is the first result. The second follows in an analogous manner.
\end{proof}

\subsection{The weighted rNSP and weighted restricted isometry property}

In the next section, we give explicit conditions in terms of $m$ under which the measurement matrices \R{def-measMatrix} satisfy the weighted rNSP over $\cV$. It is well known that showing the (weighted) rNSP directly can be difficult. In the classical, scalar setting, this is overcome by showing that the (weighted) rNSP is implied by the so-called (weighted) restricted isometry property. Hence, in this subsection, we first introduced this property and describe its relation to the (weighted) rNSP.

Let $\w > \bm{0}$ and $k > 0$. A bounded linear operator $\bm{A} \in \cB(\cV^N,\cV^m)$ has the \textit{weighted Restricted Isometry Property (RIP)} over $\cV$ of order $(k,\w)$ if there exists a constant $0<\delta <1$ such that
\begin{equation}
(1-\delta) \nm{\z}_{2;\cV}^2 \leq \nm{\bm{A} \z}_{2;\cV}^2 \leq (1+\delta) \nm{\z}_{2;\cV}^2,\quad \forall \z \in \Sigma_{k,\w} \subseteq \cV^N.
\end{equation}
 The smallest constant such that this property holds is called  the $(k,\bm{w})$th \textit{weighted Restricted Isometry Constant (wRIC)} of $\bm{A}$, and is denoted as $\delta_{k,\w}$.

It is first convenient to show an equivalence between the scalar weighted RIP over $\bbC$ and the Hilbert-valued weighted RIP over $\cV$.

\begin{lemma}
[weighted RIP over $\bbC$ is equivalent to the weighted RIP over $\cV$]
\label{l:from_R_to_V_rNSP}
Let $\w > \bm{0}$, $k > 0$ and $\bm{A} = (a_{ij})^{m,N}_{i,j=1} \in \bbC^{m \times N}$ be a matrix. Then $\bm{A}$ satisfies the weighted RIP over  $\bbC$  of order $(k,\w)$ with constant  $0<\delta<1$ if and only if the corresponding bounded linear operator  $\bm{A} \in \cB (\cV^N ,\cV^m)$ defined by
\bes{
\x = (x_i)^{N}_{i=1} \in \cV^N \mapsto \bm{A}\bm{x} : = \left ( \sum^{N}_{i=1} a_{ij} x_j \right )^{m}_{i=1} \in \cV^m,
}
satisfies the weighted RIP over $\cV$ of order $(k,\w)$ with the same constant $\delta$.
\end{lemma}
\begin{proof}
We follow similar arguments to \cite[Rmk.\ 3.5]{dexter2019mixed}. First, we rewrite the equivalence as follows:
\begin{equation}\label{wRIP_V}
(1-\delta)\nm{\x}_{2;\cV}^2 \leq \nm{\bm{A}\x}_{2;\cV}^2 \leq (1+\delta) \nm{\x}_{2;\cV}^2, \qquad \forall \x \in \cV^N, |\supp(\x)|_{\w} \leq k,
\end{equation}
if and only if 
\begin{equation}\label{wRIP_R}
(1-\delta)\nm{\x}_{2}^2 \leq \nm{\bm{A}\x}_{2}^2 \leq (1+\delta) \nm{\x}_{2}^2, \qquad \forall \x \in \bbC^N, |\supp(\x)|_{\w} \leq k.
\end{equation}
Suppose that \eqref{wRIP_R} holds. Let $\x = (x_j)^{N}_{i=1} \in \cV^N$ be $(k,\w)$-sparse and $\lbrace \phi_i\rbrace_{i}$ be an orthonormal basis of $\cV$. Then, for each $i \in [N]$, $x_i \in \cV$ can be uniquely represented as
\begin{equation*}
x_i= \sum_{j } \alpha_{ij} \phi_j, \quad \alpha_{ij} \in \bbC.
\end{equation*}
Let $\bm{x}_j = (\alpha_{ij})^{N}_{i=1} \in \bbC^N$. Then $\supp(\bm{x}_j) \subseteq \supp(\bm{x})$ and therefore $\bm{x}_j$ is $(k,\bm{w})$-sparse. Hence \R{wRIP_R} gives
\be{
\label{wRIP-cols}
(1-\delta)\nmu{\x_j}_{2}^2 \leq \nmu{\bm{A} \x_j}_{2}^2 \leq (1+\delta) \nmu{\x_j}_{2}^2 .
}
Now observe that
\bes{
\sum_{j } \nm{\bm{x}_j}^2_2 = \sum^{N}_{i=1} \sum_{j } | \alpha_{ij} |^2 = \sum^{N}_{i=1} \nm{x_i}^2_{\cV} = \nm{\bm{x}}^2_{2 ; \cV},
}
and
\bes{
\sum_{j } \nm{\bm{A} \bm{x}_j}^2_2 = \sum_{j } \sum^{m}_{i=1} \left | \sum^{N}_{k = 1} a_{ik} \alpha_{kj} \right |^2 = \sum^{m}_{i=1} \nm{\sum^{N}_{k=1} a_{ik} x_k }^2_{\cV} = \nm{\bm{A} \bm{x}}^2_{2 ; \cV}.
}
Summing \R{wRIP-cols} over $j $, we deduce that \R{wRIP_V} holds.

Conversely, suppose that \R{wRIP_V} holds and let $\bm{z} = (z_i)^{N}_{i=1} \in \bbC^N$ with $|\supp(\z)|_{\w} \leq k$. Define $\bm{x} = (z_i \phi_i ) \in \cV^N$ and notice that $\nm{\bm{x}}_{2;\cV} = \nm{\bm{z}}_2$ and $\nm{\bm{A} \bm{x}}_{2;\cV} = \nm{\bm{A} \bm{z}}_{2}$. Since $\supp(\x) = \supp(\bm{z})$ and $|\supp(\z)|_{\w} \leq k$, we now apply \R{wRIP_V} to deduce that $(1-\delta) \nm{\bm{z}}^2_2 \leq \nm{\bm{A} \bm{z}}^2_2 \leq (1+\delta) \nm{\bm{z}}^2_2$. We conclude that \R{wRIP_R} holds.
\end{proof}

The following result shows that the weighted RIP is a sufficient condition for the weighted rNSP. This result is well known in the scalar-valued case (see, e.g., \cite[Theorem 6.26]{adcock2021sparse}). Since its extension to the Hilbert-valued case is straightforward, we omit the proof.
\begin{lemma}[weighted RIP implies the weighted rNSP]
\label{implies-RIP} Let $\w > \bm{0}$, $k > 0$ and suppose that $\bm{A} \in \bbC^{m\times N}$ has the weighted RIP over $\bbC$ of order $(2k,\w)$ with constant $\delta_{2k,\w}< (2 \sqrt{2}-1)/7$. Then $\bm{A}$ has the weighted rNSP of order $(k,\w)$ over  $\cV$ with constants $\rho=2 \sqrt{2} \delta_{2k,\w} / (1-\delta_{2k,\w})$ and $\gamma=\sqrt{1+\delta_{2k,\w}} / (1-\delta_{2k,\w})$. 
\end{lemma}

\section{Error bounds for polynomial approximation via the Hilbert-valued, weighted SR-LASSO}\label{s:poly-app-CS}
 
Having developed the necessary tools for compressed sensing in the Hilbert-valued setting, we now specialize to the case introduced in \S \ref{SRecovery} of polynomial approximation via the Hilbert-valued, weighted SR-LASSO problem \R{wsr-LASSO}. Our main results in this section, Theorems \ref{t:main-res-map-alg_inex}--\ref{t:main-res-map-alg_inex_exponential}, yield error bounds for (inexact) minimizers of this problem in terms of the best polynomial approximation error, the Hilbert space discretization error and the noise.

\subsection{The weighted RIP for the polynomial approximation problem}

In this subsection, we assert conditions on $m$ under which the relevant measurement matrix satisfies the weighted RIP. As in \S \ref{SRecovery}, we let $\{ \Psi_{\bm{\nu}} \}_{\bm{\nu} \in \cF} \subset L^2_{\varrho}(\cU)$ be either the tensor Chebyshev or Legendre polynomial basis, 
\be{
\label{Lambda-choice}
\Lambda = \begin{cases} \Lambda^{\mathsf{HC}}_{n,d} & d < \infty, \\ \Lambda^{\mathsf{HCI}}_{n} & d = \infty, \end{cases}
}
be the hyperbolic cross index set and draw $\bm{y}_1,\ldots,\bm{y}_m$ independently and identically from the measure $\varrho$. Then we define the measurement matrix $\bm{A}$ exactly as in \R{def-measMatrix}.

\begin{lemma}
[Weighted RIP for Chebyshev and Legendre polynomials]
\label{l:LegMat_RIP}
Let $\{ \Psi_{\bm{\nu}} \}_{\bm{\nu} \in \bbN^d_0}$ be the orthonormal tensor Legendre  or {Chebyshev} polynomial basis of $L^2_{\varrho}(\cU)$, $\Lambda$ be as in \R{Lambda-choice} for some $n \geq 1$ and $\bm{y}_1,\ldots,\bm{y}_m$ be drawn independently and identically from the measure $\varrho$. Let $0 < \epsilon < 1$, $k > 0$, $\u$ be the intrinsic weights \eqref{weights_def},  
\[
L' = L'(k,n,d,\epsilon):= 
\begin{cases} 
\log(2k) \cdot \left ( \log(2 k) \cdot \min \{ \log(n) + d , \log(\E d) \cdot \log(2n) \} + \log(\epsilon^{-1}) \right )  & d < \infty,
\\
\log(2k) \cdot \left ( \log(2 k) \cdot \log^2(2n) + \log(\epsilon^{-1}) \right )  & d = \infty,
\end{cases}
\]
and suppose that
\be{
\label{m-cond-for-wRIP}
m \geq c \cdot k \cdot L'(k,n,d,\epsilon),
}
where  $c>0$ is a universal constant.
Then, with probability at least $1-\epsilon$, the matrix $\bm{A}$ defined in \R{def-measMatrix} satisfies the weighted RIP of order $(k,\bm{u})$ with constant $\delta_{k,\u} \leq 1/4$.
\end{lemma}
\begin{proof}
The proof uses ideas that are now standard. The matrix $\bm{A}$ is a specific type of measurement matrix associated to the \textit{bounded orthonormal system} $\{ \Psi_{\bm{\nu}} \}_{\bm{\nu} \in \Lambda}$ (see, e.g., \cite[Sec.\ 6.4.3]{adcock2021sparse} or \cite[Chpt.\ 12]{foucart2013mathematical}). Such a matrix satisfies the weighted RIP of order $k > 0$ with constant $\delta_{k,\bm{u}} \leq \delta$ whenever
\be{
\label{BOS-wRIP}
m \geq c \cdot k \cdot \delta^{-2} \cdot \log \left ( \frac{2 k}{\delta^2} \right ) \cdot \left [ \frac{1}{\delta^4} \log \left ( \frac{2 k}{\delta^2} \right ) \cdot \log(2N) + \frac{1}{\delta} \log(\epsilon^{-1}) \right],
}
where $c > 0$ is a universal constant.
See, e.g., \cite[Thm.\ 6.27 and eqn.\ (6.36)]{adcock2021sparse} (this result is based on \cite{chkifa2018polynomial}). To obtain the result, we set $\delta = 1/4$. Hence \R{BOS-wRIP} is implied by
\bes{
m \geq c \cdot k \cdot \log(2 k) \cdot \left [ \log(2 k) \cdot \log(2 N) + \log(\epsilon^{-1}) \right ],
}
for a potentially different universal constant $c$. Next, we use \R{N_bound} (and recall that $|\Lambda^{\mathsf{HCI}}_n| = |\Lambda^{\mathsf{HC}}_{n,n} |$) to estimate
\bes{
\log(2 N) \leq c \begin{cases} \min \{ d + \log(n) , \log(2 d) \cdot \log(2 n) \} & d < \infty,
\\
\log^2(2 n) & d = \infty,
\end{cases}
}
for a potentially different universal constant. The result now follows after substituting this into the previous expression.
\end{proof}

Note that the choice of $1/4$ in this lemma is arbitrary. Any value less than $(2\sqrt{2}-1)/7 \approx 0.261$ (see Lemma \ref{implies-RIP}) will suffice.

\subsection{Bounds for polynomial approximations obtained as inexact minimizers}

We now present the main results of this section. These three results provide error bounds for polynomial approximations obtained as (inexact) minimizers to the weighted SR-LASSO problem \R{wsr-LASSO}. Each theorem corresponds to one of the three scenarios in our main results in \S \ref{Sec_Main_results}. Hence, we label them accordingly as algebraic and finite dimensional, algebraic and infinite dimensional, and exponential. In order to state these results, we now define some additional notation. Given $f \in L^2_{\varrho}(\cU;\cV)$ and $\Lambda \subseteq \cF$, where $\cF$ is as in \R{Fdef-1}--\R{Fdef-2}, we let
\bes{
E_{\Lambda,2}(f) = \nmu{f - f_{\Lambda}}_{L^2_{\varrho}(\cU ; \cV)},\qquad E_{\Lambda,\infty}(f) =\nmu{f - f_{\Lambda}}_{L^{\infty}(\cU ; \cV)},
}
where $f_{\Lambda}$ is as in \R{f_exp_trunc}, and, given a subspace $\cV_h \subseteq L^2_{\varrho}(\cU ; \cV)$, we let
\bes{
E_{h,\infty}(f) = \nmu{f - \cP_h(f)}_{L^{\infty}(\cU ; \cV)},
}
where $\cP_h(f)$ is as in \R{Phf-def}.

\thm{
[Error bounds for inexact minimizers, algebraic and finite-dimensional case]
\label{t:main-res-map-alg_inex}
Let $d \in \bbN$, $m \geq 3$, $ 0 < \epsilon < 1$, $\{ \Psi_{\bm{\nu}} \}_{\bm{\nu} \in \bbN^d_0} \subset L^2_{\varrho}(\cU)$ be either the orthonormal Chebyshev or Legendre basis, $\cV_h \subseteq L^2_{\varrho}(\cU)$ be a subspace of $L^2_{\varrho}(\cU)$ and $\Lambda = \Lambda_{n,d}^{\mathsf{HC}}$ be the hyperbolic cross index set with $n= \lceil m/L\rceil$ where  $L = L(m,d,\epsilon)$ is as in \eqref{Ldef}. Let $f \in L^2_{\varrho}(\cU;\cV)$, draw $\bm{y}_1,\ldots,\bm{y}_m$ randomly and independently according to $\varrho$ and suppose that $\bm{A}$, $\bm{b}$ and $\bm{e}$ are as in \R{def-measMatrix} and \R{def-measVec}. Consider the Hilbert-valued, weighted SR-LASSO problem \R{wsr-LASSO} with weights $\bm{w} = \bm{u}$ as in \R{weights_def} and $\lambda = (4 \sqrt{m/L})^{-1}$. Then there exists universal constants $c_0, c_1,c_2 \geq 1$ such that the following holds with probability at least $1-\epsilon$. Any $\tilde{\bm{c}} = (\tilde{c}_{\bm{\nu}})_{\bm{\nu} \in \Lambda} \in \bbC^N$ satisfies
\bes{
\nmu{f -\tilde{f}   }_{L^2_{\varrho}(\cU ; \cV)} \leq c_1 \cdot \xi,\quad \nmu{f - \tilde{f}  }_{L^{\infty}(\cU ; \cV)} \leq c_2 \cdot \sqrt{k} \cdot \xi,\qquad \tilde{f} : = \sum_{\bm{\nu} \in \Lambda} \tilde{c}_{\bm{\nu}} \Psi_{\bm{\nu}},
}
where
\bes{
\xi = \dfrac{\sigma_{k}(\bm{c}_{\Lambda})_{1,\u;\cV}}{\sqrt{k}}   +   \frac{E_{\Lambda,\infty}(f)}{\sqrt{k}}+ E_{\Lambda,2}(f) + E_{h,\infty}(f) + \cG(\tilde{\bm{c}}) - \cG(\cP_h(\bm{c}_{\Lambda}))   + \frac{\nm{\bm{n}}_{2;\cV}}{\sqrt{m}} ,
}
$\bm{c}_{\Lambda}$ is as in \R{f_coeff_trunc}, $\cP_{h}(\bm{c}_{\Lambda}) = (\cP_h(c_{\bm{\nu}}))_{\bm{\nu} \in \Lambda}$, $k = m / (c_0 L)$ for $L = L(m,d,\epsilon)$ as in \R{Ldef}, and $\bm{n}$ is as in \R{def-measVec}. 
}

\begin{proof}
We divide the proof into several steps.  

\pbk \textit{Step 1: Splitting the error into separate terms.} Consider the $L^2_{\varrho}(\cU;\cV)$-norm error first. By the triangle inequality and the fact that $\cP_h$ is a projection, we have
\eas{
\nmu{f -\tilde{f}   }_{L^2_{\varrho}(\cU ; \cV)} 
 &\leq  \nm{f - \cP_h(f)}_{L^2_{\varrho}(\cU; \cV)} + \nm{\cP_h(f) - \cP_{h}(f_{\Lambda}) }_{L^2_{\varrho}(\cU ; \cV)} + \nmu{\cP_h(f_{\Lambda}) - \tilde{f} }_{L^2_{\varrho}(\cU ; \cV)} 
\\
 & \leq   \nm{f - \cP_h(f)}_{L^{\infty}(\cU; \cV)} + \nm{ f -  f_{\Lambda} }_{L^2_{\varrho}(\cU ; \cV)} + \nmu{\cP_h(f_{\Lambda}) - \tilde{f}  }_{L^2_{\varrho}(\cU ; \cV)} 
\\
 &=   E_{h,\infty}(f) + E_{\Lambda,2}(f) + \nmu{\cP_h(f_{\Lambda}) - \tilde{f}  }_{L^2_{\varrho}(\cU ; \cV)} .
}
Then, by orthonormality, we have 
\bes{
\nmu{f -\tilde{f}   }_{L^2_{\varrho}(\cU ; \cV)}  \leq  E_{h,\infty}(f) + E_{\Lambda,2}(f) + \nmu{ \cP_h(\bm{c}_{\Lambda}) - \tilde{\bm{c}}}_{2 ; \cV} .
}
Similarly, for the $L^{\infty}(\cU;\cV)$-norm error, we have 
\eas{
\nmu{f -\tilde{f}   }_{L^{\infty}(\cU ; \cV)} 
 &\leq  \nm{f - \cP_h(f)}_{L^{\infty}(\cU; \cV)} + \nm{\cP_h(f) - \cP_{h}(f_{\Lambda}) }_{L^{\infty}(\cU ; \cV)} + \nmu{\cP_h(f_{\Lambda}) - \tilde{f} }_{L^{\infty}(\cU ; \cV)} 
\\
 & \leq   \nm{f - \cP_h(f)}_{L^{\infty}(\cU; \cV)} + \nm{ f -  f_{\Lambda} }_{L^{\infty}(\cU ; \cV)} + \nmu{\cP_h(f_{\Lambda}) - \tilde{f}  }_{L^{\infty}(\cU ; \cV)} 
\\
 &=   E_{h,\infty}(f) + E_{\Lambda,\infty}(f) + \nmu{\cP_h(f_{\Lambda}) - \tilde{f}  }_{L^{\infty}(\cU ; \cV)} .
}
Using the definition \R{weights_def} of the weights $\bm{u}$, we deduce that
\bes{
\nmu{f -\tilde{f}   }_{L^{\infty}(\cU ; \cV)}  \leq E_{h,\infty}(f) + E_{\Lambda,\infty}(f) + \nmu{\cP_h(\bm{c}_{\Lambda}) - \tilde{\bm{c}}}_{1,\bm{u};\cV}.
}
Therefore, the rest of the proof is devoted to showing the following bounds:
\be{
\label{bound-needed-1}
\nmu{ \cP_h(\bm{c}_{\Lambda}) - \tilde{\bm{c}}}_{2 ; \cV} \leq c_1 \cdot \xi,\quad  \nmu{\cP_h(\bm{c}_{\Lambda}) - \tilde{\bm{c}}}_{1,\bm{u};\cV} \leq c_2 \cdot \sqrt{k} \cdot \xi.
}
We do this in the next two steps by first asserting that $\bm{A}$ has the weighted rNSP (Step 2) and then by applying the error bounds of Lemma \ref{lemma-musuboptimal} (Steps 3 and 4).

\pbk  
\textit{Step 2: Asserting the weighted rNSP.}  We now show that $\bm{A}$ has the weighted rNSP over $\cV_h$ of order $(k,\u)$ with probability at least $1-\epsilon/2$. This is based on Lemma \ref{l:LegMat_RIP}. First observe that
\bes{
L = L(m,d,\epsilon) \geq \log^2(3) \cdot \min \{ \log(3) + 1 , \log(3) \cdot \log(\E) \} \geq 1,
}
since $m \geq 3$. This implies that $m \geq m / L \geq m / (c_0 L) = k$ since $c_0 \geq 1$ as well. Since $n = \lceil m / L \rceil \leq m/L + 1 \leq 2 m$, we get
\eas{
\log(4 k) & \cdot  \left ( \log(4 k) \cdot \min \left \{ \log(n) + d , \log(\E d) \cdot \log(2 n) \right \} + \log(2/\epsilon) \right ) 
\\
& \leq  \log(4m) \cdot  \left ( \log(4m) \cdot \min \left \{ \log(2m) + d , \log(\E d) \cdot \log(4m) \right \} + \log(2/\epsilon) \right ) 
 \\
& \leq  c_0 L(m,d,\epsilon) / 2
}
for a suitably-large choice of $c_0$. Hence
\bes{
m = c_0  k L(m,d,\epsilon) \geq 2 c_0 k L'(2k,d,\epsilon/2),
}
where $L'$ is defined as in Lemma~\ref{l:LegMat_RIP}, and therefore (again assuming a suitably-large choice of $c_0$) \R{m-cond-for-wRIP} holds with $k$ replaced by $2k$. We deduce that $\bm{A}$ satisfies the weighted RIP of order $(2k,\u)$ with constant $\delta_{2k,\u} \leq 1/4$, with probability at least $1-\epsilon/2$. Then, we deduce from Lemmas \ref{l:from_R_to_V_rNSP} and \ref{implies-RIP} that $\bm{A}$ has (with the same probability) the weighted rNSP of order $(k,\u)$ over $\cV_h$ with constants $\rho = 2 \sqrt{2} / 3$ and $\gamma = 2 \sqrt{5} / 3$.

\pbk
\textit{Step 3: Bounding $\cP_h(\bm{c}_{\Lambda}) - \tilde{\bm{c}}$ using the weighted rNSP.}  
We use Lemma \ref{lemma-musuboptimal}. First, consider the value of $\lambda$. Since $c_0 \geq 1$ we have $m/L \geq m/(c_0 L) = k$. Hence, recalling the values for $\rho$ and $\gamma$ obtained in the previous step, we have
\be{
\label{lambda-bound}
\frac{1}{4 \sqrt{c_0}} \frac{1}{\sqrt{k}} = \frac{1}{4 \sqrt{m/L}} = \lambda \leq \frac{1}{4 \sqrt{k}} <   \dfrac{(1+ \rho)^2}{(3+\rho) \gamma }    \frac{1}{\sqrt{k}}.
}
Therefore \R{lambda-bound-rNSP-err} holds. We now apply this lemma with $\cV = \cV_h$, $\bm{x} = \cP_h(\bm{c}_{\Lambda})$, $\tilde{\bm{x}} = \tilde{\bm{c}}$ and $\bm{e} = \bm{A} \cP_h(\bm{c}_{\Lambda}) - \bm{b}$. Notice first that the best $(k,\u)$-approximation error \R{sigma-k-w-equiv} satisfies
\be{
\label{sigma-k-Ph}
\sigma_{k}(\cP_h(\bm{c}_{\Lambda}))_{1,\bm{u};\cV} = \inf \left \{ \sum_{\bm{\nu} \in \Lambda \backslash S} u_{\bm{\nu}} \nm{\cP_h(c_{\bm{\nu}})}_{\cV} : S \subseteq \Lambda,\ |S|_{\bm{u}} \leq k \right \} \leq \sigma_{k}(\bm{c}_{\Lambda})_{1,\bm{u};\cV},
}
since $\cP_h$ is a projection. Hence, applying Lemma \ref{lemma-musuboptimal} and using the lower bound in \R{lambda-bound}, we get
\be{
\label{tilde-c-Ph-c-bd}
\begin{split}
\nmu{\tilde{\bm{c}} - \cP_h(\bm{c}_{\Lambda}) }_{2;\cV}  & \leq c_1 \left [ \frac{\sigma_{k}(\bm{c}_{\Lambda})_{1,\bm{w};\cV}}{\sqrt{k}} + \cG(\tilde{\bm{c}}) - \cG(\cP_h(\bm{c}_{\Lambda})) + \nm{\bm{A} \cP_h(\bm{c}_{\Lambda}) - \bm{b}}_{2;\cV} ,
\right ],
\\
\nmu{\tilde{\bm{c}} - \cP_h(\bm{c}_{\Lambda}) }_{1,\bm{u};\cV}  & \leq c_2 \left [ \sigma_{k}(\bm{c}_{\Lambda})_{1,\bm{w};\cV} + \sqrt{k} \left (\cG(\tilde{\bm{c}}) - \cG(\cP_h(\bm{c}_{\Lambda})) \right ) + \sqrt{k} \nm{\bm{A} \cP_h(\bm{c}_{\Lambda}) - \bm{b}}_{2;\cV} 
\right ],
\end{split}
}
with probability at least $1-\epsilon/2$.
Therefore, to show \R{bound-needed-1} and therefore complete the proof, it suffices to show that the following holds with probability at least $1-\epsilon/2$:
\be{
\label{bound-needed-2}
\nm{\bm{A} \cP_h(\bm{c}_{\Lambda}) - \bm{b}}_{2;\cV} \leq \sqrt{2} \left ( \frac{E_{\Lambda,\infty}(f)}{\sqrt{k}}  + E_{\Lambda,2}(f) \right )+ E_{h,\infty}(f) + \frac{\nm{\bm{n}}_{2;\cV}}{\sqrt{m}}.
}
The overall result then follows by the union bound.

\pbk
\textit{Step 4: Showing that \R{bound-needed-2} holds.} Observe that
\eas{
\sqrt{m} \nmu{(\bm{A} \cP_h(\bm{c}_{\Lambda}) - \bm{b})_i}_{\cV} & \leq \nmu{\cP_h(f_{\Lambda})(\y_i) - f(\y_i) }_{\cV} + \nm{n_i}_{\cV}
\\
& \leq \nmu{\cP_h(f_{\Lambda})(\y_i) - \cP_h(f)(\y_i) }_{\cV} + \nmu{f(\y_i) - \cP_h(f)(\y_i) }_{\cV} + \nm{n_i}_{\cV}
\\
& \leq \nmu{f(\bm{y}_i) - f_{\Lambda}(\y_i) }_{\cV} + E_{h,\infty}(f) + \nmu{n_i}_{\cV}.
}
Therefore
\be{
\label{A-E-disc-bd}
\nmu{\bm{A} \cP_h(\bm{c}_{\Lambda}) - \bm{b}}_{\cV;2} \leq E_{\Lambda,\mathrm{disc}}(f) + E_{h,\infty}(f) + \frac{\nm{\bm{n}}_{2;\cV} }{\sqrt{m}},
}
where 
\be{
\label{discrete-error}
E_{\Lambda,\mathrm{disc}}(f) = \sqrt{\frac1m \sum^{m}_{i=1} \nmu{f(\bm{y}_i) - f_{\Lambda}(\y_i) }^2_{\cV} }.
}
For this final step, we follow near-identical arguments to those found in \cite[Lem.\ 7.11]{adcock2021sparse}. This shows that 
\bes{
E_{\Lambda,\mathrm{disc}}(f) \leq \sqrt{2} \left ( \frac{E_{\Lambda,\infty}(f)}{\sqrt{k}} + E_{\Lambda,2}(f) \right ),
}
with probability at least $1-\epsilon/2$, provided $m \geq 2 k \log(2/\epsilon)$. However, this follows due to the assumptions on $m$ and the arguments given in Step 2. Thus we obtain \R{bound-needed-2} and the proof is complete.
\end{proof}

\thm{
[Error bounds for inexact minimizers, algebraic and infinite-dimensional case]
\label{t:main-res-map-alg_inex_infty}
Let $d = \infty$, $m \geq 3$, $ 0 < \epsilon < 1$, $\{ \Psi_{\bm{\nu}} \}_{\bm{\nu} \in \cF} \subset L^2_{\varrho}(\cU)$ be either the orthonormal Chebyshev or Legendre basis, $\cV_h \subseteq L^2_{\varrho}(\cU)$ be a subspace of $L^2_{\varrho}(\cU)$ and $\Lambda = \Lambda_{n}^{\mathsf{HCI}}$ be the hyperbolic cross index set with $n= \lceil m/L\rceil$ where  $L = L(m,d,\epsilon)$ is as in \eqref{Ldef}. Let $f \in L^2_{\varrho}(\cU;\cV)$, draw $\bm{y}_1,\ldots,\bm{y}_m$ randomly and independently according to $\varrho$ and suppose that $\bm{A}$, $\bm{b}$ and $\bm{e}$ are as in \R{def-measMatrix} and \R{def-measVec}. Consider the Hilbert-valued, weighted SR-LASSO problem \R{wsr-LASSO} with weights $\bm{w} = \bm{u}$ as in \R{weights_def} and $\lambda = (4 \sqrt{m/L})^{-1}$. Then there exists universal constants $c_0, c_1,c_2 \geq 1$ such that the following holds with probability at least $1-\epsilon$. Any $\tilde{\bm{c}} = (\tilde{c}_{\bm{\nu}})_{\bm{\nu} \in \Lambda} \in \bbC^N$ satisfies
\bes{
\nmu{f -\tilde{f}   }_{L^2_{\varrho}(\cU ; \cV)} \leq c_1 \cdot \xi,\quad \nmu{f - \tilde{f}  }_{L^{\infty}(\cU ; \cV)} \leq c_2 \cdot \sqrt{k} \cdot \xi,\qquad \tilde{f} : = \sum_{\bm{\nu} \in \Lambda} \tilde{c}_{\bm{\nu}} \Psi_{\bm{\nu}},
}
where
\bes{
\xi = \dfrac{\sigma_{k}(\bm{c}_{\Lambda})_{1,\u;\cV}}{\sqrt{k}}   +   \frac{E_{\Lambda,\infty}(f)}{\sqrt{k}}+ E_{\Lambda,2}(f) + E_{h,\infty}(f) + \cG(\tilde{\bm{c}}) - \cG(\cP_h(\bm{c}_{\Lambda}))   + \frac{\nm{\bm{n}}_{2;\cV}}{\sqrt{m}} ,
}
$\bm{c}_{\Lambda}$ is as in \R{f_coeff_trunc}, $\cP_{h}(\bm{c}_{\Lambda}) = (\cP_h(c_{\bm{\nu}}))_{\bm{\nu} \in \Lambda}$, $k = m / (c_0 L)$ for $L = L(m,d,\epsilon)$ as in \R{Ldef}, and $\bm{n}$ is as in \R{def-measVec}. 
}
\prf{
The proof has the same structure as that of the previous theorem. Steps 1, 3 and 4 are identical. The only differences occur in Step 2. We now describe these changes. Once more we observe that $L = L(m,\infty,\epsilon) \geq 1$ since $m \geq 3$. Hence $m \geq m/L \geq m/(c_0 L) = k$ since $c_0 \geq 1$. We also have $n = \lceil m/L \rceil \leq 2 m$. Therefore
\eas{
 \log(4 k) \cdot \left ( \log(4 k) \cdot \log^2(2 n) + \log(2/\epsilon) \right )
\leq  \log(4 m) \cdot \left ( \log^3(4 m)  + \log(2/\epsilon) \right )
\leq  c_0 L(m,\infty,\epsilon)/2
}
for a suitably-large choice of $c_0$. We deduce that $m = c_0 k L(m,\infty,\epsilon) \geq 2 c_0 k L'(2k,\infty,\epsilon/2)$, where $L'$ is as in Lemma \ref{l:LegMat_RIP}. An application of this lemma now shows that $\bm{A}$ has the weighted RIP of order $(2k,\u)$ with constant $\delta_{2k,\u} \leq 1/4$, as required.
}

\thm{
[Error bounds for inexact minimizers, exponential case]
\label{t:main-res-map-alg_inex_exponential}
Let $d \in \bbN$, $m \geq 3$, $ 0 < \epsilon < 1$, $\{ \Psi_{\bm{\nu}} \}_{\bm{\nu} \in \bbN^d_0} \subset L^2_{\varrho}(\cU)$ be either the orthonormal Chebyshev or Legendre basis, $\cV_h \subseteq L^2_{\varrho}(\cU)$ be a subspace of $L^2_{\varrho}(\cU)$ and $\Lambda = \Lambda_{n,d}^{\mathsf{HC}}$ 
be the hyperbolic cross index set with $n$ as in \R{n-exp-case}. Draw $\bm{y}_1,\ldots,\bm{y}_m$ randomly and independently according to $\varrho$. Then, with probability at least $1-\epsilon$, the following holds. Let $f \in L^2_{\varrho}(\cU;\cV)$ and suppose that $\bm{A}$, $\bm{b}$ and $\bm{e}$ are as in \R{def-measMatrix} and \R{def-measVec}. Consider the Hilbert-valued, weighted SR-LASSO problem \R{wsr-LASSO} with weights $\bm{w} = \bm{u}$ as in \R{weights_def} and $\lambda = (4 \sqrt{m/L})^{-1}$. Then there exists universal constants $c_0, c_1,c_2 \geq 1$ such that any $\tilde{\bm{c}} = (\tilde{c}_{\bm{\nu}})_{\bm{\nu} \in \Lambda} \in \bbC^N$ satisfies
\bes{
\nmu{f -\tilde{f}   }_{L^2_{\varrho}(\cU ; \cV)} \leq c_1 \cdot \xi,\quad \nmu{f - \tilde{f}  }_{L^{\infty}(\cU ; \cV)} \leq c_2 \cdot \sqrt{k} \cdot \xi,\qquad \tilde{f} : = \sum_{\bm{\nu} \in \Lambda} \tilde{c}_{\bm{\nu}} \Psi_{\bm{\nu}},
}
where
\bes{
\xi = \dfrac{\sigma_{k}(\bm{c}_{\Lambda})_{1,\u;\cV}}{\sqrt{k}}   +  E_{\Lambda,\infty}(f)+ E_{h,\infty}(f) + \cG(\tilde{\bm{c}}) - \cG(\cP_h(\bm{c}_{\Lambda}))   + \frac{\nm{\bm{n}}_{2;\cV}}{\sqrt{m}} ,
}
$\bm{c}_{\Lambda}$ is as in \R{f_coeff_trunc}, $\cP_{h}(\bm{c}_{\Lambda}) = (\cP_h(c_{\bm{\nu}}))_{\bm{\nu} \in \Lambda}$, $k = m / (c_0 L)$
for $L = L(m,d,\epsilon)$ as in \R{Ldef}, and $\bm{n}$ is as in \R{def-measVec}. 
}

\prf{
The proof has the same structure as that of Theorem \ref{t:main-res-map-alg_inex}. Step 1 is identical, and reduces the proof to showing that \R{bound-needed-1} holds. We now describe the modifications needed in Steps 2--4:

\pbk
\textit{Step 2: Asserting the weighted rNSP.} We now show that $\bm{A}$ has the weighted rNSP over $\cV_h$ of order $(k,\bm{u})$ with probability at least $1-\epsilon$. This step is essentially the same, except for the choice of $n$ and the probability $1-\epsilon$ instead of $1-\epsilon/2$. 

\pbk
\textit{Step 3: Bounding $\cP_h(\bm{c}_{\Lambda}) - \tilde{\bm{c}}$ using the weighted rNSP.} Since $\lambda$ and $k$ are the same as in Theorem \ref{t:main-res-map-alg_inex}, the bound \R{lambda-bound} also holds in this case. We then follow the same arguments, leading to \R{tilde-c-Ph-c-bd} holding with probability at least $1-\epsilon$. Finally, rather than \R{bound-needed-2}, we ask for the slightly modified bound
\be{
\label{bound-needed-2-exp}
\nm{\bm{A} \cP_h(\bm{c}_{\Lambda}) - \bm{b}}_{2;\cV} \leq E_{\Lambda,\infty}(f) + E_{h,\infty}(f) + \frac{\nm{\bm{n}}_{2;\cV}}{\sqrt{m}} ,
}
to hold with probability one.

\pbk
\textit{Step 4: Showing \R{bound-needed-2-exp} holds.} By the same argument, we see that \R{A-E-disc-bd} holds. Instead of the probabilistic bound for $E_{\Lambda,\mathrm{disc}}(f)$, we now simply bound it as
\bes{
E_{\Lambda,\mathrm{disc}}(f) \leq \nmu{f - f_{\Lambda}}_{L^{\infty}(\cU ; \cV)} = E_{\Lambda,\infty}(f).
}
This immediately implies \R{bound-needed-2-exp}.

\pbk
Finally, we observe that we can simplify the previous estimates in this case using the bound $E_{\Lambda,2}(f) \leq E_{\Lambda,\infty}(f)$.
}

\section{Error bounds and the restarting scheme for the primal-dual iteration}\label{s:errbds_PDI}

Theorems \ref{t:main-res-map-alg_inex}--\ref{t:main-res-map-alg_inex_exponential} reduce the problem of proving the main results (Theorems \ref{t:main-res-map-alg}--\ref{t:main-res-effic-algo-alg_exp}) to two tasks. The first involves bounding the error in the objective function, i.e.\ the term
\bes{
\cG(\tilde{\bm{c}}) - \cG(\cP_h(\bm{c}_{\Lambda})),
}
where $\tilde{\bm{c}}$ is either an exact minimizer or an approximate minimizer obtained via the primal dual iteration. The second involves the various approximation error terms depending on $f$ and its polynomial coefficients. 

In this section, we address the first task. We first provide an error bound for the (unrestarted) primal-dual iteration when applied to Hilbert-valued weighted SR-LASSO problem \R{SRLASSO-CS-sec}, and then use this to derive the specific restart scheme.

\subsection{Error bounds for the primal-dual iteration}

We now return to the general setting of the primal-dual iteration, where it is applied to the problem \R{primal1} and takes the form \R{PDI}. The following result from \cite[Theorem 5.1]{ChambolleEtAl2016}  establishes an important error bound for the Lagrangian difference.

 \begin{theorem}\label{Primal-dual-bound}
Let $\tau , \sigma >0$, initial points $(x^{(0)},{\xi}^{(0)}) \in \cX \times \cY $ and  a bounded linear operator  $A \in \cB(\cX,\cY)$, be such that $\|A\|_{\mathcal{B}(\cX,\cY)}^2  \leq (\tau \sigma)^{-1} $. Consider the sequence $\{ (x^{(n)},{\xi}^{(n)}) \}^{\infty}_{n=1}$ generated by the primal-dual iteration \R{PDI}. Then, for any $(x,{\xi}) \in \cX \times \cY $, 
\begin{equation}\label{reduced-PD-gap}
\cL({\bar{x}}^{(n)},{\xi}) - \cL(x,{\bar{\xi}}^{(n)}) \leq \dfrac{\tau^{-1}\nmu{x-x^{(0)}}_{2;\cV}^2+\sigma^{-1}\nmu{{\xi}-{\xi}^{(0)}}_{2;\cV}^2}{n},
\end{equation}
where
\begin{equation*}
\bar{x}^{(n)}= \dfrac{1}{n}\sum_{k=1}^n {x}^{(k)} \qan  {\bar{\xi}}^{(n)}= \dfrac{1}{n}\sum_{k=1}^n {\xi}^{(k)}, 
\end{equation*}
are the ergodic sequences and $\cL$ is the Lagrangian \R{Lag1}. 
\end{theorem}

The following lemma shows a decay rate of $1/n$ on the objective function in the case of the primal-dual iteration when applied to the problem \R{SRLASSO-CS-sec}. It is an extension of \cite[Lem.\ 8.6]{adcock2021compressive} to the  weighted and Hilbert-valued setting.
\begin{lemma}
\label{lemma-subopti}
Let $\bm{A} \in \cB(\cV^N,\cV^m)$ and $\tau , \sigma >0$ be such that $\|\bm{A}\|_{\mathcal{B}(\cVN,\cVM)}^2  \leq (\tau \sigma)^{-1} $. Consider the sequence $\{ (\x^{(n)},\bm{\xi}^{(n)}) \}^{\infty}_{n=1}$ generated by the primal-dual iteration in  \R{PDI}  applied to \R{SRLASSO-CS-sec} 
 with $\x^{(0)}\in \cVN$ and $\bm{\xi}^{(0)} = \bm{0} \in \cV^m$. Then, for any $\x \in \cVN$,
\begin{equation}
\label{subopti-error}
\cG(\bm{\bar{x}}^{(n)}) - \cG(\bm{x}) \leq \dfrac{{\tau}^{-1}\nm{\x-\x_0}_{2;\cV}^2+{\sigma}^{-1}}{n},\qquad \bar{\x}^{(n)} = \frac{1}{n}\sum_{k=1}^n \bm{x}^{(k)}.
\end{equation}
\end{lemma}

\begin{proof}
Using \R{Lag1} and \R{eq-fstar}, the left-hand side of \R{reduced-PD-gap} is given by
\eas{
\cT_{n}(\bm{x} , \bm{\xi} ) : = & \left ( \lambda \nmu{\bar{\bm{x}}^{(n)}}_{1,\bm{w};\cV} + \Re \ip{\bm{A} \bar{\bm{x}}^{(n)} - \bm{b} }{\bm{\xi}}_{2;\cV} + \delta_{B}(\bm{\xi}) \right )
\\
& - \left ( \lambda \nmu{\x}_{1,\bm{w};\cV} + \Re \ip{\bm{A} \x- \bm{b} }{\bar{\bm{\xi}}^{(n)}}_{2;\cV} + \delta_{B}(\bar{\bm{\xi}}^{(n)}) \right ),
}
where $B$ is the unit ball in $\cV^m$.
Observe that the term $\bxi^{(n)}$ produced by this iteration satisfies $\nm{\bxi^{(n)}}_{2;\cV} \leq 1$. This follows from the observation shown in \S \ref{Sec:PDI-A} that the proximal mapping
\bes{
\mathrm{prox}_{\sigma h^*}(\bm{\xi}) =  \mathrm{proj}_{B}(\bm{\xi} - \sigma \bm{b})
}
involves the projection onto the unit ball $B$. Hence the ergodic sequence $\bar{\bxi}^{(n)}$ satisfies $\nm{\bar{\bxi}^{(n)}}_{2;\cV} \leq 1$ as well. Suppose now that $\bm{A}{\x}^{(n)}-\b \neq \0$ and set
\bes{
\bxi=\dfrac{\bm{A}{\x}^{(n)}-\b}{\nm{\bm{A}{\x}^{(n)}-\b}_{2;\cV}}.
}
Then $\delta_B(\bxi)=\delta_B(\bar{\bxi}^{(n)})=1$ and therefore
\eas{
\cT_{n}(\x , \bm{\xi}) & = \left ( \lambda \nmu{\bar{\bm{x}}^{(n)}}_{1,\bm{w};\cV} + \nmu{\bm{A}\bar{\bm{x}}^{(n)} - \bm{b}}_{2;\cV} \right ) - \left (  \lambda \nmu{\x}_{1,\bm{w};\cV} + \Re \ip{\bm{A} \x- \bm{b} }{\bar{\bm{\xi}}^{(n)}}_{2;\cV} \right )
\\
& \geq \left ( \lambda \nmu{\bar{\bm{x}}^{(n)}}_{1,\bm{w};\cV} + \nmu{\bm{A}\bar{\bm{x}}^{(n)} - \bm{b}}_{2;\cV} \right ) - \left (  \lambda \nmu{\x}_{1,\bm{w};\cV} + \nm{\bm{A} \x- \bm{b} }_{2;\cV} \right ).
}
Clearly, the same bound also holds in the case $\bm{A}{\x}^{(n)}-\b =\0$ where $\bm{\xi}$ is an arbitrary unit vector. Hence Theorem \ref{Primal-dual-bound} and the fact that $\nmu{\bm{\xi} - \bm{\xi}_0 }_{2;\cV} = \nm{\bm{\xi}}_{2;\cV} = 1$ gives the result. 
\end{proof}

\subsection{The restarting scheme}\label{ss:restart-scheme}

For convenience, we now introduce new and slightly modify some existing notation. First, we redefine the objective function $\cG$ of the Hilbert-valued weighted SR-LASSO problem \R{SRLASSO-CS-sec} to make the dependence on the term $\bm{b}$ explicit: namely, we set
\bes{
\cG(\x,\bm{b}) = \lambda \nm{\bm{\x}}_{1,\w;\cV} + \nmu{\bm{A} \bm{x} - \bm{\bm{b}} }_{2;\cV}, \quad  \x \in \cVN,\ \bm{b} \in \cVm.
}
We then let
\begin{equation}\label{Error_F}
\cE(\z,\x,\b) = \cG(\z,\bm{b}) - \cG(\x,\bm{b}), \quad \x, \z \in \cVN,\ \bm{b} \in \cVm.
\end{equation}
Now consider the ergodic sequence $\bar{\bm{x}}^{(n)}$ produced by $n$ iterations of the primal-dual iteration \R{PDI} applied to \R{SRLASSO-CS-sec} with parameters $\tau,\sigma > 0$, $\x_0 \in \cV^N$ and $\bm{\xi}_0 = \bm{0} \in \cV^m$. For reasons that will become clear in a moment, we now make the dependence on the vector $\bm{b}$ in \R{SRLASSO-CS-sec}, the number of iterations $\bar{\x}^{(n)}$ and the initial vector $\x_0$ explicit, by defining
\bes{
\cP(\x_0 , \bm{b} , n) = \bar{\x}^{(n)}.
}
With this in hand, we conclude this discussion by noting the following two scaling properties:
\begin{equation}\label{scaling_prop}
\cG(a \x , \b) = a \cG(\x,\b/a), \quad \cE(a \z , \x , \b) = a \cE (\z , \x/a , \b/a).
 \end{equation} 
 These hold for any $a > 0$ and for any $\x, \z \in \cV^N$ and $\b \in \cV^m$.

\lem{
Suppose that $\bm{A} \in \cB(\cV^N,\cV^m)$ has the weighted rNSP over $\cV$ of order $(k,\w)$ with constants  $0<\rho<1$ and $\gamma>0$. Consider the Hilbert-valued weighted SR-LASSO problem \R{SRLASSO-CS-sec} with parameter $\lambda = c / \sqrt{k}$, where $0 < c \leq \frac{(1+\rho)^2}{(3+\rho) \gamma}$. Let $\cE$ and $\cP$ be as defined above, $\tau,\sigma$ satisfy $\nmu{\bm{A}}^2_{\cB(\cV^N,\cV^m)} \leq (\tau \sigma)^{-1}$ and $\x,\x_0 \in \cVN$, $\bm{b} \in \cV^m$, $a > 0$. Then
\bes{
\cE(a \cP(\x_0/a,\bm{b}/a,n) , \bm{x} , \bm{b} ) \leq \frac{C^2}{a \tau n}  \left (  \cE(\x_0,\x,\bm{b}) + \xi \right )^2 + \frac{a}{\sigma n},
} 
where
\be{
\label{C-def-restart-scheme}
C = 2  \max \left \{ C'_1 / c , C'_2 \right \} ,
}
$C'_1,C'_2$ are as in Lemma \ref{lemma-wrNSP-l1} 
and
\be{
\label{xi-def}
\xi = \xi(\bm{x},\bm{b}) = \frac{\sigma_{k}(\x)_{1,\w;\cV}}{\sqrt{k}} + \nm{\bm{A} \bm{x} - \bm{b}}_{2;\cV}.
}
}

\prf{
The scaling property \R{scaling_prop} and Lemma \ref{lemma-subopti} give
\eas{
\cE(a \cP(\x_0/a,\bm{b}/a,n) , \bm{x} , \bm{b} ) & = a \cE(\cP(\x_0/a,\bm{b}/a,n) , \bm{x}/a , \bm{b}/a) 
\\
& \leq a \left ( \frac{\tau^{-1} \nm{\x/a - \x_0/a}^2_{2;\cV} + \sigma^{-1} }{n} \right )
\\
& = \frac{\nm{\x-\x_0}^2_{2;\cV}}{a \tau n} + \frac{a}{\sigma n}.
}
Now consider the term $\nm{\x-\x_0}_{2;\cV}$. Since $\bm{A}$ has the weighted rNSP and $\lambda$ satisfies \R{lambda-bound-rNSP-err}, we may use Lemma \ref{lemma-musuboptimal} to get
\eas{
\nm{\x-\x_0}_{2;\cV} & \leq  \dfrac{C_1'}{\sqrt{k}} \left( 2 {\sigma_{k}(\x)_{1,\w;\cV}} 
+\dfrac{\cG(\x_0,\bm{b}) - \cG(\x,\bm{b}) }{ \lambda} \right) + \left( \dfrac{C'_1}{\sqrt{k}\lambda} + C'_2 \right)\nm{\bm{A} \bm{x} - \bm{b}}_{2;\cV}
\\
& = \frac{C'_1}{\sqrt{k} \lambda} \cE(\x_0,\x,\bm{b}) + 2 C'_1 \frac{\sigma_{k}(\x)_{1,\w;\cV}}{\sqrt{k}} + \left( \dfrac{C'_1}{\sqrt{k}\lambda} + C'_2 \right)\nm{\bm{A} \bm{x} - \bm{b}}_{2;\cV}
\\
& \leq 2 \max \left \{ C'_1 / c , C'_2 \right \} \left (  \cE(\x_0,\x,\bm{b}) + \xi \right ).
}
Substituting this into the previous expression now gives the result.
}

This lemma gives the rationale behind the restarted scheme. It says the error in the objective function of the scaled output $a \cP(\x_0/a,\bm{b}/a,n)$ of the primal-dual iteration with initial value $\x_0$ can be bounded in terms of the error in the objective function at the initial value, plus terms depending on the scaling parameter $a$, the number of iterations $n$ and the compressed sensing error term $\xi$. By choosing these parameters suitably and iterating this procedure, we obtain the restarting scheme. We summarize this in the following theorem:

\thm{[Restarting scheme]
\label{thm:restart-scheme-bound}
Suppose that $\bm{A} \in \cB(\cV^N,\cV^m)$ has the weighted rNSP over $\cV$ of order $(k,\w)$ with constants  $0<\rho<1$ and $\gamma>0$. Consider the Hilbert-valued weighted SR-LASSO problem \R{SRLASSO-CS-sec} with parameter $\lambda = c / \sqrt{k}$, where $0 < c \leq \frac{(1+\rho)^2}{(3+\rho) \gamma}$. Let $\x \in \cV^N$, $\bm{b} \in \cV^m$, $\zeta' \geq \xi$, where $\xi$ is as in \R{xi-def}, $0 < r < 1$ and define the sequence
\bes{
\varepsilon_0 = \nm{\bm{b}}_{2;\cV },\qquad \varepsilon_{k+1} = r (\varepsilon_k + \zeta'),\ k = 0,1,2,\ldots .
}
Let $\cE$ and $\cP$ be as defined above, $\tau,\sigma$ satisfy $\nmu{\bm{A}}^2_{\cB(\cV^N,\cV^m)} \leq (\tau \sigma)^{-1}$ and set
\bes{
n = \left \lceil \frac{2 C}{r \sqrt{\sigma \tau} } \right \rceil,\qquad a_k = \frac12  \sigma \varepsilon_{k+1} n , \ k = 0,1,2,\ldots ,
}
where $C$ is as in \R{C-def-restart-scheme}. Then the iteration $\tilde{\bm{x}}^{(0)},\tilde{\bm{x}}^{(1)},\tilde{\bm{x}}^{(2)},\ldots$, defined by
\bes{
\tilde{\bm{x}}^{(0)} = \bm{0},\qquad \tilde{\bm{x}}^{(k+1)} = a_k \cP(\tilde{\bm{x}}^{(k)} / a_k , \bm{b} / a_k , n ),\ k = 0,1,2,\ldots ,
}
satisfies
\bes{
\cE(\bm{x}^{\star}_k , \bm{x} , \bm{b} ) \leq \varepsilon_k \leq r^k \nm{\bm{b}}_{2;\cV} + \frac{r}{1-r} \zeta',\quad k = 0,1,2,\ldots.
}
}

\prf{
We use induction on $k$. Suppose first that $k = 0$. Then, by definition,
\bes{
\cE(\tilde{\bm{x}}^{(k)} , \bm{x} , \bm{b} ) = \cE(\bm{0} , \bm{x} , \bm{b} ) \leq \cG(\bm{0},\bm{b}) = \nm{\bm{b}}_{2;\cV} = \varepsilon_0.
}
Now suppose that the result holds for $k$. The previous lemma gives
\eas{
\cE(\tilde{\bm{x}}^{(k+1)} , \bm{x} , \bm{b} ) & = \cE(a_k \cP(\tilde{\bm{x}}^{(k)} / a_k , \bm{b} / a_k , n) ,\x , \bm{b} ) 
\\
& \leq \frac{C^2}{a_k \tau n} \left ( \cE(\tilde{\bm{x}}^{(k)} , \bm{x} , \bm{b}) + \zeta \right )^2 + \frac{a_k}{\sigma n}
\\
& \leq \frac{C^2}{a_k \tau n} \left ( \varepsilon_k + \zeta \right )^2 + \frac{a_k}{\sigma n}.
}
We now substitute the values of $n$ and $a_k$ to obtain
\bes{
\cE(\tilde{\bm{x}}^{(k+1)} , \bm{x} , \bm{b} ) = \frac{2 C^2 (\varepsilon_k + \zeta)}{r \sigma \tau n^2} + \frac12 r (\varepsilon_k + \zeta) \leq \frac12 r(\varepsilon_k + \zeta ) + \frac12 r (\varepsilon_k + \zeta) = \varepsilon_{k+1}.
}
This completes the proof.
}

This theorem states that the restarted primal-dual iteration $\tilde{\bm{x}}^{(0)},\tilde{\bm{x}}^{(1)},\tilde{\bm{x}}^{(2)},\cdots$ yields an objective function error $\cE(\tilde{\bm{x}}^{(k)} , \bm{x} , \bm{b})$ that converges exponentially fast in the number of restarts $k$. Further, each (inner) primal-dual iteration involves a number of steps $n$ that depends on the parameters $C$, $r$, $\sigma$ and $\tau$. In other words, $n$ is a constant independent of $k$. Hence, the restarted scheme converges exponentially fast in the total number of primal-dual iterations as well. 

As discussed in \S \ref{ss:hyperparam-values}, it is typical to use this theorem to optimize the choice of $r$. Recall that this leads to the explicit choice $r = \E^{-1}$. We use this value in our algorithms -- see Table \ref{tab:the-effic-algorithms}.

\section{Final arguments}\label{s:final-args}

We are now ready to prove the main results, Theorems \ref{t:main-res-map-alg}--\ref{t:main-res-effic-algo-alg_exp}.  In several of these proofs, we require the following definition. Let $s \in \bbN$ and define
\begin{equation}\label{def_k(s)}
k(s) := \max \{ |S|_{\u} : S \subset \bbN_0^d, |S| \leq s, \, S  \, \mathrm{lower} \},
\end{equation}
where $\bm{u}$ are the intrinsic weights \R{weights_def} (recall the definition of a lower set from Definition \ref{def:lower-anchored-set}). It can be shown that
\be{
\label{k(s)-bound}
k(s) = s^2,\quad \mbox{(Legendre)},
\qquad
k(s) \leq \min \{ 2^d s , s^{\log(3)/\log(2)} \},\quad \mbox{(Chebyshev)}.
}
See, e.g., \cite[Eqn.\ (7.42) and Props.\ 5.13 \& 5.17]{adcock2021sparse}. We will use this property several times in what follows.

\subsection{Algebraic rates of convergence, finite dimensions}

\prf{
[Proof of Theorem \ref{t:main-res-map-alg}]

The mapping was described in Table \ref{tab:the-mappings}. As shown therein, we can write the corresponding approximation as $\hat{f} = \sum_{\bm{\nu} \in \Lambda} \hat{c}_{\bm{\nu}} \Psi_{\bm{\nu}}$, where $\hat{\bm{c}} = (\hat{c}_{\bm{\nu}})_{\bm{\nu} \in \Lambda}$ is a minimizer of \R{wsr-LASSO}. Next, due to the various assumptions made, we may apply Theorem \ref{t:main-res-map-alg_inex}. Setting $\tilde{f} = \hat{f}$ and $\tilde{\bm{c}} = \hat{\bm{c}}$, we deduce that 
\be{
\label{f-fhat-error-final-arg}
\nmu{f - \hat{f}}_{L^2_{\varrho}(\cU ; \cV)} \leq c_1 \cdot \xi,\quad \nmu{f - \hat{f}}_{L^{\infty}(\cU ; \cV)} \leq c_2 \cdot \sqrt{k} \cdot \xi,
}
where (after writing out the term $E_{h,\infty}(f)$ explicitly)
\be{
\label{xi-final-arg-exact-alg-fin}
\xi = \dfrac{\sigma_{k}(\bm{c}_{\Lambda})_{1,\u;\cV}}{\sqrt{k}}   +   \frac{E_{\Lambda,\infty}(f)}{\sqrt{k}}+ E_{\Lambda,2}(f) + \nmu{f - \cP_h(f)}_{L^{\infty}(\cU ; \cV)} + \cG(\hat{\bm{c}}) - \cG(\cP_h(\bm{c}_{\Lambda}))   + \frac{\nm{\bm{n}}_{2;\cV}}{\sqrt{m}} ,
}
and $k = m / (c_0 L)$ with $c_0 \geq 1$ a universal constant. We now bound each term separately.

\pbk
\textit{Step 1.\ The terms $\sigma_{k}(\bm{c}_{\Lambda})_{1,\u;\cV}/\sqrt{k}$, $E_{\Lambda,\infty}(f) / \sqrt{k}$ and $E_{\Lambda,2}(f)$.} 
The term $\sigma_{k}(\bm{c}_{\Lambda})_{1,\u;\cV}/\sqrt{k}$ is estimated via (ii) of Theorem \ref{thm:best_s-term_fin-dim} with $q = 1$. This gives
\be{
\label{xi-final-arg-sigma-term}
\frac{\sigma_{k}(\bm{c}_{\Lambda})_{1,\u;\cV}}{\sqrt{k}} \leq C(d,p,\bm{\rho}) \cdot k^{1/2-1/p} = C(d,p,\bm{\rho}) \cdot \left ( \frac{m}{c_0 L} \right )^{1/2-1/p}.
}
We estimate the term $E_{\Lambda,2}(f)$ by first recalling that $\Lambda = \Lambda^{\mathsf{HC}}_{n,d}$ is the union of all lower sets (see Definition \ref{def:lower-anchored-set}) of size at most $n = \lceil m / L \rceil $ (see \S \ref{ss:prob-stat}). Hence, using (i) of Theorem \ref{thm:best_s-term_fin-dim} with $s = n$ and $q = 2$, we get
\be{
\label{xi-final-arg-E2-term}
E_{\Lambda,2}(f) = \nmu{\bm{c} - \bm{c}_{\Lambda}}_{2;\cV} \leq  \nmu{\bm{c} - \bm{c}_{S}}_{2;\cV} \leq C(d,p,\bm{\rho}) \cdot n^{1/2-1/p} \leq C(d,p,\bm{\rho}) \cdot \left ( \frac{m}{c_0 L} \right )^{1/2-1/p}.
}
Here, in the last step we recall that $n = \lceil m / L \rceil $ and $c_0 \geq 1$. 

It remains to consider $E_{\Lambda,\infty}(f) / \sqrt{k}$.  Due to the choice of weights, we have $E_{\Lambda,\infty}(f) \leq \nmu{\bm{c} - \bm{c}_{\Lambda}}_{1,\bm{u};\cV}$.
We now apply (i) of Theorem \ref{thm:best_s-term_fin-dim} once more, with $s = n$ and $q = 1$, to get
\bes{
E_{\Lambda,\infty}(f) \leq \nmu{\bm{c} - \bm{c}_{S}}_{1,\bm{u};\cV} \leq C(d,p,\bm{\rho}) \cdot n^{1-1/p}.
}
Since $n = \lceil m/L \rceil \geq m/(c_0 L) = k$, we obtain
\be{
\label{xi-final-arg-Einf-term}
\frac{E_{\Lambda,\infty}(f)}{\sqrt{k}} \leq C(d,p,\bm{\rho}) \cdot \left ( \frac{m}{c_0 L} \right )^{1/2-1/p}.
}

\pbk
\textit{Step 2.\ The term $\cG(\hat{\bm{c}}) - \cG(\cP_h(\bm{c}_{\Lambda}))$.} Since $\hat{\bm{c}}$ is a minimizer of \R{wsr-LASSO} and $\cP_h(\bm{c}_{\Lambda}) \in \cV^N_h$ is feasible for \R{wsr-LASSO}, this term satisfies
\be{
\label{xi-final-arg-obj-term}
\cG(\hat{\bm{c}}) - \cG(\cP_h(\bm{c}_{\Lambda})) \leq 0.
}

\pbk
\textit{Step 3.\ Conclusion.} We now substitute the bounds \R{xi-final-arg-sigma-term}--\R{xi-final-arg-obj-term} into \R{xi-final-arg-exact-alg-fin}. Since $k \leq m/L$, we deduce that $\xi \leq \zeta$, where $\zeta$ is given by \R{zeta-alg-def}. This completes the proof.
}

\prf{
[Proof of Theorem \ref{t:main-res-algo-alg}]

The argument is similar to that of the previous theorem. Recall from \S \ref{ss:the-algorithms} that, in this case the approximation $\hat{f} = \sum_{\bm{\nu} \in \Lambda} \tilde{c}_{\bm{\nu}} \Psi_{\bm{\nu}}$, where $\hat{\bm{c}} = \bar{\bm{c}}^{(T)}$ is the ergodic sequence obtained after $T$ steps of the primal-dual iteration applied to \R{wsr-LASSO}. Hence, the only difference is the estimation of $\cG(\hat{\bm{c}}) - \cG(\cP_h(\bm{c}_{\Lambda}))$ in Step 2.

We now do this using Lemma \ref{lemma-subopti}. In order to apply this lemma we first need to estimate $\nmu{\bm{A}}_{\cB(\cV^N_h,\cV^m_h)}$. Let $\bm{x} = (x_{\bm{\nu}})_{\bm{\nu} \in \Lambda} \in \cV^N_h$ and define $p(\y) = \sum_{\bm{\nu} \in \Lambda} x_{\bm{\nu}} \Psi_{\bm{\nu}}$. Then
\bes{
\nmu{\bm{A} \bm{x}}_{2;\cV} = \sqrt{\frac1m \sum^{m}_{i=1} \nm{p(\y_i)}^2_{\cV}} \leq \sup_{\bm{\y} \in \cU} \nm{p(y)}_{\cV} \leq \sum_{\bm{\nu} \in \Lambda} \nm{x_{\bm{\nu}}}_{\cV} u_{\bm{\nu}} \leq \nm{\bm{x}}_{2;\cV} \sqrt{|\Lambda|_{\bm{u}}}.
}
Now the set $\Lambda$ is lower and of cardinality $|\Lambda| = \Theta(n,d)$. Hence, by \R{k(s)-bound} with $s = N$, we have $|\Lambda|_{\bm{u}} \leq (\Theta(n,d))^{2\alpha}$, where $\alpha$ is as in \R{main_alpha_def}. Since $\bm{x}$ was arbitrary, we get
\begin{equation}\label{eq:Bound_A}
\nm{\bm{A}}_{2;\cV} \leq (\Theta(n,d))^{\alpha}.
\end{equation}
Since the primal-dual iteration in \S \ref{ss:the-algorithms} is used with $\tau = \sigma = (\Theta(n,d))^{-\alpha}$, we have that $\nm{\bm{A}}^2_{2;\cV} \leq (\tau \sigma)^{-1}$. Hence we may apply Lemma \ref{lemma-subopti}. Since the iteration is also initialized with the zero vector and run for a total of $T = \lceil 2 (\Theta(n,d))^{\alpha} t \rceil$ iterations (see \S \ref{ss:the-algorithms} once more), this gives
\bes{
\cG(\hat{\bm{c}}) - \cG(\cP_h(\bm{c}_{\Lambda})) \leq (\Theta(n,d))^{\alpha} \frac{\nm{\cP_{h}(\bm{c}_{\Lambda} )}^2_{2;\cV} + 1}{T}.
}
Observe that
\bes{
\nmu{\cP_h(\bm{c}_{\Lambda})}_{2;\cV} \leq \nmu{\bm{c}_{\Lambda}}_{2;\cV} \leq \nmu{\bm{c}}_{c;\cV} = \nmu{f}_{L^2_{\varrho}(\cU ; \cV)} \leq 1.
}
Here, in the last step, we use the fact that $f \in \cB(\bm{\rho})$, and therefore $\nmu{f}_{L^2_{\varrho}(\cU ; \cV)} \leq \nmu{f}_{L^{\infty}(\cU;\cV)} \leq 1$. Using this and the value of $T$, we deduce that 
\bes{
\cG(\hat{\bm{c}}) - \cG(\cP_h(\bm{c}_{\Lambda})) \leq \frac{1}{t}.
}
Substituting this into \R{xi-final-arg-exact-alg-fin} and combining with the other estimates \R{xi-final-arg-sigma-term}--\R{xi-final-arg-Einf-term} derived in Step 2 of the proof of Theorem \ref{t:main-res-map-alg} now gives the desired error bound.

It remains to estimate the computational cost. We do this via Lemmas \ref{l:comp-cost-PDI} and \ref{l:comp-cost-A}. First observe that the value $k$ in Lemma \ref{l:comp-cost-A} is equal to $k = d$ in this case, since the index set $\Lambda = \Lambda^{\mathsf{HC}}_{n,d}$ is a $d$-dimensional hyperbolic cross index set. Similarly, the value $n$ in Lemma \ref{l:comp-cost-A} is bounded by the order $n$ of this hyperbolic cross. As $\Lambda$ is a lower set, we also have $n \leq N$. Hence, the computational cost for forming the matrix $\bm{A}$ is bounded by $c \cdot m \cdot N \cdot d$. We now use Lemma \ref{l:comp-cost-PDI} to bound the computational cost of the algorithm. Finally, we recall that $N = \Theta(n,d)$ and $T = \lceil 2 (\Theta(n,d))^{\alpha}t \rceil$ in this case. 
}

\begin{proof}
[Proof of Theorem \ref{t:main-res-effic-algo-alg}]

As in the previous proof, we only need to estimate the term $\cG(\hat{\bm{c}}) - \cG(\cP_h(\bm{c}_{\Lambda}))$. Recall from Table \ref{tab:the-effic-algorithms} that in this case $\hat{\bm{c}} = \tilde{\bm{c}}^{(R)}$ is the output of the restarted primal-dual iteration with $R$ restarts. Our goal is to use Theorem \ref{thm:restart-scheme-bound} applied to the problem \R{wsr-LASSO} with weights $\bm{w} = \bm{u}$ as in \R{weights_def}, $\lambda = (4 \sqrt{m/L})^{-1}$ and $\bm{x} = \cP_h(\bm{c}_{\Lambda})$.

We first show that the conditions of this theorem hold. Recall from Step 2 of the proof of Theorem \ref{t:main-res-map-alg_inex} that the matrix $\bm{A}$ has the weighted rNSP of order $(k,\bm{u})$ over $\cV_h$ with constants $\rho = 2 \sqrt{2} / 3$ and $\gamma = 2 \sqrt{5} / 3$. In particular,
\bes{
\frac{(1+\rho)^2}{(3 + \rho) \gamma} \geq 0.64.
}
We now use \R{lambda-bound} to see that
\bes{
\lambda = \frac{1}{4 \sqrt{c_0}} \frac{1}{\sqrt{k}} \leq \frac{(1+\rho)^2}{(3 + \rho) \gamma} \frac{1}{\sqrt{k}},
} 
for a sufficiently large choice of $c_0$.

Next, with this choice of $\bm{x}$, we see that
\bes{
\xi(\bm{x} , \bm{b}) = \frac{\sigma_k(\cP_h(\bm{c}_{\Lambda}))_{1,\bm{u};\cV}}{\sqrt{k}} + \nmu{\bm{A} \cP_h(\bm{c}_{\Lambda}) - \bm{b}}_{2;\cV}. 
}
Using \R{sigma-k-Ph} and \R{bound-needed-2}, we get
\bes{
\xi(\bm{x} , \bm{b}) \leq  \frac{\sigma_k(\bm{c}_{\Lambda})_{1,\bm{w};\cV} }{\sqrt{k}} + \sqrt{2} \left ( \frac{E_{\Lambda,\infty}(f)}{\sqrt{k}}  + E_{\Lambda,2}(f) \right )+ E_{h,\infty}(f) + \frac{\nm{\bm{n}}_{2;\cV}}{\sqrt{m}},
}
with probability at least $1-\epsilon$. Using \R{xi-final-arg-sigma-term}--\R{xi-final-arg-Einf-term}, we deduce that
\bes{
\xi(\bm{x} , \bm{b}) \leq \zeta,
}
with probability at least $1-\epsilon$, where $\zeta$ is as in \R{zeta-alg-def}. Hence, $\xi(\bm{x},\bm{b}) \leq \zeta'$.

Next, recall from Table \ref{tab:the-effic-algorithms} that $\tau = \sigma = (\Theta(n,d))^{-\alpha}$ in this case. Due to \R{eq:Bound_A}, we see that $\nm{\bm{A}}_{2;\cV} \leq (\tau \sigma)^{-1}$ as well.

Now consider the constant $C$ defined in \R{C-def-restart-scheme}. The values for $\rho$ and $\gamma$ give that $C'_1 \leq C'_2 \leq 103$. Since $\lambda = c / \sqrt{k}$ with $c = 1/(4 \sqrt{c_0})$, we see that
\be{
\label{cstar-def}
4C \leq 812 / c = 3296 \sqrt{c_0} : = c^{\star}.
}
Therefore, recalling that $r = 1/2$ and $\tau = \sigma = (\Theta(n,d))^{-\alpha}$, we see that
\bes{
\left \lceil \frac{2 C}{r \sqrt{\sigma \tau}} \right \rceil = \left \lceil (\Theta(n,d))^{\alpha} c^{\star} \right \rceil = T,
}
where $T$ is as specified in Table \ref{tab:the-effic-algorithms}, and
\bes{
\frac12 r \sigma (\varepsilon_k + \zeta') T = \frac{(\Theta(n,d))^{\alpha} T}{4} \varepsilon_{k+1} = s \varepsilon_{k+1} = a_k,
}
where $s$ and $a_k$ are as specified in Table \ref{tab:the-effic-algorithms} and Algorithm \ref{a:primal-dual-wSRLASSO-restart}, respectively.

With this in hand, we are now finally in a position to apply Theorem \ref{thm:restart-scheme-bound}. We deduce that
\bes{
\cG(\hat{\bm{c}}) - \cG(\cP_h(\bm{c}_{\Lambda})) = \cE(\tilde{\bm{c}}^{(R)},\cP_h(\bm{c}_{\Lambda}),\bm{b}) \leq \varepsilon_k = \E^{-R} \nm{\bm{b}}_{2;\cV} + \zeta'.
}
To complete the proof of the error bound \R{err-bound-effic-algo-alg_2}, we simply note that $\nm{\bm{b}}_{2;\cV} \leq \nm{f}_{L^{\infty}(\cU ; \cV)} \leq 1$, since $f \in \cB(\bm{\rho})$.

It remains to estimate the computational cost. As before, the computational cost for forming the matrix $\bm{A}$ is bounded by $c \cdot m \cdot N \cdot d$. Next, by construction, we observe that the algorithm consists of $R = t$ primal-dual iterations, each involving $T = \lceil (\Theta(n,d))^{\alpha} c^{\star} \rceil$ steps. Therefore, by Lemma \ref{l:comp-cost-PDI} the computational cost for the algorithm is
\bes{
c \cdot \left ( m \cdot N \cdot K + (m+N) \cdot (F(\bm{G})+K) \right ) \cdot \lceil (\Theta(n,d))^{\alpha} c^{\star} \rceil \cdot t.
}
Since $N = \Theta(n,d)$ and $c^{\star}$ is a universal constant, the result follows.
\end{proof}

\subsection{Algebraic rates of convergence, infinite dimensions}

\prf{[Proof of Theorem \ref{t:main-res-map-alg_infty}]

The proof is similar to that of Theorem \ref{t:main-res-map-alg},  except that it uses Theorem \ref{t:main-res-map-alg_inex_infty} in place of Theorem \ref{t:main-res-map-alg_inex}. In particular, we see that \R{f-fhat-error-final-arg} also holds in this case with $\xi$ as in \R{xi-final-arg-exact-alg-fin} and $k = m/(c_0 L)$.

Step 2 is identical. The only differences occur in Step 1. We now describe the changes needed in this case. First consider the term $\sigma_{k}(\bm{c}_{\Lambda})_{1,\u;\cV}/\sqrt{k}$. To bound this, we use (ii) of Theorem \ref{thm:best_s-term_inf-dim} with $q = 1 > p$. This gives
\bes{
\frac{\sigma_{k}(\bm{c}_{\Lambda})_{1,\u;\cV}}{\sqrt{k}} \leq C (\bm{b},\varepsilon,p)  \cdot k^{1/2-1/p} = C (\bm{b},\varepsilon,p)  \cdot  \left ( \frac{m}{c_0 L} \right )^{1/2-1/p}.
}
To estimate $E_{\Lambda,2}(f)$, recall that $\Lambda = \Lambda^{\mathsf{HCI}}_{n}$ contains all anchored sets (see Definition \ref{def:lower-anchored-set}) of size at most $n = \lceil m/L \rceil$ ((see \S \ref{ss:prob-stat}). Hence, using (iii) of Theorem \ref{thm:best_s-term_inf-dim} with $s = n$ and $q = 2 > p$, we get
\bes{
E_{\Lambda,2}(f) = \nmu{\bm{c} - \bm{c}_{\Lambda}}_{2;\cV} \leq  \nmu{\bm{c} - \bm{c}_{S}}_{2;\cV} \leq C (\bm{b},\varepsilon,p) \cdot n^{1/2-1/p} \leq C (\bm{b},\varepsilon,p) \cdot \left ( \frac{m}{c_0 L} \right )^{1/2-1/p}.
}
Finally, for $E_{\Lambda,\infty}(f)$, we use (iii) of Theorem \ref{thm:best_s-term_inf-dim} once more (with $q = 1 > p$) to get
\bes{
\frac{E_{\Lambda,\infty}(f)}{\sqrt{k}} \leq \frac{\nmu{\bm{c}-\bm{c}_S}_{1,\bm{u};\cV}}{\sqrt{k}} \leq C (\bm{b},\varepsilon,p)  \cdot k^{1/2-1/p} = C (\bm{b},\varepsilon,p)  \cdot  \left ( \frac{m}{c_0 L} \right )^{1/2-1/p}.
}
Having done this, we also observe that $\cG(\hat{\bm{c}}) - \cG(\cP_h(\bm{c}_{\Lambda})) \leq 0$ in this case, since $\hat{\bm{c}}$ is once more an exact minimizer. Using this and the previously-derived bounds, we conclude that $\xi \leq \zeta$, where $\zeta$ is as in \R{zeta-alg-inf-def}. This gives the result.
}

\prf{[Proof of Theorem \ref{t:main-res-algo-alg_infty}]

The argument is similar to that of Theorem \ref{t:main-res-algo-alg}. Here $\hat{\bm{c}} = \bar{\bm{c}}^{(T)}$ is the ergodic sequence obtained after $T$ steps of the primal-dual iteration applied to \R{wsr-LASSO} as well.  

We recall that the set $\Lambda$ is lower and of cardinality $|\Lambda| = \Theta(n,d)$ with $d=\infty$. Hence, by \R{k(s)-bound} with $s = N$, we have $|\Lambda|_{\bm{u}} \leq (\Theta(n,d))^{2\alpha}$, where $\alpha$ is as in \R{main_alpha_def}. Using this, we get
\bes{
\nm{\bm{A}}_{2;\cV} \leq (\Theta(n,d))^{\alpha},
}
as before.
Since the primal-dual iteration in Table \ref{tab:the-effic-algorithms} is used with $\tau = \sigma = (\Theta(n,d))^{-\alpha}$, we have that $\nm{\bm{A}}^2_{2;\cV} \leq (\tau \sigma)^{-1}$. Hence, following the same steps we deduce that 
\bes{
\cG(\hat{\bm{c}}) - \cG(\cP_h(\bm{c}_{\Lambda})) \leq \frac{1}{t}.
}
Substituting this into \R{xi-final-arg-exact-alg-fin} and combining with the other estimates \R{xi-final-arg-sigma-term}--\R{xi-final-arg-Einf-term} derived in Step 2 of the proof of Theorem \ref{t:main-res-map-alg} now gives the desired error bound.

The computational cost estimate is similar to the that in the proof of Theorem \ref{t:main-res-algo-alg}. In this case, observe that the value $k$ in Lemma \ref{l:comp-cost-A} is equal to $n$. Hence the computational cost of forming $\bm{A}$ is bounded by $c \cdot m \cdot N \cdot n$ in this case. The computational cost for the algorithm is given by Lemma \ref{l:comp-cost-PDI}. To complete the estimate, we substitute the values $N = \Theta(n,d)$ and $T = \lceil 2 (\Theta(n,d))^{\alpha}t \rceil$, as before.
}

\prf{[Proof of Theorem \ref{t:main-res-effic-algo-alg_infty}]

The proof is similar to that of Theorem \ref{t:main-res-effic-algo-alg} and involves estimating the term $\cG(\hat{\bm{c}}) - \cG(\cP_h(\bm{c}_{\Lambda}))$. Using the same steps, we deduce that
\bes{
\xi(\bm{x} , \bm{b}) \leq \zeta,
}
with probability at least $1-\epsilon/2$, where $\zeta$ is as in \R{zeta-alg-inf-def}. Hence, $\xi(\bm{x},\bm{b}) \leq \zeta'$.

Next, recall from Table \ref{tab:the-effic-algorithms} that $\tau = \sigma = (\Theta(n,d))^{-\alpha}$ with $d= \infty$ in this case. Due to \R{eq:Bound_A}, we see that $\nm{\bm{A}}_{2;\cV} \leq (\tau \sigma)^{-1}$ holds. We now apply  Theorem \ref{thm:restart-scheme-bound} to obtain
\bes{
\cG(\hat{\bm{c}}) - \cG(\cP_h(\bm{c}_{\Lambda})) = \cE(\tilde{\bm{c}}^{(R)},\cP_h(\bm{c}_{\Lambda}),\bm{b}) \leq \varepsilon_R = \E^{-R} \nm{\bm{b}}_{2;\cV} + \zeta'.
}
To complete the proof of the error bound \R{err-bound-effic-algo-alg_2}, we simply note that $\nm{\bm{b}}_{2;\cV} \leq \nm{f}_{L^{\infty}(\cU ; \cV)} \leq 1$, since $f \in \cB(\bm{b},\varepsilon)$.

The computational cost estimate is as in the previous proof.
}

\subsection{Exponential rates of convergence, finite dimensions}

\begin{proof}
[Proof of Theorem \ref{t:main-res-map-alg_exp}]

The proof has the same structure to that of Theorem \ref{t:main-res-map-alg}, the only differences being the use of Theorem \ref{t:main-res-map-alg_inex_exponential} instead of Theorem \ref{t:main-res-map-alg_inex} and the estimation of the various terms in Step 1. Suppose first that $m \geq c_0 2^{d+2} L$ and define the following:
\bes{
s = \begin{cases} \lceil \sqrt{m/(4 c_0 L)} \rceil & \mbox{Legendre,} \\ \lceil m / (4 c_0 2^d L) \rceil & \mbox{Chebyshev.} \end{cases}
}
Observe that
\bes{
s \leq  \begin{cases} \sqrt{m/(c_0 L)}& \mbox{Legendre,} \\ m / (c_0 2^d L) & \mbox{Chebyshev,} \end{cases}
}
and therefore the quantity $k(s)$ defined in \R{def_k(s)} satisfies
\bes{
k(s) \leq \frac{m}{c_0 L} = k.
}
Now consider the term $\sigma_{k}(\bm{c}_{\Lambda})_{1,\u;\cV}/\sqrt{k}$. Using this and (iii) of Theorem \ref{thm:best_s-term_fin-dim} with $p = 1$ we have
\bes{
\frac{\sigma_{k}(\bm{c}_{\Lambda})_{1,\u;\cV}}{\sqrt{k}} \leq \frac{\sigma_{k(s)}(\bm{c})_{1,\u;\cV}}{\sqrt{k}} \leq \frac{C(d,\gamma,\bm{\rho}) \cdot \exp(-\gamma s^{1/d}) }{\sqrt{k}}.
}
Note that this is possible since any lower set $S$ of size at most $s$ satisfies $|S|_{\bm{u}} \leq k(s)$ by definition.

Now consider $E_{\Lambda,\infty}(f)$. Recall that $\Lambda = \Lambda^{\mathsf{HC}}_{n,d}$, where $n$ is as in \R{n-exp-case}. Clearly $n \geq s$, since $c_0 \geq 1$. Hence $\Lambda$ contains all lower sets of size at most $s$. We deduce that
\bes{
E_{\Lambda,\infty}(f) \leq \nmu{\bm{c} - \bm{c}_S}_{1,\bm{u};\cV},
}
for any lower set of size $s$. We now use (iii) of Theorem \ref{thm:best_s-term_fin-dim} with $p = 1$ once more, to get
\bes{
E_{\Lambda,\infty}(f) \leq C(d,\gamma,\bm{\rho}) \cdot \exp(-\gamma s^{1/d}) .
}
We now combine this with the previous bound to deduce that the quantity $\xi$ in Theorem \ref{t:main-res-map-alg_inex_exponential} satisfies
\bes{
\xi \leq C(d,\gamma,\bm{\rho}) \cdot \exp(-\gamma s^{1/d}) + E_{h,\infty}(f) + \frac{\nm{\bm{n}}_{2;\cV}}{\sqrt{m}},
}
(here, we also recall that the term $\cG(\hat{\bm{c}}) - \cG(\cP_h(\bm{c}_{\Lambda})) \leq 0$, as in the proof of Theorem \ref{t:main-res-map-alg}). Using the value of $s$ and recalling that $m \geq c_0 2^{d+2} L$, we deduce that
\bes{
\xi \leq C(d,\gamma,\bm{\rho}) \cdot \begin{cases}  \exp \left ( - \frac{\gamma}{2} \left ( \frac{m}{4 c_0 L} \right )^{\frac{1}{d}} \right ) & Chebyshev \\ \exp \left ( - \gamma \left ( \frac{m}{4 c_0 L} \right )^{\frac{1}{2d}} \right ) & Legendre \end{cases} + \frac{\nm{\bm{n}}_{2;\cV}}{\sqrt{m}} + \nmu{f - \cP_h(f)}_{L^{\infty}(\cU ; \cV)},\quad m \geq c_0 2^{d+2} L.
}
However, this bound also clearly holds for all $m \geq 1$, up to a change in the constant $C(d,\gamma,\bm{\rho})$. After relabelling the universal constant $4 c_0$ as $c_0$, we deduce that $\xi \leq \zeta$, where $\zeta$ is as in \R{zeta-alg-def_exp}. This concludes the proof.
\end{proof}

\prf{[Proof of Theorem \ref{t:main-res-algo-alg_exp}]
The argument is  the same as the proof of Theorem \ref{t:main-res-algo-alg}. The difference relies on the fact that now $\zeta$  has the following bound
\bes{
\xi \leq C(d,\gamma,\bm{\rho}) \cdot \begin{cases}  \exp \left ( - \frac{\gamma}{2} \left ( \frac{m}{4 c_0 L} \right )^{\frac{1}{d}} \right ) & Chebyshev \\ \exp \left ( - \gamma \left ( \frac{m}{4 c_0 L} \right )^{\frac{1}{2d}} \right ) & Legendre \end{cases} + \frac{\nm{\bm{n}}_{2;\cV}}{\sqrt{m}} + \nmu{f - \cP_h(f)}_{L^{\infty}(\cU ; \cV)} + \cG(\hat{\bm{c}}) - \cG(\cP_h(\bm{c}_{\Lambda})).
}
To estimate the final term, we argue exactly as in the proof of Theorem \ref{t:main-res-algo-alg}. The computational cost estimate is likewise identical.
}

\prf{[Proof of Theorem \ref{t:main-res-effic-algo-alg_exp}]

The proof is similar to that of Theorem \ref{t:main-res-effic-algo-alg}, except we use Theorem \ref{t:main-res-map-alg_inex_exponential} instead. Recall from Step 2 of the proof of Theorem \ref{t:main-res-map-alg_inex_exponential} that the matrix $\bm{A}$ has the weighted rNSP of order $(k,\bm{u})$ over $\cV_h$ with constants $\rho = 2 \sqrt{2} / 3$ and $\gamma = 2 \sqrt{5} / 3$ with probability $1-\epsilon$. In particular,
\bes{
\frac{(1+\rho)^2}{(3 + \rho) \gamma} \geq 0.64.
}
We now use \R{lambda-bound} to see that
\bes{
\lambda = \frac{1}{4 \sqrt{c_0}} \frac{1}{\sqrt{k}} \leq \frac{(1+\rho)^2}{(3 + \rho) \gamma} \frac{1}{\sqrt{k}},
} 
for a sufficiently large choice of $c_0$, as before.

Next, with the choice $\bm{x} = \cP_{h}(\bm{c}_{\Lambda})$ as before, we see that
\bes{
\xi(\bm{x} , \bm{b}) = \frac{\sigma_k(\cP_h(\bm{c}_{\Lambda}))_{1,\bm{u};\cV}}{\sqrt{k}} + \nmu{\bm{A} \cP_h(\bm{c}_{\Lambda}) - \bm{b}}_{2;\cV}. 
}
Using  \eqref{bound-needed-2-exp}, we get
\bes{
\xi(\bm{x} , \bm{b}) \leq  \frac{\sigma_k(\bm{c}_{\Lambda})_{1,\bm{w};\cV} }{\sqrt{k}} +E_{\Lambda,\infty}(f) + E_{h,\infty}(f) + \frac{\nm{\bm{n}}_{2;\cV}}{\sqrt{m}},
}
with probability $1- \epsilon$. It now follows from the proof of Theorem \ref{t:main-res-map-alg_exp} that
\bes{
\xi(\bm{x} , \bm{b}) \leq \zeta,
}
with probability at least $1-\epsilon$, where $\zeta$ is as in \R{zeta-alg-def_exp}. Hence, $\xi(\bm{x},\bm{b}) \leq \zeta'$.

The rest of the proof follows the same steps as the proof of Theorem \ref{t:main-res-effic-algo-alg}.
}

\section{Conclusions}\label{s:conclusions}

Sparse polynomial approximation is a useful tool in parametric model problems, including surrogate model construction in UQ. The theory of best $s$-term approximation supports the use of polynomial-based methods, and techniques such as least squares and compressed sensing are known to have desirable sample complexity bounds for obtaining polynomial approximations. In this work, we have closed a key gap between these two areas of research, by showing the existence of algorithms that achieve the algebraic and exponential rates of the best $s$-term approximation with respect to the number of samples $m$. Thus, sparse polynomial approximation can be practically realized in a provably sample-efficient manner. As our numerical experiments confirm, our algorithms are practical, and actually perform better than the theory suggests.

There are a number of avenues for further research. First, this work has focused on Chebyshev and Legendre polynomials on the hypercube $[-1,1]^d$. It is plausible that it can be extended to general ultraspherical or Jacobi polynomials. A more significant challenge involves Hermite or Laguerre polynomials on $\bbR^d$ or $[0,\infty)^d$. This is an interesting problem for future research.

It is notable that the algorithms developed in this paper do not generally compute $m$-term polynomial approximations. Indeed, (inexact) minimizers of the SR-LASSO problem will generally be nonsparse vectors of length $N = \Theta(n,d)$. It is interesting to investigate whether one can develop algorithms that achieve the same error bounds while computing $m$-term polynomial approximations. In classical compressed sensing, one can typically computes sparse solutions by using a greedy or iterative procedure (see, e.g., \cite{foucart2013mathematical}). Unfortunately, it is not clear how to extend these procedures to the weighted case with theoretical guarantees. Nonetheless, certain weighted greedy methods appear to work well in practice for sparse polynomial approximation \cite{adcock2020sparse}.

Another motivation for considering different algorithms is to see if the computational cost estimates can be reduced. While this is often not the main computational bottleneck in parametric model problems (generally, computing the samples is the most computationally-intensive step), it is still an important issue.
We have shown that the computational cost is at worst subexponential in $m$ in infinite dimensions, and algebraic in $m$ (for fixed $d$) in finite dimensions. Whether these are optimal is an interesting open problem. Here, ideas from sublinear-time algorithms \cite{choi2021sparse,choi2021sparse2} may be particularly useful.

In the case of the exponential rates, it is notable that the best $s$-term approximation error is exponentially small in $\gamma \cdot s^{1/d}$ (see Theorem \ref{t:best_s_term_exp_1}), whereas the exponents in \S \ref{ss:main-res-exp} are $\gamma / 2 (m / (c_0 L))^{1/d}$ (Chebyshev) and $\gamma (m/(c_0 L))^{1/(2d)}$ (Legendre case). The reason for this can be traced to the sample complexity estimate for computing a sparse (and lower) polynomial approximation via compressed sensing with Monte Carlo sampling, i.e., $m \approx c_0 \cdot 2^d \cdot s \cdot L$ (Chebyshev) or $m \approx c_0 \cdot s^2 \cdot L$ (Legendre). To see why this is the case, combine Lemma \ref{l:LegMat_RIP} with \R{k(s)-bound}. In the setting of least squares, in which the desired polynomial subspace is known, it is possible to change the sampling measure to obtain sample complexity bounds that are log-linear in $s$ and therefore near optimal. See, e.g., \cite{adcock2020nearoptimal,cohen2017optimal,hampton2015coherence}. More recently, several works \cite{cohen2021optimal,kammerer2019worst,limonova2021sampling,temlyakov2021optimal,bartel2022constructive,dolbeault2023sharp} have also introduced sampling schemes that achieve linear sample complexity in $s$ -- i.e., optimal up to a constant. Unfortunately, it is unknown whether linear or log-linear sample complexity possible in the compressed sensing setting, where the target subspace is unknown. See \cite{adcock2022towards} for further discussion on this issue.

Finally, as previously noted in \S \ref{ss:discussion}, this work focuses on polynomial approximation, and not on fundamental issues pertaining to tractability and the information complexity of the classes of multivariate holomorphic functions considered. For some related work in this direction, see \cite{xu2015weak,huang2007approximation,novak2009approximation} and references therein. A question of particular interest is whether pointwise samples (i.e., \textit{standard information}), and more specifically, i.i.d.\ pointwise samples (i.e., \textit{random information}) constitutes optimal or near-optimal information for these classes of functions. These questions have recently been considered in a broader context in \cite{krieg2023exponential,hinrichs2019power}. See also \cite{krieg2021recovery} for the case of functions in Sobolev spaces. As we observed in \S \ref{ss:discussion}, in a recent work \cite{adcock2023optimala} we derived lower bounds for the (adaptive) $m$-widths for classes of $(\bm{b},\varepsilon)$-holomorphic functions in infinite dimensions. Showing that the algorithms (or small modifications thereof) developed in this work also attain (nearly) matching upper bounds -- and, consequently, that i.i.d.\ pointwise samples constitute (near) optimal information -- is an interesting problem for future work.

\section*{Acknowledgements}

BA acknowledges the support of NSERC through grant RGPIN-2021-611675. SB acknowledges the support of NSERC through grant RGPIN-2020-06766, the Faculty of Arts and Science of Concordia University and the CRM Applied Math Lab. ND acknowledges support from a PIMS postdoctoral fellowship.

\appendix

\section{Best polynomial approximation rates for holomorphic functions}\label{s:best-poly-rates}

In this appendix, we recap a series of standard best approximation error bounds for polynomial approximation of holomorphic functions. These are used in \S \ref{s:final-args} to estimate the various error terms appearing in Theorems \ref{t:main-res-map-alg_inex}--\ref{t:main-res-map-alg_inex_exponential}. 

\subsection{The finite-dimensional case}\label{app:fin-dim}

We first consider the finite-dimensional case, where $\cU = [-1,1]^d$ for $d < \infty$ and $f : \cU \rightarrow \cV$ is a Hilbert-valued function (in fact, the following results also apply in the more general setting of Banach-valued functions; however, we shall not consider this explicitly). We now summarize the various approximation error bounds in the following theorem. This result combines various well-known results in the literature. It is essentially the same as \cite[Thm.\ 3.25]{adcock2021sparse}. However, we have made a number of minor edits to fit the notation and setup of this paper (see Remark \ref{rem:differences-finite} below).

\thm{
[Best $s$-term decay rates; finite dimensions]
\label{thm:best_s-term_fin-dim}
Let $d \in \bbN$, $f \in \cB(\bm{\rho})$ for some $\bm{\rho} > \bm{1}$, where $\cB(\bm{\rho})$ is as in \R{B-def}, and $\bm{c} = (c_{\bm{\nu}})_{\bm{\nu} \in \bbN^d_0}$ be its Chebyshev or Legendre coefficients. Then the following best $s$-term decay rates hold:
\begin{itemize}
\item[(i)] for any $0 < p \leq q \leq 2$ and $s \in \mathbb{N}$, there exists a lower set $S \subset \bbN^d_0$ of size $|S| \leq s$ such that 
$$
\sigma_s(\bm{c})_{q;\mathcal{V}} 
\leq \nmu{\bm{c} - \bm{c}_S}_{q;\cV}
\leq \nmu{\bm{c} - \bm{c}_S}_{q,\bm{u};\cV}
\leq C \cdot s^{1/q-1/p},
$$
where $\sigma_{s}(\bm{c})_{q;\cV}$ is as in Definition \ref{def:best_s_term} (with $\Lambda = \bbN^d_0$), $\bm{u}$ is as in \R{weights_def} and $C = C(d,p,\bm{\rho}) > 0$ depends on $d$, $p$ and $\bm{\rho}$ only;

\item[(ii)] for any $0< p \leq q \leq 2$ and $k >0$, 
$$
\sigma_k(\bm{c})_{q,\bm{u};\mathcal{V}} 
\leq C \cdot k^{1/q-1/p}, 
$$
where $\sigma_k(\bm{c})_{q,\bm{u};\mathcal{V}}$ is as in Definition \ref{def:best_s_term_weighted}, $\bm{u}$ is as in \R{weights_def} and $C = C(d,p,\bm{\rho}) > 0$ depends on $d$, $p$ and $\bm{\rho}$ only;

\item[(iii)] for any $0 < p \leq 2$,
\bes{
0 < \gamma < (d+1)^{-1} \left ( d! \prod^{d}_{j=1} \log(\rho_j) \right )^{1/d},
}
and $s \in \bbN$, there exists a lower set $S \subset \bbN^d_0$ of size $|S| \leq s$ such that 
\bes{
\sigma_{s}(\bm{c})_{p;\mathcal{V}}  \leq \nmu{\bm{c} - \bm{c}_{S}}_{p;\mathcal{V}} \leq \nmu{\bm{c} - \bm{c}_{S}}_{p,\bm{u};\mathcal{V}} \leq C \cdot \exp(-\gamma s^{1/d}),
}
where $\sigma_{s}(\bm{c})_{p;\cV}$ is as in Definition \ref{def:best_s_term} (with $\Lambda = \bbN^d_0$), $\bm{u}$ is as in \R{weights_def} and $C = C(d,\gamma,p,\bm{\rho}) > 0$ depends on $d$, $\gamma$, $p$ and $\bm{\rho}$ only.
\end{itemize}
}

\rem{
\label{rem:differences-finite}
There are several differences between Theorem \ref{thm:best_s-term_fin-dim} and \cite[Thm.\ 3.25]{adcock2021sparse}. A minor difference is that we do not specify the various constants $C$ appearing in the result. Another difference is in the presentation of (iii). Here we allow arbitrary $s \geq 1$ (instead of $s \geq \bar{s}$) at the expense of a larger (and unspecified) constant $C$. The main difference, however, is the additional term $\nmu{\bm{c} - \bm{c}_S}_{q,\bm{u};\cV}$ appearing in (i). This can be shown as follows. First, one defines the sequence $\bar{\bm{c}} = (u^{2/q-1}_{\bm{\nu}} c_{\bm{\nu}} )_{\bm{\nu} \in \bbN^d_0}$ so that $\nmu{\bm{c} - \bm{c}_S}_{q,\bm{u};\cV} = \nm{\bar{\bm{c}} - \bar{\bm{c}}_S }_{q;\cV}$ and then uses Stechkin's inequality in lower sets (see, e.g., \cite[Lem.\ 3.9]{adcock2021sparse}) to show that $\nm{\bar{\bm{c}} - \bar{\bm{c}}_S }_{q;\cV} \leq s^{1/q-1/p} \nm{\bar{\bm{c}}}_{p,M;\cV}$, where $\nm{\cdot}_{p,M;\cV}$ is the norm on the \textit{majorant $\ell^p$ space} $\ell^p_{M}(\bbN^d_0;\cV)$ (see, e.g., \cite[Defn.\ 3.8]{adcock2021sparse}). Finally, it can be shown that $\nm{\bar{\bm{c}}}_{p,M;\cV} \leq C(d,p,\bm{\rho})$ using standard arguments. See, e.g., \cite[Lem.\ 7.19]{adcock2021sparse} (this lemma only considers the scalar-valued case; however the extension to the Hilbert-valued case is straightforward).
}

Note that Theorem \ref{thm:best_s-term_fin-dim} immediately implies Theorems \ref{t:best_s_term_alg} and \ref{t:best_s_term_exp_1}. For the former, we note that $\nmu{f - f_{S_1}}_{L^2_{\varrho}(\cU ; \cV)} = \nm{\bm{c} - \bm{c}_{S_1}}_{2;\cV}$ and $\nmu{f - f_{S_2}}_{L^{\infty}(\cU ; \cV)} \leq \nmu{\bm{c} - \bm{c}_{S_2}}_{1,\bm{u} ; \cV}$. We then apply (i) with $q = 2$ or $q = 1$. For the latter, we use (iii) with $p = 1$.

\subsection{The infinite-dimensional case}\label{app:inf-dim}

We now consider the infinite-dimensional case, where $d = \infty$ and $\cU = [-1,1]^{\bbN}$.

\thm{
[Best $s$-term decay rates; infinite-dimensional case]
\label{thm:best_s-term_inf-dim}

Let $d = \infty$, $0 < p <1$, $\varepsilon >0$, $\bm{b} \in \ell^p(\bbN)$ with $\bm{b} > \bm{0}$ and $f \in \cB(\bm{b},\varepsilon)$, where $\cB(\bm{b},\varepsilon)$ is as in \R{B-b-eps-def}. Let $\bm{c} = (c_{\bm{\nu}})_{\bm{\nu} \in \cF}$ be the Chebyshev or Legendre coefficients of $f$. Then the following best $s$-term decay rates hold:

\begin{itemize}
\item[(i)] For any $p \leq q < \infty$ and $s \in \bbN$, there exists a lower set $S \subset \cF$ of size $|S| \leq s$ such that 
\bes{
\sigma_s(\bm{c})_{q;\mathcal{V}} \leq \nmu{\bm{c} - \bm{c}_S}_{q;\cV} \leq \nmu{\bm{c} - \bm{c}_S}_{q,\bm{u};\cV} \leq   C\cdot s^{1/q-1/p},
}
where $\sigma_{s}(\bm{c})_{q;\cV}$ is as in Definition \ref{def:best_s_term} (with $\Lambda = \cF$), $\bm{u}$ is as in \R{weights_def} and $C = C (\bm{b},\varepsilon,p) >0$ depends on $\bm{b}$, $\varepsilon$ and $p$ only.

\item[(ii)] For any $p \leq q \leq 2$ and $k >0$, 
$$
\sigma_k(\bm{c})_{q,\bm{u};\mathcal{V}} 
\leq C \cdot k^{1/q-1/p}, 
$$
where $\sigma_k(\bm{c})_{q,\bm{u};\mathcal{V}}$ is as in Definition \ref{def:best_s_term_weighted}, $\bm{u}$ is as in \R{weights_def} and $C = (\bm{b},\varepsilon,p) >0$ depends on $\bm{b}$, $\varepsilon$ and $p$ only.

\item[(iii)] Suppose that $\bm{b}$ is monotonically nonincreasing. Then, for any $p \leq q < \infty$ and $s \in \bbN$, there exists an anchored set $S \subset \cF$ of size $|S| \leq s$ such that 
\bes{
\sigma_s(\bm{c})_{q;\mathcal{V}} \leq \nmu{\bm{c} - \bm{c}_S}_{q;\cV} \leq \nmu{\bm{c} - \bm{c}_S}_{q,\bm{u};\cV} \leq   C\cdot s^{1/q-1/p},
}
where $\sigma_{s}(\bm{c})_{q;\cV}$ is as in Definition \ref{def:best_s_term} (with $\Lambda = \cF$), $\bm{u}$ is as in \R{weights_def} and $C = (\bm{b},\varepsilon,p) >0$ depends on $\bm{b}$, $\varepsilon$ and $p$ only.

\end{itemize}

}

\rem{
This theorem is based on \cite[Thms.\ 3.29 and 3.33]{adcock2021sparse}. Besides the term $\nmu{\bm{c} - \bm{c}_S}_{q,\bm{u};\cV}$, parts (i) and (iii) can be found in \cite[Thm.\ 3.29]{adcock2021sparse} and \cite[Thm.\ 3.33]{adcock2021sparse}, respectively. As in the finite-dimensional case (see Remark \ref{rem:differences-finite}), the main difference is the assertion of the bound on $\nmu{\bm{c} - \bm{c}_S}_{q,\bm{u};\cV}$. This can be established through similar arguments, using either the majorant $\ell^p$ space $\ell^p_M(\cF;\cV)$ or the anchored $\ell^p$ space $\ell^p_A(\cF;\cV)$ (see, e.g., \cite[Defn.\ 3.31]{adcock2021sparse}) and then Stechkin's inequality in lower or anchored sets (see, e.g., \cite[Lem.\ 3.32]{adcock2021sparse}). See also \cite[Lem.\ 7.23]{adcock2021sparse} (this lemma only considers the scalar-valued case; however the extension to the Hilbert-valued case is straightforward).

Note that neither \cite[Thm.\ 3.29]{adcock2021sparse} nor \cite[Thm.\ 3.33]{adcock2021sparse} asserts part (ii) of Theorem \ref{thm:best_s-term_inf-dim}. This can be shown via the weighted Stechkin's inequality (see, e.g., \cite[Lem.\ 3.12]{adcock2021sparse}), which gives the bound $\sigma_{k}(\bm{c})_{q,\u;\cV} \leq \nmu{\bm{c}}_{p,\u;\cV} \cdot k^{1/q-1/p}$, and then by showing that $\nmu{\bm{c}}_{p,\u;\cV} \leq C(\bm{b},\varepsilon,p)$. This latter fact can be obtained by the straightforward extension of \cite[Lem.\ 7.23]{adcock2021sparse} to the Hilbert-valued setting.
}

Note that Theorem \ref{thm:best_s-term_inf-dim} implies Theorem \ref{t:best_s_term_alg_inf}. This follows from (i) with $q = 2$ or $q = 1$.

\bibliographystyle{plain}
\bibliography{efficientpolyalgsrefs}

\end{document}